\documentclass[12pt]{amsart}

\title{The orbifold topological vertex}
\author[J. Bryan]{Jim Bryan}
\email{jbryan@math.ubc.ca}
\author[C. Cadman]{Charles Cadman}
\email{mythirty@gmail.com}
\author[B. Young]{Ben Young}
\email{byoung@math.mcgill.ca,benyoung@msri.org}

\address{
Department of Mathematics\\
University of British Columbia \\
Room 121, 1984 Mathematics Road  \\
Vancouver, B.C., Canada V6T 1Z2
}

\usepackage{epic}
\usepackage{color}

\usepackage{amsmath,amsthm,amsfonts}
\usepackage{hyperref}
\usepackage{times}

\usepackage{amsmath,amsthm,amsfonts,amscd}
\usepackage{times}
\usepackage{verbatim}
\usepackage{afterpage}
\usepackage{amsmath,amsfonts,amssymb}
\usepackage{bbm,dsfont}
\usepackage{mathrsfs}
\usepackage{ifthen}
\usepackage{rotating}

\newtheorem{thm}{Theorem}
\newtheorem{theorem}{Theorem}

\newtheorem{prop}[thm]{Proposition}

\newtheorem{conj}[thm]{Conjecture}

\newtheorem{lemma}[theorem]{Lemma}
\newtheorem{corollary}[theorem]{Corollary}
\theoremstyle{definition}

\newtheorem{def-thm}[thm]{Definition-Theorem}
\newtheorem{remark}[theorem]{Remark}
\newtheorem{defn}[theorem]{Definition}
\newtheorem{exmpl}[theorem]{Example}

\newcommand{\cnums} {{\mathbb C}}          
\newcommand{\nnums} {{\mathbb N}}		
\newcommand{\znums} {{\mathbb Z}}		
\newcommand{\qnums} {{\mathbb Q}}		

\newcommand{\Hom}{\operatorname{Hom}}

\newcommand{\Hilb}{\operatorname{Hilb}}

\newcommand{\Edges}{\mathsf{Edges}}
\newcommand{\Vertices}{\mathsf{Vertices}}
\newcommand{\Esf}{\mathsf{E}}
\newcommand{\Vsf}{\mathsf{V}}

\newcommand{\Nsf}{\mathsf{N}}
\newcommand{\Ssf}{\mathsf{S}}
\newcommand{\SEsf}{\mathsf{SE}}
\newcommand{\SVsf}{\mathsf{SV}}
\newcommand{\CY}{Ca\-la\-bi-Yau }
\newcommand{\DT}{DT }
\newcommand{\GW}{GW }
\newcommand{\CYthree}{CY3 }
\newcommand{\CYthrees}{CY3s }

\newcommand{\Ext}{\operatorname{Ext}}
\newcommand{\cExt}{\mathcal{E}xt}
\newcommand{\cHom}{\mathcal{H}om}

\newcommand{\X}{\mathcal{X}}
\newcommand{\pt}{\mathrm{pt}}
\renewcommand{\P}{\mathbb{P}}
\renewcommand{\O}{\mathcal{O}}

\newcommand{\football}{\P^{1} _{k_{0},k_{\infty }}}

\newcommand{\cC}{\mathcal{C}}
\newcommand{\cF}{\mathcal{F}}
\newcommand{\cG}{\mathcal{G}}

\renewcommand{\hat}{\widehat}
\renewcommand{\tilde}{\widetilde}
\newcommand{\QED}{\qed \bigskip}
\newcommand{\coeff}{\operatorname{Coeff}}
\newcommand{\Tot}{\operatorname{Tot}}

\newcommand{\frakq}{\mathfrak{q}}
\newcommand{\spaceship}[1]{\underset{#1}{\prec \negthinspace \succ}}
\newcommand{\Mon}{\mathsf{Mon}}
\newcommand{\Hook}{\mathsf{Hook}}
\newcommand{\Vempty}{\Vsf ^{n}_{\emptyset \emptyset \emptyset }}

\newcommand{\restricts}[1]{\bigg|_{#1}}

\begin{document}

\begin{abstract}
We define Donaldson-Thomas invariants of Calabi-Yau orbifolds and we
develop a topological vertex formalism for computing them. The basic
combinatorial object is the orbifold vertex $\Vsf ^{G}_{\lambda \mu
\nu }$, a generating function for the number of 3D partitions
asymptotic to 2D partitions $\lambda ,\mu ,\nu $ and colored by
representations of a finite Abelian group $G$ acting on $\cnums
^{3}$. In the case where $G\cong \znums _{n}$ acting on $\cnums ^{3}$
with transverse $A_{n-1}$ quotient singularities, we give an explicit
formula for $\Vsf ^{G}_{\lambda \mu \nu }$ in terms of Schur
functions. We discuss applications of our formalism to the
Donaldson-Thomas Crepant Resolution Conjecture and to the orbifold
Donaldson-Thomas/Gromov-Witten correspondence. We also explicitly
compute the Donaldson-Thomas partition function for some simple
orbifold geometries: the local football $\P ^{1}_{a,b}$ and the local
$B\znums _{2}$ gerbe.
\end{abstract}

\maketitle

\markright{The Orbifold Topological Vertex}

\tableofcontents

\section{Introduction}

The topological vertex is a powerful tool for computing the
Gromov-Witten (GW) or Donaldson-Thomas (DT) partition function of any
toric Calabi-Yau threefold (toric CY3).  The vertex was originally
discovered in physics using the duality between Chern-Simons theory
and topological string theory \cite{AKMV}. A mathematical treatment of
the topological vertex in \GW theory was given in
\cite{Liu-Liu-Zhou,Li-Liu-Liu-Zhou,MOOP}, and the topological vertex
for \DT theory was developed in \cite{MNOP1}, where it was used to
prove the DT/GW correspondence in the toric \CYthree case.

In this paper we develop a topological vertex formalism which computes
the \DT partition function of an \emph{orbifold} toric
\CYthree. 

The central object in our theory is the orbifold vertex $\Vsf
^{G}_{\lambda \mu \nu }$. It is a generating function for the number
of 3D partitions, colored by representations of $G$, and asymptotic to
a triple of 2D partitions $(\lambda ,\mu, \nu )$. Here $G$ is an
Abelian group acting on $\cnums ^{3}$ with trivial determinant and the
action dictates a fixed coloring scheme for the boxes in the 3D
partition (see \S~\ref{subsec: 3D parts, 2D parts, and the
vertex}). The usual topological vertex is the case where $G$ is the
trivial group.

Associated to an orbifold toric \CYthree $\X $ is a trivalent graph
whose vertices are the torus fixed points and whose edges are the
torus invariant curves. There is additional data at the vertices
describing the stabilizer group of the fixed point and there is
additional data at the edges giving the degrees of the line bundles
normal to the fixed curve. The general orbifold vertex formalism
determines the \DT partition function $DT (\X )$ by a
formula of the form
\begin{equation}\label{eqn: general vertex formula}
DT (\X )=
\sum _{\begin{smallmatrix} \text{edge}\\
\text{assignments} \end{smallmatrix}} \prod _{e\in \Edges }
\Esf (e) \prod _{v\in \Vertices }\hat{\Vsf} ^{G
}_{\lambda \mu \nu } (v)
\end{equation}
where the sum is over all ways of assigning 2D partitions to the
edges. The edge terms $\Esf (e)$ are relatively simple and depend on
the normal bundle of the corresponding curve as well as the partition
assigned to the edge. The vertex terms $\hat{\Vsf }^{G}_{\lambda \mu
\nu } (v)$ are given by the universal series $\Vsf ^{G}_{\lambda \mu
\nu }$ modified by certain signs with $G,\lambda, \mu, \nu $
obtained as the local group of the vertex $v$ and the partitions along
the incident edges.

To make the above formula computationally effective, one needs a
closed formula for the universal series $\Vsf ^{G}_{\lambda \mu \nu
}$.  One of our main results is Theorem~\ref{thm: formula for Z_n
vertex} which gives an explicit formula, in terms of Schur functions,
for $\Vsf ^{G}_{\lambda \mu \nu }$ in the case where $G$ is $\znums
_{n}$ acting on $\cnums ^{3}$ with weights $(1,-1,0)$. This
corresponds to the case where the orbifold structure of $\X $ occurs
along smooth, disjoint curves which then necessarily have transverse
$A_{n-1}$ singularities ($n$ can vary from curve to curve). We call
this the \emph{transverse $A_{n-1}$ case} and we make the above
formula fully explicit in that instance (Theorem~\ref{thm: main
formula for transverse An}).

Besides providing a tool to compute \DT partition
functions of orbifolds, our orbifold vertex formalism gives insight
into two central questions in the \DT theory of
orbifolds.
\begin{itemize}
\item How is the \DT theory of an orbifold $\X $ related
to the \GW theory of $\X $?
\item How is the \DT theory of $\X $ related to the
\DT theory of $Y$, a \CY resolution of $X$, the
singular space underlying $\X $?
\end{itemize}

The four relevant theories can be arranged schematically in the
diagram below:

\begin{center}
\setlength{\unitlength}{.5cm}
\begin{picture}(10,8) (1,-1) 
\thicklines
\put(0,0){$DT (\X )$}
\put(10,0){$GW (\X )$}
\put(0,5){$DT (Y )$}
\put(10,5){$GW (Y )$}
\put(3,.3){\line(1,0){6.5}}
\put(3,5.3){\line(1,0){6.5}}
\put(1,1.3){\line(0,1){3}}
\put(11,1.3){\line(0,1){3}}
\put(5,6.4){\scriptsize DT/GW }
\put(3.8,5.7){\scriptsize correspondence }
\put(4,1.4){\scriptsize orbi-DT/GW }
\put(3.8,0.7){\scriptsize correspondence }
\put(-4.0,3.1){\scriptsize DT crepant}
\put(-5.1,2.4){\scriptsize resolution conjecture }
\put(12.6,3.1){\scriptsize GW crepant}
\put(11.5,2.4){\scriptsize resolution conjecture }
\end{picture}
\end{center}

In the transverse $A_{n-1}$ case, or more generally when $\X $
satisfies the Hard Lefschetz condition \cite[Defn~1.1]{Bryan-Graber}
c.f. \cite[Lem~24]{Bryan-Gholampour3}, the (conjectural) equivalences
of the four theories take on a particularly nice form. Namely, the
(suitably renormalized) partition functions of the four theories are
equal after a change of variables and analytic continuation. For the
top equivalence, this is the famous DT/GW correspondence of Maulik,
Nekrasov, Okounkov, and Pandharipande \cite{MNOP1}, for the right
equivalence, this is the Bryan-Graber version of the crepant
resolution conjecture in \GW theory \cite{Bryan-Graber}.

In \S\ref{sec: applications of vertex}, we formulate the \DT crepant
resolution conjecture for $\X $ satisfying the hard Lefschetz
condition.  In a forthcoming paper \cite{Bryan-Cadman-Young-DTCRC}, we
will use our orbifold vertex to prove the conjecture for the case
where $\X $ is toric with transverse $A_{n-1}$ orbifold structure. We
will also formulate an orbifold version of the DT/GW
correspondence. This correspondence can be proved for a large class of
toric orbifolds with transverse $A_{n-1} $ structure by using the
other three equivalences in the diagram: our proof of the \DT
correspondence, the (non-orbifold) DT/GW correspondence of
\cite{MNOP1}, and a proof of the \GW crepant resolution conjecture for
a large class of toric orbifolds with transverse $A_{n-1}$ structure
which has been obtained by Coates and
Iritani\cite{Coates-Ititani-communication}.

Our paper is organized as follows.  In \S~\ref{sec: orbi-CY3s and DT
theory}, we define \DT theory for orbifolds. In \S~\ref{sec: the
vertex} we introduce the vertex formalism and give our main two
results: Theorem~\ref{thm: main formula for transverse An}, an
explicit formula for the partition function of an orbifold toric
\CYthree with transverse $A_{n-1}$ orbifold structure and
Theorem~\ref{thm: formula for Z_n vertex}, a formula for the $\znums
_{n}$ vertex in terms of Schur functions. In \S~\ref{sec: applications
of vertex}, we formulate the \DT crepant resolution conjecture. We
then use our vertex formalism to compute the partition function of the
local football (Proposition~\ref{prop: DT of local football}) and the
local $B\znums _{2}$-gerbe (\S~\ref{subsec: the local BZ2
gerbe}). Each of these examples is used to illustrate the \DT crepant
resolution conjecture and the orbifold DT/GW correspondence. The
derivation of the vertex formalism and the proof of Theorem~\ref{thm:
main formula for transverse An} begins in \S~\ref{sec: proof of main
thm}. A key component of the proof is a K-theory decomposition of the
structure sheaf of a torus invariant substack into edge and vertex
terms (Propositions~\ref{prop: K-theory decomposition of O_Y} and
\ref{prop: K-theory decomp of an edge} and Lemma~\ref{lem: vertex
K-theory decomp}). The proof of Theorem~\ref{thm: main formula for
transverse An} is finished in \S~\ref{sec: signs} where the signs in
the vertex formula are derived. Finally, a proof of Theorem~\ref{thm:
formula for Z_n vertex} is given in \S~\ref{sec: proof of Zn vertex
formula} using vertex operators. Necessary background on orbifold
toric \CYthrees and orbifold Riemann-Roch is collected in two brief
appendices.


\section{Orbifold \CYthrees and \DT theory}\label{sec: orbi-CY3s and DT theory}

\subsection{Orbifold \CYthrees} An \emph{orbifold \CYthree}
is defined to be a smooth, quasi-projective, Deligne-Mumford stack $\X
$ over $\cnums $ of dimension three having generically trivial
stabilizers and trivial canonical bundle,
\[
K_{\X }\cong \O_{\X }.
\]
The definition implies that the local model for $\X $ at a point $p$
is $[\cnums ^{3}/G_{p}]$ where $G_{p}\subset SL (3,\cnums )$ is the
(finite) group of automorphisms of $p$.

\subsection{The $K$-theory of $\X $.}  Our \DT invariants
will be indexed by compactly supported elements of $K$-theory, up to
numerical equivalence. Let $K_{c} (\X )$ be the Grothendieck group of
compactly supported coherent sheaves on $\X $. We say that
$F_{1},F_{2}\in K_{c} (\X )$ are \emph{numerically equivalent},
\[
F_{1}\sim_{num} F_{2},
\]
if
\[
\chi (E\otimes F_{1}) = \chi (E\otimes F_{2})
\]
for all sheaves $E$ on $\X $.

In this paper, $K$-theory will always mean compactly supported
$K$-theory modulo numerical equivalence:
\[
K (\X ) = K_{c} (\X )/\sim _{num}.
\]
There is a natural filtration
\[
F_{0}K(\X)\subset F_{1}K(\X)\subset F_{2}K (\X )\subset K (\X )
\]
given by the dimension of the support. An element of $F_{d}K (\X )$ can
be represented by a formal sum of sheaves having support of dimension
$d$ or less.

\subsection{The Hilbert scheme of substacks} Given $\alpha \in K (\X
)$, we define
\[
\Hilb ^{\alpha } (\X )
\]
to be the category of families
of substacks $Z\subset \X $ having $[\O _{Z}]=\alpha $. By a theorem
of Olsson-Starr \cite{Olsson-Starr}, $\Hilb ^{\alpha } (\X )$ is
represented by a scheme which we also denote by $\Hilb ^{\alpha } (\X
)$. Note that our indexing is slightly different than Olsson-Starr
who index instead by the corresponding Hilbert function
\[
E\mapsto \chi (E\otimes \alpha ).
\]
Note that the Hilbert scheme $\Hilb ^{\alpha } (\X )$ is a scheme rather
than just a stack, as its objects (substacks $Z\subset \X $) do not
have automorphisms.

\subsection{Definition of \DT invariants} In
\cite{Behrend-micro}, Kai Behrend defined an integer-valued
constructible function
\[
\nu_{S}: S\to \znums
\]
associated to any scheme $S$ over $\cnums $.

\begin{defn}\label{defn: DT invariant}
The \DT invariant of $\X $ in the class $\alpha \in K (\X
)$ is given by the topological Euler characteristic of $\Hilb ^{\alpha
} (\X )$, weighted by Behrend's function $\nu :\Hilb ^{\alpha } (\X
)\to \znums $. That is
\begin{align*}
DT_{\alpha } (\X )&=e (\Hilb ^{\alpha } (\X ),\nu )\\
&=\sum _{k\in \znums }k \, e\left(\nu ^{-1} (k) \right)
\end{align*}
where $e (-)$ is the topological Euler characteristic.
\end{defn}

\begin{remark}
In the case where $\X $ is compact and a scheme, and $\alpha \in F_{1}K (\X
)$, our definition coincides (via Behrend
\cite[Theorem~4.18]{Behrend-micro}) with the definition given in
\cite{MNOP1} which uses a perfect obstruction theory. It should
be possible to construct a perfect obstruction theory on $\Hilb
^{\alpha } (\X )$ along the lines of \cite{MNOP1,Thomas}, but we don't
pursue that in this paper. One advantage of defining the invariants
directly in terms of the weighted Euler characteristic is that
$DT_{\alpha } (\X )$ is well defined for non-compact geometries.
\end{remark}

\begin{remark}
If $\alpha =[\O _{Z}]\in F_{1}K(\X)$ and $\X =X$ is a scheme, we can
recover the more familiar discrete invariants $n=\chi (\O _{Z})$ and
$\beta =[Z]\in H_{2} (X)$ via the Chern character:
\[
ch (\O _{Z}) = [Z]^{\vee } + \chi (\O _{Z}) [pt]^{\vee }.
\]
\end{remark}

\subsection{\DT partition functions.}

We define the \emph{\DT partition function} by
\[
DT (\X ) = \sum _{\alpha \in F_{1}K(\X)} DT_{\alpha } (\X ) q^{\alpha }.
\]

With an appropriate choice of a basis $e_{1},\dotsc ,e_{r}$ for $F_{1}K (\X
)$, we can regard $DT (\X )$ as a formal Laurent series in
a set of variables $q_{1},\dotsc ,q_{r}$ where
\[
q^{\alpha } = q_{1}^{d_{1}}\dotsb q_{r}^{d_{r}}
\]
for $\alpha =\sum _{i=1}^{r}d
_{i}e_{i}$.

We define the \emph{degree zero} \DT partition function by
\[
DT_{0} (\X ) = \sum _{\alpha \in F_{0}K(\X)}DT_{\alpha } (\X )q^{\alpha },
\]
and we define the \emph{reduced} \DT partition function by
\[
DT' (\X ) = \frac{DT (\X )}{DT_{0} (\X )}.
\]

In the case where $\X =X$ is a scheme, Maulik, Nekrasov, Okounkov, and
Pandharipande conjectured that the reduced \DT partition
function is equal to the reduced \GW partition function
after a change of variables \cite[Conjecture~2]{MNOP1}.

\section{The Orbifold Vertex Formalism}\label{sec: the vertex}

In the case where $\X =X$ is a scheme and toric, the topological
vertex formalism computes the \DT partition function $DT (X )$ in
terms of the topological vertex $\Vsf _{\lambda \mu \nu }$, a
universal object which is a generating function for 3D partitions
asymptotic to $(\lambda ,\mu ,\nu )$. We extend the vertex formalism
to toric orbifolds, particularly in the case where $\X $ has
transverse $A_{n-1}$ orbifold structure.

\subsection{3D partitions, 2D partitions, and the vertex.} \label{subsec: 3D parts, 2D parts, and the vertex}

\begin{defn}\label{defn: 3D partition asympt to (a,b,c)} Let $(\lambda
,\mu ,\nu )$ be a triple of ordinary partitions. A \emph{3D
partition $\pi $ asymptotic to $(\lambda ,\mu ,\nu )$} is a subset
\[
\pi \subset \left(\znums _{\geq 0} \right)^{3}
\]
satisfying
\begin{enumerate}
\item if any of $(i+1,j,k)$, $(i,j+1,k)$, and $(i,j,k+1)$ is in $\pi
$, then $(i,j,k)$ is also in $\pi $, and
\item
\begin{enumerate}
\item $(j,k)\in \lambda $ if and only if $(i,j,k)\in \pi $ for all $i\gg 0$,
\item $(k,i)\in \mu  $ if and only if $(i,j,k)\in \pi $ for all $j\gg 0$,
\item $(i,j)\in \nu  $ if and only if $(i,j,k)\in \pi $ for all $k\gg 0$.
\end{enumerate}
\end{enumerate}
where we regard ordinary partitions as finite subsets of $\left(\znums
_{\geq 0} \right)^{2}$ via their diagram.
\end{defn}

Intuitively, $\pi $ is a pile of boxes in the positive octant of
3-space.  Condition (1) means that the boxes are stacked stably with
gravity pulling them in the $(-1,-1,-1)$ direction; condition (2)
means that the pile of boxes is infinite along the coordinate axes
with cross-sections asymptotically given by $\lambda $, $\mu $, and
$\nu $.

The subset $\{(i,j,k ): (j,k)\in \lambda \}\subset \pi $ will be
called the \emph{leg} of $\pi $ in the $i$ direction, and the legs in
the $j$ and $k$ directions are defined analogously. Let
\begin{equation}\label{eqn: xi=1-number of legs}
\xi _{\pi } (i,j,k) = 1 - \# \text{ of legs of $\pi $ containing }
(i,j,k) .
\end{equation}

We define the renormalized volume of $\pi $ by
\[
|\pi | = \sum _{(i,j,k)\in \pi } \xi _{\pi } (i,j,k).
\]
Note that $|\pi |$ can be negative.
\begin{defn}\label{defn: box counting vertex}
The topological vertex $\Vsf_{\lambda \mu \nu }$ is defined to be
\[
\Vsf _{\lambda \mu \nu } = \sum _{\pi } q^{|\pi |}
\]
where the sum is taken over all 3D partitions $\pi $ asymptotic to
$(\lambda ,\mu ,\nu )$. We regard $\Vsf _{\lambda \mu \nu }$ as a
formal Laurent series in $q$. Note that $\Vsf _{\lambda \mu \nu }$ is
clearly cyclically symmetric in the indices, and reflection about the
$i=j$ plane yields
\[
\Vsf _{\lambda \mu \nu } = \Vsf _{\mu '\lambda '\nu '}
\]
where $'$ denotes conjugate partition:
\[
\lambda ' = \{(i,j): (j,i)\in \lambda  \}.
\]

\end{defn}
This definition of topological vertex differs from the vertex $C
(\lambda ,\mu ,\nu )$ of the physics literature by a normalization
factor. Our $\Vsf _{\lambda \mu \nu }$ is equal to $P (\lambda ,\mu
,\nu )$ defined by Okounkov, Reshetikhin, and Vafa
\cite[eqn~3.16]{Ok-Re-Va}. They derive an explicit formula for $\Vsf
_{\lambda \mu \nu }=P (\lambda, \mu, \nu )$ in terms of Schur
functions \cite[eqns~3.20 and 3.21]{Ok-Re-Va}.

The $\znums _{n}$ orbifold vertex counts 3D partitions colored with $n
$ colors.  We color the boxes of a 3D partition $\pi $ according to
the rule that a box $(i,j,k)\in \pi $ has color $i-j \mod n$
(c.f. \cite{Bryan-Young}).

\begin{defn}\label{defn: Zn vertex}
The $\znums _{n}$ vertex $\Vsf ^{n}_{\lambda \mu \nu }$ is defined by
\[
\Vsf ^{n}_{\lambda \mu \nu } = \sum _{\pi } q_{0}^{|\pi |_{0}}\dotsb q_{n-1}^{|\pi |_{n-1}}
\]
where the sum is taken over all 3D partitions $\pi $ asymptotic to
$(\lambda, \mu, \nu )$ and $|\pi |_{a}$ is the (normalized) number of
boxes of color $a$ in $\pi $. Namely
\[
|\pi |_{a} = \sum _{\begin{smallmatrix} i,j,k\in \pi \\
i-j=a\mod n \end{smallmatrix}} \xi _{\pi } (i,j,k)
\]
where $\xi _{\pi}$ is defined in equation~\eqref{eqn: xi=1-number of
legs}.
\end{defn}
 
Note that the $\znums _{n}$-orbifold vertex $\Vsf ^{n}_{\lambda \mu
\nu }$ has fewer symmetries than the usual vertex since the $k $ axis
is distinguished. However, reflection through the $i=j $ plane yields
\[
\Vsf ^{n}_{\lambda \mu \nu } (q_{0},q_{1},\dotsc ,q_{n-1}) = \Vsf
^{n}_{\mu '\lambda '\nu '} (q_{0},q_{n-1},\dotsc ,q_{1}).
\]
In general, if $F$ is a series in the variables $q_{k}$ with $k\in
\znums _{n}$, we let $\overline{F}$ denote the same series with the
variable $q_{k} $ replaced by $q_{-k}$. So for example, the above
symmetry can be written
\[
\Vsf ^{n}_{\lambda \mu \nu } = \overline{\Vsf }^{n}_{\mu '\lambda '\nu '}.
\]

The $G$ vertex is defined in general as follows. Given a finite
Abelian group $G$ acting on $\cnums ^{3}$ via characters
$r_{1},r_{2},r_{3}$ we define
\begin{equation}\label{eqn: definition of VG}
\Vsf ^{G}_{\lambda \mu \nu } = \sum _{\pi } \prod _{r\in \hat{G}}
q_{r}^{|\pi |_{r}}
\end{equation}
where the sum is over 3D partitions asymptotic to $(\lambda ,\mu ,\nu
)$ and where $|\pi |_{r}$ is the (normalized) number of boxes in $\pi
$ of color $r\in \hat{G}$:
\[
|\pi |_{r} = \sum _{\begin{smallmatrix} i,j,k\in \pi \\
r_{1}^{i}r_{2}^{j}r_{3}^{k} =r \end{smallmatrix}}\xi _{\pi } (i,j,k).
\]

One of our main results is an explicit formula for the $\znums
_{n}$-orbifold vertex (see Theorem~\ref{thm: formula for Z_n vertex}).

\subsection{Orbifolds with transverse $A_{n-1}$
singularities.}\label{subsec: conventions for An-1 toric orbifolds}

Let $\X $ be a orbifold toric \CYthree whose orbifold structure is
supported on a disjoint union of smooth curves. Then the local group
along each curve is $\znums _{n}$ (where $n$ can vary from curve to
curve) and the coarse space $X$ has transverse $A_{n-1}$ singularities
along the curves. By
Lemma~\ref{lem:toric_orbiCY_determined_by_cspace}, $\X $ is determined
by its coarse space $X$.

The combinatorial data determining a toric variety $X$ is well
understood and is most commonly expressed as the data of a fan (by the
Lemma~\ref{lem:toric_orbiCY_determined_by_cspace}, we do not require
the stacky fans of Borisov, Chen and Smith
\cite{Borisov-Chen-Smith}). In the case of a orbifold toric \CYthree,
it is convenient to use equivalent (essentially dual) combinatorial
data, namely that of a $(p,q)$-web diagram. Web diagrams are discussed
in more detail in \S~\ref{app: web diagrams}.

Associated to $\X $ is a planar trivalent graph $\Gamma=\{\Edges
,\Vertices \} $ where the vertices correspond to torus fixed points,
the edges correspond to torus fixed curves, and the regions in the
plane delineated by the graph correspond to torus fixed
divisors. $\Gamma $ will necessarily have some non-compact edges;
these correspond to non-compact torus fixed curves. We denote the set
of compact edges by $\Edges ^{cpt}.$

It will be notationally convenient to choose an \emph{orientation} on
$\Gamma $:
\begin{defn}\label{defn: orientation on a graph}
Let $\Gamma $ be a trivalent planar graph. An \emph{orientation} is a
choice of direction for each edge and an ordering $(e_{1} (v),e_{2}
(v),e_{3} (v))$ of the edges incident to each vertex $v$ which is
compatible with the counterclockwise cyclic ordering.
\end{defn}

Given an orientation on the graph $\Gamma $ associated to $\X $, let
the two regions in the plane incident to an edge $e$ be denoted by $D
(e)$ and $D' (e)$ with the convention that $D (e)$ lies to the right
of $e$ (see Figure~\ref{fig: edge e with f,f',g,g'}). We also use $D
(e)$ and $D' (e)$ to denote the corresponding torus invariant divisors
and we let $C (e)\subset \X $ denote the torus invariant curve
corresponding to $e$. Let $p_{0} (e)$ and $p_{\infty } (e)$ denote the
the torus fixed points corresponding to the initial and final vertices
incident to $e$. Let $D_{0} (e)$ and $D_{\infty } (e)$ denote the
torus invariant divisors meeting $C (e)$ transversely at $p_{0} (e)$
and $p_{\infty } (e)$ respectively.  Let $D_{1} (v),D_{2} (v),D_{3}
(v)$ be the regions (and the corresponding torus invariant divisors)
opposite the edges $e_{1} (v),e_{2} (v),e_{3} (v)$.

Let
\begin{align*}
m=m (e)&= \deg \O _{C (e)} (D (e))\\
m'=m' (e)&= \deg \O _{C (e)} (D' (e)).
\end{align*}

\begin{figure}[h!]
\setlength{\unitlength}{.7cm}
\begin{picture}(25,6) (10,3.5)
\thicklines
\put(22,6){\line(1,1){2}}
\put(22,6){\line(1,-1){2}}
\put(16,6){\line(1,0){6}}
\put(16,6){\line(-1,1){2}}
\put(16,6){\line(-1,-1){2}}
\put(22,6){\vector(1,1){1.3}} 
\put(23.5,4.5){\vector(-1,1){.5}}  
\put(16,6){\vector(-1,-1){1.3}}  
\put(14.5,7.5){\vector(1,-1){.5}} 
\put(16,6){\vector(1,0){2.5}}  

\put(19,5.5){$e$}
\put(21.9,7.0){$g'$}
\put(15.3,7.0){$f'$}
\put(21.9,4.5){$g$}
\put(15.2,4.5){$f$}
\put(18.5,7){$D' (e)$}
\put(18.5,4.5){$D (e)$}
\put(13.3,6){$D_{0} (e)$}
\put(23,6){$D_{\infty } (e)$}
\end{picture}
\caption{The edge $e$ with orientations chosen for adjacent edges.}\label{fig: edge e with f,f',g,g'}
\end{figure}
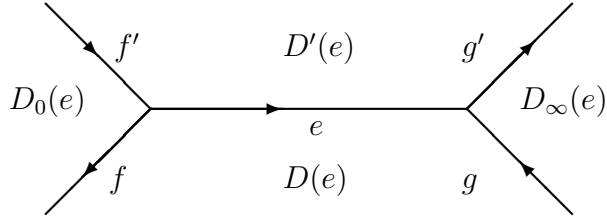

Define $n (e)$ such that $\znums _{n (e)}$ is the local group of $C
(e)\subset \X $. If $n (e)\neq 1$, then $C (e)$ is a $B\znums _{n
(e)}$ gerbe over $\P ^{1}$ and
\[
m ,m'\in \frac{1}{n (e)}\znums 
\]
with 
\[
m +m' =-2.
\]
If $n
(e)=1$, then one of
\[
a=n (f),\quad 
a'=n (f'),
\]
and/or one of
\[
b=n (g),\quad 
b'=n (g'),
\]
is possibly greater than one and $C (e)$ is a football: a $\P ^{1}$
with root constructions of order $\max (a,a')$ and $\max (b,b')$ at $0
$ and $\infty $.

We define
\[
\delta _{0}=\begin{cases}
1&\text{ if $a>1$}\\
0&\text{ if $a=1$}.
\end{cases}
\]
We define $\delta _{0}'$, $\delta _{\infty }$, and $\delta '_{\infty
}$ similarly according to the values of $a'$, $b$, and $b'$
respectively. Note that at least one of $(\delta _{0},\delta _{0}')$
is zero and likewise for $(\delta _{\infty },\delta _{\infty
}')$. Using the condition that $\O _{C} (D+D')=K_{C}=\O _{C}
(-p_{0}-p_{\infty })$, we can write
\begin{align*}
\O _{C} (D)&= \O _{C} (\tilde{m}p-\delta _{0}p_{0}-\delta _{\infty }p_{\infty }),\\
\O _{C} (D')&= \O _{C} (\tilde{m}'p-\delta' _{0}p_{0}-\delta' _{\infty }p_{\infty }),
\end{align*}
where
\begin{align*}
m&=\tilde{m}-\frac{\delta _{0}}{a}-\frac{\delta _{\infty }}{b},\\
m'&=\tilde{m}'-\frac{\delta' _{0}}{a'}-\frac{\delta' _{\infty }}{b'}
\end{align*}
since $p_{0},p_{\infty }\in C$ are orbifold points of order $\max
(a,a')$ and $\max (b,b')$ respectively. Note that $\tilde {m},\tilde {m}'\in \znums $ and the \CY condition implies
\[
\tilde {m}+\tilde {m}' = \delta _{0}+\delta _{0}' +\delta _{\infty } +\delta _{\infty }'-2.
\]
By convention, we define $\tilde {m}=m$ and $\tilde {m}'=m'$ if $n
(e)>1$ (but note that in this case, $\tilde {m}$ and $\tilde {m}'$ may not be integers).

\subsection{Generators for $F_{1}K (\X )$}\label{subsec: generators for K}

To write an explicit formula for $DT (\X )$, we must choose generators
for $F_{1}K (\X )$. Let $p\in \X $ be a generic point and let $p (e)$
be a generic point on the curve $C (e)$ (so $p (e)\cong B\znums _{n
(e)}$). Let $\rho _{k}$, $k\in \{0,\dotsc ,n (e)-1 \}$ be the
irreducible representations of $\znums _{n (e)}$ with the indexing
chosen so that
\[
\O _{p (e)} (-kD (e)) \cong \O _{p (e)}\otimes \rho _{k}.
\]
We define the following classes in $F_{1}K (\X )$
and their associated variables.
\smallskip

\begin{center}
\begin{tabular}{|c|c|c|}
\hline 
Class in $F_{1}K (\X )$&Associated variable & Indexing set\\
\hline 
&      &	\\
$[\O _{p}]$&$q$&\\
&    &	\\
\hline 
&      &\\	
$[\O _{p (e)}\otimes \rho _{k}]$&$q_{e,k}$&$ e\in \Edges 
,\quad k\in  \{{\scriptstyle 0,\dotsc ,n (e)-1} \}$\\
&    &	 \\
\hline
&	&	\\
$[\O _{C (e)} (-1)\otimes \rho _{k}]$&$v_{e,k}$&$e\in \Edges ^{cpt}, k\in \{{\scriptstyle 0,\dotsc ,n (e)-1} \}$\\
&    &	 \\
\hline 
\end{tabular}
\end{center}
\smallskip

Pushforwards by the inclusions of $p$, $p (e)$, and $C (e) $ into $\X
$ are implicit in the above. The class $[\O _{C (e)} (-1)\otimes \rho
_{k}]$ is defined as follows. The curve $C (e)$ is a $B\znums _{n
(e)}$ gerbe over $\P ^{1}$. If $C (e)\cong \P^1\times B\znums _{n
(e)}$ is the trivial gerbe, then $\O _{C (e)} (-1)$ is pulled back
from $\P ^{1}$ and $\rho _{k}$ is pulled back from $B\znums _{n
(e)}$. More generally, let $\pi :\tilde{C} (e)\to C (e)$ be the degree
$n$ cover obtained from the base change $\P ^{1}\to \P ^{1}$,
$z\mapsto z^{n}$. Then $\tilde{C} (e)$ is the trivial $B\znums _{n
(e)}$ gerbe and we define $[\O _{C (e)} (-1)\otimes \rho _{k}]$ to be
the class $\frac{1}{n}\pi _{*}[\O _{\tilde{C} (e)} (-1)\otimes \rho
_{k}]$. In general, this class is not defined with $\znums $
coefficients.

The above classes generate $F_{1}K (\X )$ (over $\qnums $) but there
are relations. In particular, for each $e\in \Edges $, there is the
relation
\begin{equation}\label{eqn: O_p=O_p(e)otimes Reg}
[\O _{p}]=[\O _{p (e)}\otimes R_{reg}]
\end{equation}
where $R_{reg}=\sum _{k}\rho _{k}$ denotes the regular representation
of $\znums _{n (e)}$. This relation gives rise to the relation
\[
q=\prod _{k=0}^{n (e)-1}q_{e,k}.
\]
There may be additional relations among the classes supported on
curves coming from the global geometry of $\X $. We leave relations
among the corresponding variables implicit in all our formulas.

\begin{remark}
If $n (e)=1$ for all edges $e$, then $\X =X$ is not an orbifold. In
this case, the only variables are $q$ corresponding to $[\O _{p}]$ and
$v_{e}$ corresponding to $\O _{C (e)} (-1)$. If $Z\subset X$
is a subscheme with $\chi (\O _{Z})=n$ and 
\[
\beta =[Z]=\sum _{i}d_{i}[C(e_{i})],
\]
then
\[
[\O _{Z}] = n[\O _{p}] +\sum _{i}d_{i}[\O _{C (e_{i})} (-1)]
\]
in K-Theory. Thus the associated \DT invariant appears as
the coefficient of $q^{n}v^{\beta }=q^{n}\prod _{i}v_{i}^{d_{i}}$
which is consistent with the notation of \cite{MNOP1}.
\end{remark}
 
\subsection{The vertex formula}\label{subsec: the vertex formula}
Let 
\[
\lambda [k,n]  = \{(i,j)\in \lambda :i-j=k\mod n \}
\]
be the set of boxes in $\lambda $ of color $k\mod n$. Let
\[
|\lambda |_{k}=|\lambda [k,n]|
\]
denote the number of boxes of color $k$ in $\lambda $. Usually, $n$ is
understood from the context, but if we need to make it explicit, we
write $|\lambda |_{k,n}$.

\begin{defn}\label{defn: edge assignment}
An \emph{edge assignment} on $\Gamma $ is a choice of a partition
$\lambda (e)$ for each edge $e$ such that $\lambda (e)=\emptyset $ for
every non-compact edge. An edge assignment is called
\emph{multi-regular} if each $\lambda =\lambda (e)$ satisfies
$|\lambda |_{k} = \frac{1}{n}|\lambda |$ for all $k$.
\end{defn}

Assume that $\Gamma $ has an orientation (Definition~\ref{defn:
orientation on a graph}).  Given an edge assignment and a vertex $v$,
we get a triple of partitions $(\lambda _{1} (v),\lambda _{2}
(v),\lambda _{3} (v))$ by setting $\lambda _{i} (v)=\lambda (e_{i}
(v))$ if $e (v_{i}) $ has the orientation pointing outward from $v$
and $\lambda _{i} (v)=\lambda (e_{i} (v))'$ if $e_{i} (v)$ has the
inward orientation. We also impose the convention that if any of the
edges $e_{i} (v)$ have $n (e_{i} (v))\neq 1$, then we fix the ordering
so that this (necessarily unique) edge is given by $e_{3} (v)$. We
will call such an edge the \emph{special edge} and denote it also as
simply $e (v)$.

The following quantities are used in the vertex formula. Let
\[
C^{\lambda }_{\tilde{m},\tilde{m}'} = \sum _{(i,j)\in \lambda }
-\tilde{m}i-\tilde{m}'j+1.
\]
and let
\[
C^{\lambda }_{\tilde{m},\tilde{m}'}[k,n] = \sum _{(i,j)\in
\lambda[k,n] } -\tilde{m}i-\tilde{m}'j+1.
\]

We define
\[
A_{\lambda } (k,n) =\sum _{(i,j)\in \lambda } \left\lfloor
\frac{i+k}{n} \right\rfloor.
\]

Let $e=e (v)$ be the special edge associated to a vertex. We write
\[
q_{v} = \begin{cases}
(q_{e,0},q_{e,1}\dotsb ,q_{e,n (e)-1} )& \text{ if $e$ is oriented outward from $v$ and}\\
(q_{e,0},q_{e,n (e)-1},\dotsc ,q_{e,1} )& \text{ if $e$ is oriented inward toward $v$.}
\end{cases}
\]
We define
\begin{equation}\label{eqn: signs of vertex vars}
(-1)^{s (\lambda )} q_{v}
\end{equation}
to be the same as $ q_{v}$ but with the variable $q_{e,k}$
multiplied by the additional sign $(-1)^{s_{k} (\lambda )}$ where
\[
s_{k} (\lambda ) = |\lambda |_{k-1} +|\lambda |_{k+1}.
\]
Note that this sign is trivial in the multi-regular case.

We also adopt a product convention for our variables. Namely, we set
\begin{align*}
v_{e}^{|\lambda|} &:= \prod _{k=0}^{n (e)-1} v_{e,k} ^{|\lambda |_{k,n (e)}},\\
q_{e}^{C^{\lambda }_{\tilde {m},\tilde {m}'}}&:= \prod _{k=0}^{n (e)-1}  q_{e,k} ^{C^{\lambda }_{\tilde {m},\tilde {m}'}[k,n (e)]},\\
q_{e}^{A_{\lambda } }&:=\prod _{k=0}^{n (e)-1} q_{e,k}^{A_{\lambda }
(k,n (e))}.
\end{align*}

We will need an additional sign $(-1)^{\Ssf _{\lambda (e)}( e)}$
associated to each edge $e$. Let $\lambda =\lambda (e)$, $n=n (e)$,
and let
\begin{multline*}
\Ssf_{\lambda } (e) = \sum _{k=0}^{n-1}C^{\lambda }_{m,m'}[k,n]\, \Big(|\lambda |_{k-1} - |\lambda |_{k+1} \Big)
 + | \lambda |_{k} \,\Big( 1+ (1+\tilde {m}+\delta _{0}+\delta _{\infty })|\lambda |_{k-1}\Big).
\end{multline*}
Note that in the multi-regular case this sign simplifies
significantly:
\[
(-1)^{\Ssf _{\lambda }(e)} = (-1)^{(\tilde {m}+\delta _{0}+\delta
_{\infty }) |\lambda |}
\]

Finally, we need on more sign $(-1)^{\Sigma _{\pi (v)}}$ attached to
each vertex partition. Here
\[
\Sigma _{\pi (v)} = \sum _{k=0}^{n-1} |\lambda _{3}|_{k}\left(|\lambda _{1}|_{k}+|\lambda _{2}|_{k} +|\lambda _{1}|_{k-1}+|\lambda _{2}|_{k+1} \right)
\]
where $\lambda _{1},\lambda _{2},\lambda _{3}$ are the legs of $\pi
(v)$ and the color of $(j,k)\in \lambda _{1}$, $(k,i)\in \lambda
_{2}$, and $(i,j)\in \lambda _{3}$ is given by $i-j\mod n$. Note that
in the multi-regular case, this sign is trivial. Indeed, then
$|\lambda _{3}|_{k}$ is independent of $k$ and so the sum can be
rearranged so that the other terms cancel mod 2 in pairs.

The following theorems provide an explicit formula for the
\DT partition function of a toric orbifold with
transverse $A_{n-1}$ singularities.

\begin{theorem}\label{thm: main formula for transverse An}
Let $\X $ be a orbifold toric \CYthree with transverse $A_{n-1}$
singularities and let $\Gamma $ be the diagram of $\X $. Define
$\underline {DT} (\X )$ to be
\[
\sum _{\begin{smallmatrix} \text{edge}\\
\text{assignments} \end{smallmatrix}} \prod _{e\in \Edges } \Esf
_{\lambda (e)}(e) \prod _{v\in
\Vertices } (-1)^{\Sigma _{\pi  (v)}}\Vsf ^{n (e (v))}_{\lambda _{1} (v)\lambda _{2} (v)\lambda
_{3} (v)} \left((-1)^{s (\lambda _{3} (v))} q_{v} \right)
\]
where
\[
\Esf _{\lambda}(e) = (-1)^{\Ssf_{\lambda} (e)}\,\,
v_{e}^{|\lambda |}\,
q_{e}^{C^{\lambda }_{\tilde {m} (e),\tilde {m}' (e)}}\,
\left(q_{f}^{A_{\lambda } } \right)^{\delta _{0}}
\left(q_{f'}^{A_{\lambda' }} \right)^{\delta _{0}'}
\Big(q_{g}^{A_{\lambda }} \Big)^{\delta _{\infty }}
\left(q_{g'}^{A_{\lambda' }} \right)^{\delta _{\infty }'}
\]
and where $(f,f',g,g')$ are the edges meeting $e$ arranged and
oriented as in Figure~\ref{fig: edge e with f,f',g,g'}.  Then the
\DT partition function $ DT (\X )$ is obtained from
$\underline{DT} (\X )$ by adding a minus sign to the variables
$q_{e,0}$ (and hence also to $q$). 
\end{theorem}

Note that for multi-regular edge assignments, the signs $(-1)^{\Sigma
_{\pi (v)}}$ and $(-1)^{s (\lambda _{3} (v))}$ are both 1.

\begin{remark}\label{rem: changing the orientation of an edge}
Switching the orientation of an edge $e$ has the effect of switching
the variables $q_{e,k}\leftrightarrow q_{e,n (e)-k}$, for $k=1,\dotsc
,n (e)-1$. The edge term in the formula is written for the
orientations in Figure~\ref{fig: edge e with f,f',g,g'} but is easily
modified to an arbitrary orientation using this rule.
\end{remark}

To make the above formula fully explicit, we give a closed formula for
the $\znums _{n}$ vertex $\Vsf ^{n}_{\lambda \mu \nu } (q_{0},\dotsc
,q_{n-1})$.  We first introduce a little more notation.

Consider the indices on the variables $q_{0},\dotsc ,q_{n-1}$ to be in
$\znums _{n}$ and define $\frakq _{t}$ recursively by $\frakq _{0}=1$
and
\[
\frakq _{t} = q_{t}\cdot \frakq _{t-1}
\]
for positive and negative $t$, in other words
\[
\{\dotsc ,\frakq _{-2},\frakq _{-1},\frakq _{0},\frakq
_{1},\frakq _{2},\dotsc \} = \{\dotsc
,q_{0}^{-1}q_{-1}^{-1},q_{0}^{-1},1,q_{1},q_{1}q_{2},\dotsc \}.
\]

Let
\[
q=q_{0}\dotsb q_{n-1}
\]
and let
\[
\frakq _{\bullet }  = \{\frakq _{0},\frakq _{1},\frakq _{2},\frakq _{3},\dotsc  \} = \{1,q_{1},q_{1}q_{2},q_{1}q_{2}q_{3},\dotsc  \}.
\]
Given a partition $\nu = (\nu _{0}\geq \nu _{1}\geq \dotsb )$, let
\[
\frakq _{\bullet -\nu } =\{\frakq _{-\nu _{0}},\frakq _{1-\nu _{1}},\frakq _{2-\nu _{2}},\frakq _{3-\nu _{3}},\dotsc  \}.
\]

\begin{theorem}\label{thm: formula for Z_n vertex}
The $\znums _{n}$ vertex $\Vsf ^{n}_{\lambda \mu \nu }(q_{0},\dotsc ,q_{n-1})$ is given by the following formula:
\[
\Vsf ^{n}_{\lambda \mu \nu } = \Vsf ^{n}_{\emptyset \emptyset
\emptyset } \cdot q^{-A_{\lambda }} \cdot \overline{q^{-A_{\mu '}}}
\cdot H_{\nu }\cdot O_{\nu } \cdot \sum _{\eta } q_{0}^{-|\eta |}\cdot
\overline{s_{\lambda '/\eta } (\frakq _{\bullet -\nu })}\cdot s_{\mu
/\eta } (\frakq _{\bullet -\nu' }).
\]
where $s_{\alpha /\beta }$ is the skew Schur function associated to
partitions $\beta \subset \alpha $ ($s_{\alpha /\beta }=0$ if $\beta
\not \subset \alpha $), the overline denotes the exchange of variables
$q_{k}\leftrightarrow q_{-k}$, and
\begin{align*}
H_{\nu' } &=  \prod_{(j,i) \in \nu'}\frac{1}{1-\prod_{s=1}^{n}q_s^{h^{s}_{\nu' }(j,i)}} \\
h^{s}_{\nu' } (j,i) &=\text{the number of boxes of color $s$ in the $(j,i)$-hook of $\nu '$,}\\
O_{\nu } &= \prod _{k=0}^{n-1}\Vsf ^{n}_{\emptyset \emptyset \emptyset } (q_{k},q_{k+1},\dotsc ,q_{n+k-1})^{-2|\nu |_{k} + |\nu |_{k+1}+ |\nu |_{k-1}},\\
\Vsf ^{n}_{\emptyset \emptyset \emptyset } &= M (1,q)^{n}\prod
_{0<a\leq b<n} M (q_{a}\dotsb q_{b},q) M (q^{-1}_{a}\dotsb q^{-1}_{b},q), \\
M(v,q) &= \prod _{m=1}^{\infty } \frac{1}{(1-vq^{m})^{m}}.
\end{align*}
Recall that by our product convention
\[
q^{-A_{\lambda }} = \prod _{k=0}^{n-1}q_{k}^{-A_{\lambda } (k,n)}. 
\]
\end{theorem}

Note that in the multi-regular case, $O_{\nu }=1$.

\section{Applications of the orbifold vertex}\label{sec: applications of vertex}

\subsection{The orbifold \DT  crepant resolution conjecture and the orbifold DT/GW correspondence.}

We give a brief description of the \DT Crepant Resolution Conjecture
which will be spelled out in detail in
\cite{Bryan-Cadman-Young-DTCRC}. 

Let $\X $ be an orbifold \CYthree and let $X$ be its coarse space. Let 
\[
Y=\Hilb ^{[\O _{p}]} (\X ) 
\]
be the Hilbert scheme parameterizing substacks in the class $[\O
_{p}]\in F_{0}K (\X )$. $Y$ is birational to $X$ and admits a proper
morphism $\pi :Y\to X$. By a theorem of Bridgeland, King, and Reid
\cite{BKR}, $Y$ is a smooth \CYthree and moreover, there is a
Fourier-Mukai isomorphism \cite{BKR,Chen-Tseng-orbiBKR}
\[
\Phi :K (\X )\to K (Y)
\]
defined by 
\[
E\mapsto Rq_{*}p^{*}E
\]
where 
\[
p:Z\to \X ,\quad q:Z\to Y
\]
are the projections from the universal substack $Z\subset \X \times Y$
onto each factor.

The Fourier-Mukai isomorphism does not respect the filtrations
$F_{\bullet }K (\X )$ and $F_{\bullet }K (Y)$. However, if $\X $ has
transverse $A_{n-1}$ orbifold structure, or more generally satisfies
the Hard Lefschetz condition \cite[Defn~1.1]{Bryan-Graber}
c.f. \cite[Lem~24]{Bryan-Gholampour3}, then the image of $F_{0}K (\X
)$ under $\Phi $ is contained in $F_{1}K (Y)$. We call this image
$F_{exc}K (Y)$; its elements can be represented by formal differences
of sheaves supported on the exceptional fibers of $\pi :Y\to X$. We
define the multi-regular part of $K$-theory, $F_{mr}K (\X )$, to be
the pre-image of $F_{1}K (Y)$ under $\Phi $. Its elements can be
represented by formal differences of sheaves supported in dimension
one where at the generic point of each curve in the support, the
associated representation of the stabilizer of that point is a
multiple of the regular representation. In summary, the following
filtration is respected by the Fourier-Mukai isomorphism
\[
F_{0}K (\X )\subset F_{mr} K (\X ), \quad F_{exc}K (Y)\subset F_{1}K (Y).
\]
We define the exception partition function of $Y$ and the
multi-regular partition function of $\X $ as follows
\begin{align*}
DT_{exc} (Y)& = \sum _{\alpha \in F_{exc}K (Y)} DT_{\alpha } (Y)q^{\alpha }, \\
DT_{mr} (\X )& = \sum _{\alpha \in F_{mr}K (\X )} DT_{\alpha } (Y)q^{\alpha }.
\end{align*}
We then have our \DT crepant resolution conjecture:
\begin{conj}\label{conj: DT CRC}
Let $\X $ be an orbifold \CYthree satisfying the Hard Lefschetz condition. Let $Y$ the the \CY resolution of $X$ described above. Then using $\Phi $ to identify the variables we have an equality
\[
\frac{DT_{mr} (\X )}{DT_{0} (\X )}  = \frac{DT (Y)}{DT_{exc} (Y)}.
\]
\end{conj}

The series $DT_{0} (\X )$ and $DT_{exc} (Y)$ are not unrelated. The
conjecture in \cite[Conjecture~A.6]{Bryan-Young} globalizes to
\begin{conj}\label{conj: DT0 = DTexc*DTexcbar/DT0}
Using $\Phi $ to identify variables, we have the equality
\[
DT_{0} (\X ) =\frac{DT_{exc} (Y)\tilde{DT}_{exc} (Y)}{DT_{0} (Y)}
\]
where $\tilde{DT}_{exc} (Y) (q) = DT_{exc} (Y) (q^{-1})$. 
\end{conj}

Conjecture~\ref{conj: DT CRC} will be proven in the toric transverse
$A_{n-1} $ case in \cite{Bryan-Cadman-Young-DTCRC} using the orbifold
vertex developed in this paper. Conjecture~\ref{conj: DT0 =
DTexc*DTexcbar/DT0} was proven in the transverse $A_{n-1}$ case in
\cite{Bryan-Young}\footnote{The theorem in \cite{Bryan-Young} is for
the local case $\X =[\cnums ^{3}/\znums _{n}]$. Conjecture~\ref{conj:
DT0 = DTexc*DTexcbar/DT0} is local in nature; extending from $\X
=[\cnums ^{3}/\znums _{n}]$ to $\X $ global is routine.}

We will see in the examples below that the series 
\[
DT'_{mr} (\X )  = \frac{DT_{mr} (\X )}{DT_{0} (\X )}
\]
which we call the reduced, multi-regular \DT partition function of $\X
$, is equal to the reduced \GW partition function $GW' (\X )$ after a
change of variables and analytic continuation. The general change of
variables can be formulated in terms of Iritani's stacky Mukai vector
\cite{Iritani-integral}, but we will not formulate that explicitly
here.

\subsection{Example: the local football.}

Let 
\[
\X _{a,b} = \Tot (\O (-p_{0})\oplus \O (-p_{\infty })\to \P ^{1}_{a,b})
\]
be the total space of the bundle $\O (-p_{0})\oplus \O (-p_{\infty })$
over the football $\P ^{1}_{a,b}$ which is by definition $\P ^{1}$
with root constructions \cite{Cadman} of order $a$ and $b$ at the
points $p_{0}$ and $p_{\infty }$ respectively. $\X _{a,b}$ is a
natural orbifold generalization of the resolved conifold which is the
special case $\X _{1,1}$. We use our orbifold vertex formalism to
derive a closed formula for the partition function $DT (\X _{a,b})$.

Let $\O (D) = \O (-p_{0})$ and let $\O (D')=\O (-p_{\infty })$. Then
the graph in Figure~\ref{fig: edge e with f,f',g,g'} is the whole
graph of $\X _{a,b}$ and we have
\[
n (f) = a,\quad n (g')=b, \quad n (f')=n (g)=n (e)=1,\quad \tilde {m}=\tilde {m}'=0,
\]
and so
\[
\tilde {m}+\delta _{0}+\delta _{\infty } = 1.
\]

We write our variables as follows:
\begin{align*}
p_{k} &= q_{f,k}, \quad k=0,\dotsc ,a-1\\
r_{k} &= q_{g',k}, \quad k=0,\dotsc ,b-1\\
v&=v_{e}
\end{align*}
and of course 
\[
q=p_{0}\dotsb p_{a-1} = r_{0}\dotsb r_{b-1}.
\]

As in the usual conifold case, the variables $v$ and $q$ keep track of
the degree and the holomorphic Euler characteristic of the curve
respectively. Loosely speaking, the new variables $p_{k}$ and $r_{l}$
can be thought of as keeping track of embedded points on the stacky
locus having representation $k\in \hat{\znums }_{a}$ and $l\in
\hat{\znums }_{b}$ respectively.

Since the orbifold edges, namely $f$ and $g'$, are non-compact, the
edge assignments are multi-regular and so only sign in the formula for
$\underline{DT} (\X _{a,b})$ is the sign $(-1)^{(\tilde {m}+\delta
_{0}+\delta _{\infty }) |\lambda |}$. Thus
\[
\underline{DT} (\X _{a,b}) = \sum _{\lambda } \Esf _{\lambda } \cdot
\Vsf ^{a}_{\lambda \emptyset \emptyset } (p_{0},\dotsc ,p_{a-1})\cdot
\Vsf ^{b}_{ \lambda ' \emptyset \emptyset } (r_{0},\dotsc ,r_{b-1})
\]
where
\[
\Esf _{\lambda } =(-1)^{|\lambda |}\,v^{|\lambda |}\,
q^{|\lambda | } \,
p_{0}^{A_{\lambda } (0,a)}\dotsb p_{a-1}^{A_{\lambda } (a-1,a)} \,
r_{0}^{A_{\lambda' } (0,b)}\dotsb r_{b-1}^{A_{\lambda' } (b-1,b)}.
\]

Applying the formula in Theorem~\ref{thm: formula for Z_n vertex}, we
get
\begin{align*}
\Vsf ^{a}_{ \lambda \emptyset \emptyset } (p)& =\Vsf ^{a}_{\emptyset
\emptyset \emptyset } (p) \cdot p_{0}^{-A_{\lambda } (0,a)}\dotsb p_{a-1}^{-A_{\lambda } (a-1,a)}\cdot \overline{s_{\lambda '}
(\frak{p}_{\bullet })}\\
\Vsf ^{b}_{ \lambda' \emptyset \emptyset } (r)& =\Vsf ^{b}_{\emptyset
\emptyset \emptyset } (r) \cdot r_{0}^{-A_{\lambda' } (0,b)}\dotsb r_{b-1}^{-A_{\lambda' } (b-1,b)}\cdot \overline{s_{\lambda }
(\frak{r}_{\bullet })}
\end{align*}
where 
\begin{align*}
p= (p_{0},\dotsc ,p_{a-1}),& \quad 
\frak{p}_{\bullet } = (1,p_{1},p_{1}p_{2},p_{1}p_{2}p_{3},\dotsc ),\\
r= (r_{0},\dotsc ,r_{b-1}),& \quad 
\frak{r}_{\bullet } = (1,r_{1},r_{1}r_{2},r_{1}r_{2}r_{3},\dotsc ).
\end{align*}

The formula then reads
\[
\underline{DT} (\X _{a,b}) = \Vsf ^{a}_{\emptyset \emptyset \emptyset } (p)\Vsf ^{b}_{\emptyset \emptyset \emptyset } (r)\sum _{\lambda } s_{\lambda '}
(-vq\, \overline{\frak{p}}_{\bullet })\,s_{\lambda}
(\overline{\frak{r}}_{\bullet }).
\]
If we write $Q= (1,q,q^{2},q^{3},\dotsc )$, then we can rewrite the
variables $\frak{p}_{\bullet }$ and $\frak{r}_{\bullet }$ as
\begin{align*}
\frak{p}_{\bullet }&= (Q,p_{1}Q,p_{1}p_{2}Q,\dotsc ,p_{1}\dotsb p_{a-1}Q)\\
\frak{r}_{\bullet }&= (Q,r_{1}Q,r_{1}r_{2}Q,\dotsc ,r_{1}\dotsb r_{b-1}Q)
\end{align*}
and hence
\begin{align*}
\overline{\frak{p}}_{\bullet } &=(Q,p_{a-1}Q,p_{a-1}p_{a-2}Q,\dotsc ,p_{1}\dotsb p_{a-1}Q)\\
\overline{\frak{r}}_{\bullet } &= (Q,r_{b-1}Q,r_{b-1}r_{b-2}Q,\dotsc ,r_{1}\dotsb r_{b-1}Q)
\end{align*}

Using the orthogonality of Schur functions \cite[\S~I.4
(4.3')]{MacDonald} and the fact that
\[
\prod _{i,j} (1+x_{i}y_{i}) = M (w,q)^{-1}
\]
if
\[
(x_{1},x_{2},x_{3},\dotsc )=-wqQ,\quad (y_{1},y_{2},y_{3},\dotsc )=Q,
\]
we get
\[
\underline{DT} (\X _{a,b}) = \Vsf ^{a}_{\emptyset \emptyset \emptyset
} (p)\Vsf ^{b}_{\emptyset \emptyset \emptyset } (r) \prod
_{k=1}^{a}\prod _{l=1}^{b} M (vp_{k}\dotsb p_{a-1}r_{l}\dotsb
r_{b-1},q)^{-1}.
\]
Using the formula for $\Vsf ^{n}_{\emptyset \emptyset \emptyset }$, we
arrive at the following
\begin{prop}\label{prop: DT of local football}
The \DT partition function of the local football $\X
_{a,b}$ is given by
\[
DT (\X _{a,b}) = M (1,-q)^{a+b}\prod _{w\in \mathcal{C}_{a,b}^{+}}M (w,-q)\prod
_{u\in \mathcal{C}_{a,b}^{-}} M (u,-q)^{-1}
\]
where
\begin{align*}
\mathcal{C}_{a,b}^{+}& = \{p_{i}\dotsb p_{j},\,p^{-1}_{i}\dotsb
p^{-1}_{j},\,r_{k}\dotsb r_{l},\,r^{-1}_{k}\dotsb r^{-1}_{l}, 0<i\leq j<a,\,0<k\leq l<b \}\\
\mathcal{C}_{a,b}^{-}& = \{vp_{k}\dotsb p_{a-1}r_{l}\dotsb r_{b-1}: k=1,\dotsc ,a,l=1,\dotsc ,b \}.
\end{align*}
\end{prop}

Since the only stacky curves in $\X _{a,b}$ are non-compact, the
reduced multi-regular DT partition function is equal to the usual
reduced partition function:
\[
DT'_{mr} (\X _{a,b}) = DT' (\X _{a,b}) = \prod _{u\in \mathcal{C}^{-}_{a,b}} M (u,-q)^{-1}.
\]
The \CY  resolution $Y\to X$ has a single $(-1,-1)$ curve given by
the proper transform of the football to which are attached two chains
of $(0,-2)$-curves having $a-1$ and $b-1$ components each. Using the
usual (non-orbifold) vertex formalism, one can verify that as
predicted
\[
\frac{DT (Y)}{DT_{exc} (Y)} = \prod _{u\in \mathcal{C}^{-}_{a,b}} M (u, -q)^{-1}
\]
where on $Y$, the variables $p_{1},\dotsc ,p_{a-1}$ and $r_{1},\dotsc
,r_{b-1}$ correspond to the classes of the curves in each of the
chains and $v$ corresponds to the class of the $(-1,-1)$-curve.

\subsection{Example: The local $B\znums _{2}$ gerbe.}\label{subsec: the local BZ2 gerbe}

Another example related to the conifold is the local $B\znums _{2}$
gerbe. In this case, $\X $ is the global quotient of the resolved
conifold $\Tot (\O (-1)\oplus \O (-1)\to \P ^{1})$ by $\znums _{2}$
acting fiberwise by -1. The graph of $\X $ is again given by the one
in Figure~\ref{fig: edge e with f,f',g,g'} but now with $e$ being the
only orbifold edge. The numerical invariants are 
\[
n (e)=2,\,m=\tilde m=m'=\tilde m'=-1,
\]
and the variables are
\[
q_{0}, \,q_{1},\,v_{0},\,v_{1}
\]
corresponding to the $K$-theory classes 
\[
\O _{p}\otimes \rho _{0},\,\,\,\O _{p}\otimes \rho _{1},\,\,\,\O _{C} (-1)\otimes \rho _{0},\,\,\,\O _{C} (-1)\otimes \rho _{1},\,
\]
where $p=p (e)$ is a point on the curve $C=C (e)$.

The \CY resolution $Y\to X$ is given by local $\P ^{1}\times \P ^{1}$,
namely
\[
Y=\Tot \left(\O  (-2,-2)\to \P ^{1}\times \P  ^{1} \right).
\]

Unlike the local football, there is not a nice closed formula for $DT
(\X )$. However, our vertex formula does provides an explicit formula
for the coefficients of the expansion of $DT (\X )$ as a series in
$v_{0}$ and $v_{1}$. For applications to the DT/GW correspondence and
the DT crepant resolution conjecture, we can restrict ourselves to
curve classes whose generic point has a representation which is a
multiple of the regular representation. This corresponds to expanding
$DT (\X )$ about the variable $v=v_{0}v_{1}$, which in the vertex
formula corresponds to summing over multi-regular edge
assignments. Recall that this series is denoted $DT_{mr} (\X )$.  We
compute with the vertex formula:
\[
\underline{DT}_{mr} (\X )  = \sum _{\begin{smallmatrix} \nu \\
|\nu |_{0}=|\nu |_{1} \end{smallmatrix} }\Esf _{\nu } \left(\Vsf
^{2}_{\emptyset \emptyset \nu } (q_{0},q_{1}) \right)^{2}
\]
where
\[
\Esf _{\nu } = v^{|\nu |_{0}} \cdot q_{0}^{\sum _{i,j\in \nu [0,2]}i+j+1} \cdot q_{1}^{\sum _{i,j\in \nu [1,2]}i+j+1}
\]
and
\[
\Vsf ^{2}_{\emptyset \emptyset \nu } = \Vsf ^{2}_{\emptyset \emptyset \emptyset } \prod _{j,i\in \nu' }\frac{1}{1-q_{0}^{h^{0}_{\nu' } (j,i)}q_{1}^{h^{1}_{\nu' } (j,i)}}
\]
Noting that $\underline{DT}_{0} (\X )= \left(\Vsf ^{2}_{\emptyset
\emptyset \emptyset } \right)^{2}$, we get
\[
\frac{\underline{DT}_{mr} (\X )}{\underline{DT}_{0} (\X )} = \sum
_{d=0}^{\infty }v^{d}\sum _{\begin{smallmatrix} \nu \\
 |\nu |_{0}=|\nu
|_{1}=d\end{smallmatrix}}\frac{q_{0}^{\sum _{i,j\in \nu [0,2]}i+j+1}q_{1}^{\sum
_{i,j\in \nu [1,2]}i+j+1}}{\prod _{j,i\in \nu' }\left(1-q_{0}^{h_{\nu' }^{0} (j,i)} q_{1}^{h_{\nu' }^{1} (j,i)} \right)^{2}}
\]

We expand the above to order 3 in $v$. The linear term corresponds to
the two partitions of size 2 and the quadratic term corresponds to the
5 partitions of size 4. The rational function in the $\nu $ sum is
invariant under $\nu \leftrightarrow \nu '$ and is easily evaluated:

\begin{multline*}
1+\,v\frac{2q_{0}q_{1}^{2}}{(1-q_{0}q_{1})^{2} (1-q_{1})^{2}}\,\\
+\,v^{2}
\Bigg\{
\frac{2q_{0}^{4}q_{1}^{6}}{(1-q_{0}^{2}q_{1}^{2})^{2}(1-q_{0}q_{1}^{2})^{2}(1-q_{0}q_{1})^{2}(1-q_{1})^{2}}+\\
\frac{2q_{0}^{4}q_{1}^{4}}{(1-q_{0}^{2}q_{1}^{2})^{2}(1-q_{0}q_{1})^{2}(1-q_{0})^{2}(1-q_{1})^{2}}
\\
+\frac{q_{0}^{4}q_{1}^{4}}{(1-q_{0}q_{1}^{2})^{2}(1-q_{0}q_{1})^{4}(1-q_{0})^{2}}\Bigg\} + O (v^{3})
\end{multline*}

As predicted by Conjecture~\ref{conj: DT CRC}, 
the above series (after replacing $q_{0}$ with $-q_{0}$) matches with ${DT} (Y)/{DT}_{exc}
(Y)$ under the change of variables
\[
q=q_{0}q_{1},\quad v_{s}=q_{1}v,\quad v_{f}=q_{1}.
\]
Here $v_{s}$ and $v_{f}$ are the variables associated to the
generating curve classes in $\P ^{1}\times \P ^{1}$ (the section and
fiber classes).

We note that it is noticeably more efficient to compute with the
orbifold vertex than to compute on local $\P ^{1}\times \P ^{1}$.

The \GW partition function of $\X $ is obtained from $DT_{mr} (\X
)/DT_{0} (\X )$ by the change of variables
\[
q_{0}q_{1} = -e^{i\lambda },\quad q_{1} = -e^{ix},\quad v=w,
\]
where $\lambda $ is the genus parameter, $w$ is the degree parameter,
and $x$ indexes the number of marked $B\znums _{2}$ points.  So for
example, if $GW_{1,g,n} (\X )$ denotes the \GW invariant of degree 1
maps whose domain curve is genus $g$ with $n$ marked $B\znums _{2}$
points, then
\[
\sum _{n,g}GW_{1,g,n} (\X ) \lambda ^{2g-2}x^{n} = \frac{1}{2}
\left(2\sin \frac{\lambda }{2} \right)^{-2}\sec^{2}\frac{x}{2}.
\]

\section{Proof of Theorem~\ref{thm: main formula for transverse An} }\label{sec: proof of main thm}

\subsection{Overview} Our computation of the \DT
partition function of $\X $ uses a localization technique. The action
of the torus $T$ on $\X$ induces a $T$ action on $\Hilb ^{\alpha } (\X
)$ with isolated fixed points. The fixed points are given by substacks
of $\X $ defined by monomial ideals on each chart and these correspond
to 3D partitions at each vertex. We use a theorem of Behrend and
Fantechi \cite[Theorem~3.4]{Behrend-Fantechi08} which says that the
weighted Euler characteristic of $\Hilb ^{\alpha } (\X )$ is given by
a signed count of fixed points:
\begin{align*}
DT_{\alpha } (\X ) &= e (\Hilb ^{\alpha } (\X ), \nu _{\Hilb ^{\alpha
}} (\X ))\\
&=\sum _{p\in \Hilb ^{\alpha } (\X )^{T}} (-1) ^{\dim T_{p}\Hilb ^{\alpha } (\X )}. 
\end{align*}
The above formula is also apparent from the point of view of virtual
localization as used in \cite{MNOP1}, although we avoid
non-compactness issues by the use of weighted Euler characteristics.

Thus the main two tasks are the following.
\begin{enumerate}
\item A combinatorial description of the $T$-fixed substacks and
the computation of the $K$-theory class of a given $T$-fixed
substack.
\item The computation of the parity of the tangent space to a fixed
point in order to determine the sign.
\end{enumerate}
Our approach to the above two tasks are quite different from
\cite{MNOP1} whose techniques do not readily generalize to the
orbifold case. In fact our approach provides a substantial
simplification in the non-orbifold case over the proof of
\cite{MNOP1}; in particular, we avoid the need for the combinatorial
analysis in \cite[\S~4.11]{MNOP1}.

To handle (1), we find a $K$-theory decomposition of
$T$-invariant substacks into edge and vertex terms, and we use well
chosen functions on $K$-theory to write the class in our basis. This is carried out in \S~\ref{subsec: K-Theory decomp}. 

To handle (2), we exploit $T$-equivariant Serre duality and the Euler
pairing in $K$-theory to determine the vertex and edge contributions
to the signs. This is quite involved and is carried out in
\S~\ref{sec: signs}.

Our techniques yield a vertex formalism for an \emph{arbitrary}
orbifold toric \CYthree $\X $ (not just the transverse $A_{n-1}$
case). Namely, we derive a formula of the form given by
equation~\eqref{eqn: general vertex formula} where $\Esf (e)$ is a
signed monomial depending on $\lambda (e)$ and the local geometry of
$C (e)$ and $\hat{\Vsf }^{G}_{\lambda \mu \nu }$ is the generating
function for $3D$ partitions asymptotic to $(\lambda, \mu, \nu )$,
colored by representations of $G$ (as in equation~\eqref{eqn:
definition of VG}), except counted with the sign rule given in
Theorem~\ref{thm: general sign formula} (see Remark~\ref{rem: general
vertex formula}). Although the formula is completely combinatorial and
can be made explicit, it is not as computational effective as the
formula for the transverse $A_{n-1}$ case because we do not have an
explicit formula for the general orbifold vertex $\hat{\Vsf
}^{G}_{\lambda \mu \nu }$ as we do in the transverse $A_{n-1}$
case. We also expect there to an explicit formula for the vertex for
$G = \znums _{2}\times \znums _{2}$ but not in general.

\subsection{The $K$-theory decomposition} \label{subsec: K-Theory decomp}
\begin{lemma}\label{lem: fixed points in bijection with partitions}
Torus fixed points in
\[
\bigsqcup _{\alpha \in F_{1}K (\X )} \Hilb ^{\alpha } (\X )
\]
are isolated and in bijective correspondence with sets $\{\lambda
(e),\pi (v) \}$ where $\lambda (e)$ is an edge assignment
(Definition~\ref{defn: edge assignment}) and $\{\pi (v) : v\in
\Vertices \}$ is a collection of 3D-partitions such that $\pi (v)$
is asymptotic to $(\lambda _{1} (v),\lambda _{2} (v),\lambda _{3}
(v))$.
\end{lemma}
\begin{proof}
Fix an orientation of $\Gamma $, the graph associated to $\X $.
Recall that $D (e)$ and $D' (e)$ are the invariant divisors incident
to $C (e)$ corresponding to the regions to the left and right of $e$
respectively. Recall that $(D_{1} (e),D_{2} (e),D_{3} (e))$ are the
invariant divisors incident to $p (e)$ corresponding to the regions
opposite of $(e_{1} (v),e_{2} (v),e_{3} (v))$ from $v$. 

Let $Y\subset \X $ be a torus invariant substack of dimension at most
one. We associate to $Y$ a collection $\{\lambda (e),\pi (v) \}$ as
follows. Define $\lambda (e)$ to be the set $(i,j)$ such that the
composition
\[
\O _{\X } (-iD (e)-jD' (e))\to \O _{\X } \to \O _{Y}
\]
is non-zero at a general point of $C (e)$.

Similarly, we define $\pi (v)$ to be the set $(i,j,k)$ such that the
composition
\[
\O _{\X } (-iD_{1} (v)-jD_{2} (v)-kD_{3} (v)) \to \O _{\X }\to \O _{Y}
\]
is non-zero at $p (v)$. The fact that $\pi (v)$ is asymptotic to
$(\lambda _{1} (v),\lambda _{2} (v),\lambda _{3} (v))$ follows easily
from the construction and our conventions.

Conversely, given $\{\lambda (e), \pi (v) \}$, an edge assignment
$\lambda (e)$ and a set $\{\pi (v) \}$ of $3D$-partitions asymptotic
to $(\lambda _{1} (v),\lambda _{2} (v),\lambda _{3} (v))$, we
construct a torus invariant substack $Y\subset \X $ as follows. Note
that the edge assignment is uniquely determined by the $3D$-partitions
$\{\pi (v) \}$. For each $v$, consider the ideal sheaf $I_{\pi
(v)}\subset \O _{\X }$ generated by the image of the maps
\[
\O _{\X } (-iD_{1} (v)-jD _{2} (v) -kD_{3} (v))\to \O _{\X }
\]
for $(i,j,k)$ \emph{not} contained in $\pi (v)$. This determines a
torus invariant substack in the torus open invariant neighborhood of
the point $p (v)$ for each $v$. By the compatibility of the edge
partitions, these substacks agree on the overlaps and thus determine a
global substack.
\end{proof}

\begin{remark}\label{rem: elements of partitions as divisors}
It will be convenient notation to identify an element $(i,j,k)\in \pi
(v)$ with the corresponding divisor. Thus if we write $D\in \pi (v)$
we will mean 
\[
D=iD_{1} (v)+jD_{2} (v)+kD_{3} (v)
\]
for the corresponding $(i,j,k)\in \pi (v)$. Similarly, $D\in \lambda
(e)$ means 
\[
D=iD (e)+jD' (e)
\]
for the corresponding $(i,j)\in \lambda (e)$. Our orientation
conventions guarantee consistency between the divisors associated to
elements of edge partitions and the divisors associated to elements of
the legs of vertex partitions.
\end{remark}

We write the $K$-theory class of the structure sheaf of a torus
invariant substack as a sum over edge and vertex terms:

\begin{prop}\label{prop: K-theory decomposition of O_Y}
Let $Y\subset \X $ be a $T$-invariant substack of dimension no greater
than one. Let $\{\lambda (e),\pi (v) \}$ be the corresponding set of
vertex and edge partitions. Then in $T$-equivariant compactly
supported $K$-theory we have
\[
\O _{Y} =\sum _{e\in \Edges } \sum _{D\in \lambda (e)} \O _{C (e)}
(-D) +\sum _{v\in \Vertices }\sum _{D\in \pi (v)} \xi _{\pi (v)} (D)\O
_{p (v)} (-D).
\]
\end{prop}
\begin{proof}
For any $N\in \nnums $ let $\Nsf$ be the cubical $3D$-partition of
size $N$, that is $\Nsf =\{(i,j,k) : 0\leq i,j,k<N\}$. Let $Z_{N}$ be
the $T$-invariant substack having empty edge partitions and vertex
partitions all equal to $\Nsf $. Let $Y_{N}$ be the stack theoretic
union of $Y$ and $Z_{N}$. Choose $N$ large enough so that for each
$v$, $\pi (v)$ is contained in the the union of the legs of $\pi (v)$
with $\Nsf $. 

We have embeddings $Y\subset Y_{N}$ and $Z_{N}\subset Y_{N}$ from
which we get the following $K$-theory equalities:
\begin{align*}
I_{Y} - I_{Y_{N}} &= \sum _{v\in \Vertices } \sum _{\begin{smallmatrix} D\in \Nsf \\
D\not \in \pi (v) \end{smallmatrix}} \O_{p (v)} (-D)\\
I_{Z_{N}} - I_{Y_{N}} &= \sum _{e\in \Edges } \sum _{D\in \lambda (e)} \O _{C (e)} (-D-ND_{0} (e)-ND_{\infty } (e)).
\end{align*}
For any $D\in \lambda (e)$, we have
\begin{multline*}
\O _{C (e)} (-D) = \O _{C (e)} (-D-ND_{0} (e)-ND_{\infty } (e ))\\
+\sum _{k=0}^{N-1} \O _{p_{0} (e)} (-D-kD_{0} (e)) +\O _{p_{\infty }
(e)} (-D-kD_{\infty } (e)).
\end{multline*}
We note that if $v$ is the initial vertex of $e$, then
\[
\sum _{D\in \lambda (e)} \sum _{k=0}^{N-1} \O _{p_{0} (e)} (-D-kD_{0}
(e)) = \sum _{D\in \Nsf \cap \operatorname{Leg}_{e}\pi (v) } \O _{p
(v)} (-D)
\]
where $\operatorname{Leg}_{e}\pi (v)$ is the leg of $\pi (v)$ in the
$e$ direction. The similar statement holds for $p_{\infty } (e)$.

Putting it all together we get:
\begin{align*}
\O _{Y}=&\,\,  O_{Z_{N}} - (I_{Y}-I_{Y_{N}}) + (I_{Z_{N}}-I_{Y_{N}})\\
=&\sum _{v\in \Vertices } \left(\sum _{D\in \Nsf }\O _{p (v)} (-D) - \sum _{\begin{smallmatrix} D\in \Nsf \\
D\not \in \pi (v) \end{smallmatrix}} \O _{p (v)} (-D) \right)\\
&+\sum _{e\in \Edges }\sum _{D\in \lambda (e)} \O _{C (e)} (-D-ND_{0} (e)-ND_{\infty } (e))\\
=&\sum _{v\in \Vertices } \sum _{D\in \pi (v)\cap \Nsf } \O _{p (v)} (-D) \\
&\!\!\!\!\!\!\!\!\!\!\!\!\!\!\!\!\!\!\!\!\!\!\!\!+\sum _{e\in \Edges } \sum _{D\in \lambda (e)} \left(\O _{C (e)} (-D) -\sum _{k=0}^{N-1} (\O _{p_{0} (e)} (-D-kD_{0} (e))+\O _{p_{\infty } (e)} (-D-kD_{\infty } (e))) \right)\\
=&\sum _{v\in \Vertices }\sum _{D\in \pi (v)} \xi_{\pi (v)} (D)\O _{p
(v)} (-D) +\sum _{e\in \Edges }\sum _{D\in \lambda (e)} \O _{C (e)}
(-D)
\end{align*}
which proves the proposition.
\end{proof}

In the case where $\X $ has transverse $A_{n-1}$ orbifold structure,
we can further refine our $K$-theory decomposition of $\O_{Y}$ into
the basis described in \S~\ref{subsec: generators for K}. In the below
lemmas, we write the decompositions of the vertex and the edge terms.
\begin{lemma}\label{lem: vertex K-theory decomp}
The vertex terms decompose as follows
\[
\sum _{D\in \pi (v)}\O _{p (v)} (-D) 
= \begin{cases}
\sum _{i,j,k\in \pi (v)} [\O _{p (v)}\otimes \rho _{i-j}] &\text{if $e (v)$ is oriented outward,}\\
\sum _{i,j,k\in \pi (v)} [\O _{p (v)}\otimes \rho _{j-i}] &\text{if $e (v)$ is oriented inward.}
\end{cases}
\]
\end{lemma}
\begin{proof}
This follows immediately from our conventions (\S~\ref{subsec:
conventions for An-1 toric orbifolds}) and our choice of the indexing
of the representations $\rho _{k}$ of $\znums _{n (e (v))}$
(\S~\ref{subsec: generators for K}).
\end{proof}

\begin{prop}\label{prop: K-theory decomp of an edge}
Let $e$ be a compact edge corresponding to a curve $C=C (e)$ and let
$\lambda =\lambda (e)$ be an edge partition. Let $D=D (e)$, $D'=D'
(e)$ and let $m=\deg (D)$, $m'=\deg (D')$. Assume that $e$ and its
incident edges $f$, $f'$, $g$, $g'$ are oriented as in
figure~\ref{fig: edge e with f,f',g,g'}. Let $n=n (e)$, $a=n (f)$,
$a'=n (f')$, $b=n (g)$, and $b'=n (g')$, then
\begin{align*}
\sum _{i,j\in \lambda }\O _{C} (-iD-jD') =&\quad 
\sum _{k=0}^{n-1}\,\,|\lambda |_{k}\cdot [\O _{C} (-1)\otimes \rho _{k}]\\
&+\sum _{k=0}^{n-1}C^{\lambda }_{\tilde {m},\tilde {m}'}[k,n]\cdot [\O _{p (e)}\otimes \rho _{k}]\\
&+\delta _{0}\sum _{k=0}^{a-1}A_{\lambda\,\, } (k,a\,\,)\cdot [\O _{p (f\,\,)}\otimes \rho
_{k}]\\
&+\delta _{0}'\sum _{k=0}^{a'-1}A_{\lambda' } (k,a')\cdot [\O _{p (f')}\otimes
\rho _{k}]\\
&+\delta _{\infty }\sum _{k=0}^{b-1}A_{\lambda\,\, } (k,b\,\,)\cdot [\O _{p (g\,\,)}\otimes \rho
_{k}]\\
&+\delta _{\infty }'\sum _{k=0}^{b'-1}A_{\lambda' } (k,b')\cdot [\O _{p (g')}\otimes
\rho _{k}].
\end{align*}
\end{prop}

Since $\O _{C} (-iD-jD') $ is supported on $C$, it must be a
combination of the classes $[\O _{C} (-1)\otimes \rho _{k}]$
$k=0,\dotsc ,n-1$, $[\O _{p (edge)}\otimes \rho _{k}]$ $k=0\dotsb ,n
(edge)-1$ for $edge\in \{e,f,f',g,g' \}$, and $[\O _{p}]$ since the
remaining generators are always supported away from $C$. The classes
$[\O _{p (edge)}\otimes \rho _{0}]$ can be written in terms of the
other classes using the relation~\eqref{eqn: O_p=O_p(e)otimes
Reg}. There are no further relations and hence the decomposition of
$\O _{C} (-iD-jD')$ into the above classes (without $[\O _{p
(edge)}\otimes \rho _{0}]$) has unique coefficients. We first compute
the coefficients of that decomposition and then restore the classes
with $\rho _{0}$ via the relation~\eqref{eqn: O_p=O_p(e)otimes
Reg}. Let $\mathcal{B}$ be the set of such classes:
\[
\mathcal{B}
 = \left\{[\O _{p}], [\O _{C} (-1)\otimes \rho _{k}]_{k=0,\dotsc ,n-1}, [\O _{p (edge)}\otimes \rho _{k}]_{k=1,\dotsc ,n (edge), edge\in \{e,f,f',g,g' \}} \right\}.
\]

The coefficient of $[\O _{C} (-1)\otimes \rho _{k}]$ in $\sum _{i,j\in
\lambda }\O _{C} (-iD-jD')$ is clearly $|\lambda |_{k}$ since each
summand acts with weight $i-j\mod n$ at the generic point.

To determine the other coefficients, we construct functions on
$K$-theory vanish on all the elements of $\mathcal{B}$ except one. For
example, the holomorphic Euler characteristic $\chi $ which vanishes
on all the classes in $\mathcal{B}$ except $[\O _{p}]$ on which it is
one.

We first suppose that $n>1$. Then we have that $f=f'=g=g'=1$ and the
only point classes are $[\O _{p (e)}\otimes \rho _{k}]$ for $k=1\dotsb
n-1$ and $[\O _{p}]$. We define a function $\alpha _{k}$ on $K$-theory
as follows. Let $C_{l}\subset IC$ be the component of the inertia
stack corresponding to $l\in \znums _{n}$. Let $\omega
=\exp\left(\frac{2\pi i}{n} \right)$ and let $\tau $ be the Toen
operator (see Appendix~\ref{app: GRR and Toen operator.}). We define
\[
\alpha _{k } (E) = \sum _{l=0}^{n-1}\int _{C_{l}} \left(\omega
^{-lk}-1 \right) \tau (E).
\]
\begin{lemma}
The function $\alpha _{k}$ equals 0 on all classes of $\mathcal{B}$
except for $[\O _{p (e)}\otimes \rho _{k}]$ on which it is 1.
\end{lemma}
\begin{proof}
Recall that by definition $[\O _{C} (-1)\otimes \rho
_{l}]=\frac{1}{n}\pi _{*}[\O _{\tilde{C}} (-1)\otimes \rho _{l}]$
where $\pi :\tilde{C}\to C$ is an $n$-fold cover with $\tilde{C}\cong
\P ^{1}\times B\znums _{n}$. By the functorial properties of the Toen
operator (Theorem~\ref{thm: Toen's GRR} in Appendix~\ref{app: GRR and
Toen operator.}), we have
\[
\tau ([\O _{C} (-1)\otimes \rho _{l}]) = \frac{1}{n}\pi _{*}\tau [\O
_{\tilde{C}} (-1)\otimes \rho _{l}].
\]
However, since $\tau (\O _{\tilde{C}} (-1))$ has vanishing $H^{2}$
terms on each component of $I\tilde{C}$, all the integrals in $\alpha
_{k} ([\O _{C} (-1)\otimes \rho _{l}])$ are zero. For the point
classes, we compute (using example~\ref{exmpl: toen operator for BZn} from Appendix~\ref{app: GRR and Toen operator.})
\begin{align*}
\alpha _{k} ([\O _{p (e)}\otimes \rho _{j}]) &= \sum _{l=0}^{n-1} \int _{C_{l}} (\omega ^{-lk}-1)\tau (\O _{p (e)}\otimes \rho _{j})\\
&=\sum _{l=0}^{n-1} (\omega ^{-lk}-1)\int _{C_{l}}\omega ^{lj}[p (e)]\\
&=\frac{1}{n} \sum _{l=0}^{n-1} \omega ^{l (k-j)} - \omega ^{lj}\\
&=\delta _{k-j,0}-\delta _{j,0}.
\end{align*}
Note that 
\[
\alpha _{k} (\O _{p})=\alpha _{k} (\O _{p (e)}\otimes R_{reg}) = \sum
_{j=0}^{n-1} \delta _{k-j,0}-\delta _{j,0} = 0
\]
and the lemma is proved.
\end{proof}
By the above lemma, the coefficient of $\O _{p (e)}\otimes \rho _{k}$
in $\O _{C} (-iD-jD')$ is given by $\alpha _{k} (\O _{C} (-iD-jD'))$.
Using example~\ref{exmpl: toen operator for BZn gerbe over a curve} in
Appendix~\ref{app: GRR and Toen operator.} we can compute as follows
\begin{align*}
\alpha _{k} \left(\O _{C} (-iD-jD') \right) 
&=\sum _{l=0}^{n-1} \int _{C_{l}} \left(\omega ^{-lk}-1 \right) \tau (\O _{C} (-iD-jD'))\\
&=\sum _{l=0}^{n-1} (\omega ^{-lk}-1) \int _{C_{l}}\omega ^{l (i-j)} \left(1+[p (e)] (1-im-jm') \right)\\
&= (1-im-jm') \frac{1}{n} \sum _{l=0}^{n-1} \omega ^{l (i-j-k)}-\omega
^{l (i-j)}\\
&= (1-i\tilde m-j\tilde m') (\delta _{i-j,k}-\delta _{i-j,0}).
\end{align*}
Therefore
\[
\coeff _{[\O _{p (e)}\otimes \rho _{k}]}\left(\sum _{i,j\in \lambda }
\O _{C} (-iD-jD') \right) = C^{\lambda }_{\tilde m,\tilde m'}[k,n] -
C^{\lambda }_{\tilde m,\tilde m'}[0,n].
\]
We also have 
\begin{align*}
\coeff _{[\O _{p}]}\left( \sum _{i,j\in \lambda }
\O _{C} (-iD-jD') \right) &= \chi \left(\sum _{i,j\in \lambda }
\O _{C} (-iD-jD')  \right)\\
&=\sum _{i,j\in \lambda }\sum _{l=0}^{n-1} \int _{C_{l}}\tau (\O _{C} (-iD-jD'))\\
&=\sum _{i,j\in \lambda }\frac{1}{n}\sum _{l=0}^{n-1} \omega ^{l (i-j)} (1-im-jm')\\
&=\sum _{i,j\in \lambda} (1-im-jm')\delta _{i-j,0}\\
&=C^{\lambda }_{\tilde m,\tilde m'}[0,n]
\end{align*}
Using the relation 
\[
[\O _{p}]=\sum _{k=0}^{n-1} [\O _{p (e)}\otimes \rho _{k}],
\]
we find that 
\begin{align*}
\sum _{i,j\in \lambda }[\O _{C} (-iD-jD')] =& 
\quad \sum_{ k=0} ^{n-1} |\lambda |_{k}\cdot [\O _{C} (-1)\otimes \rho
_{k}]\\
&+\sum _{k=0}^{n-1} C^{\lambda }_{\tilde m,\tilde m'}[k,n]\cdot [\O _{p (e)}\otimes \rho _{k}]
\end{align*}
and Proposition~\ref{prop: K-theory decomp of an edge} is proved for
the case of $n>1$.

We now assume that $n=1$. Recall the definitions of $\delta
_{0},\delta _{0}',\delta _{\infty },\delta _{\infty }',\tilde m$, and
$\tilde m'$ from \S~\ref{subsec: the vertex formula}.

The holomorphic Euler characteristic of a general line bundle on the
football $C$ is given in example~\ref{exmpl: GRR for the football} in
Appendix~\ref{app: GRR and Toen operator.}:
\[
\chi \left(\O _{C} (dp+sp_{0}+tp_{\infty }) \right) = d+1
+\left\lfloor \frac{s}{\max (a,a')} \right\rfloor+\left\lfloor
\frac{t}{\max (b,b')} \right\rfloor
\]
Thus
\begin{align*}
\chi (\O _{C} (-iD-jD')) &= \chi (\O _{C} ((-i\tilde{m}-j\tilde{m}')p
+ (i\delta _{0}+j\delta _{0}')p_{0}+ (i\delta _{\infty }+j\delta
_{\infty }')p_{\infty }))\\
&=-i\tilde{m} -j\tilde{m}'+1+\left\lfloor \frac{i\delta _{0}+j\delta _{0}'}{\max (a,a')}  \right\rfloor +\left\lfloor \frac{i\delta _{\infty }+j\delta _{\infty }'}{\max (b,b')}  \right\rfloor\\
&=-i\tilde m-j\tilde m'+1+\left\lfloor \frac{i\delta _{0}}{a}  \right\rfloor 
+\left\lfloor \frac{j\delta' _{0}}{a'}  \right\rfloor 
+\left\lfloor \frac{i\delta _{\infty}}{b} \right\rfloor
+\left\lfloor \frac{j\delta' _{\infty}}{b'} \right\rfloor
\end{align*}
where in the last equality we used the fact that either $\delta _{0}$
or $\delta _{0}'$ is zero and that either $\delta _{\infty }$ or
$\delta _{\infty }'$ is zero. 

We conclude that 
\begin{align*}
\coeff _{[\O _{p}]}\left(\sum _{i,j\in \lambda }\O _{C} (-iD-jD') \right) = &C^{\lambda }_{\tilde m,\tilde m'} +
\delta _{0} A_{\lambda } (0,a) +
\delta' _{0} A_{\lambda' } (0,a') \\
&
\quad \quad +\delta _{\infty } A_{\lambda } (0,b) +
\delta' _{\infty } A_{\lambda' } (0,b') 
\end{align*}
For $k=1,\dotsc ,\max (a,a')-1$ we define
\[
\mu _{k} (E) = \chi (E (kD_{0})) - \chi (E).
\]
For $k=1,\dotsc ,\max (b,b')-1$ we define
\[
\nu _{k} (E) = \chi (E (kD_{\infty })) - \chi (E).
\]
\begin{lemma}
The function $\mu _{k}$ is zero on all the classes in $\mathcal{B}$
except for $\delta _{0}[\O _{p (f)}\otimes \rho _{k}]+\delta _{0}'[\O
_{p (f')}\otimes \rho _{k}]$ on which it is 1. Likewise, the function
$\nu _{k}$ is zero on all the classes in $\mathcal{B}$ except for
$\delta _{\infty }[\O _{p (g)}\otimes \rho _{k}]+\delta _{\infty }'[\O
_{p (g')}\otimes \rho _{k}]$ on which it is 1.
\end{lemma}
\begin{proof}
Since $\O _{C} (D_{0})=\O _{C} (p_{0})$, we have 
\begin{align*}
\mu _{k} (\O _{C} (-1)) &= \chi (\O _{C} (-p+kp_{0}))  -\chi (\O _{C} (-p))\\
&=\left\lfloor \frac{k}{\max (a,a')}  \right\rfloor = 0.
\end{align*}
By our orientation conventions, the weight of the action of $\O
(kD_{0})$ on $\O _{p (f)}$ and $\O _{p (f')}$ is $-k$. Then for
$k,l\in \{1,\dotsc ,a-1 \}$ 
\begin{align*}
\mu _{k} (\O _{p (f)}\otimes \rho _{l}) &= \chi (\O _{p (f)}\otimes \rho _{l-k}) - \chi (\O _{p (f)}\otimes \rho _{k})\\
&=\delta _{l,k}
\end{align*}
and similarly for $k,l\in \{1,\dotsc ,a'-1 \}$ we have 
\[
\mu _{k} (\O _{p (f')}\otimes \rho _{l}) = \delta _{l,k}. 
\]
Finally, $\mu _{k}$ vanishes on $[\O _{p}]$, $[\O _{p (g)}\otimes \rho
_{l}]$, and $[\O _{p (g')}\otimes \rho _{l}]$ since these classes can
be taken with support disjoint from $D_{0}$. This proves the
assertions of the lemma for $\mu _{k}$; the proof for $\nu _{k}$ is
similar.
\end{proof}
By the lemma, we can use $\mu _{k}$ and $\nu _{k}$ to determine the
remaining coefficients of $\sum _{i,j\in \lambda }\O _{C} (-iD-jD')$
in the basis $\mathcal{B}$.
\begin{align*}
\mu _{k} (\O _{C} (-iD-jD')) = &\chi (\O _{C} ((-i\tilde {m}-j\tilde {m}')p + (i\delta _{0}+j\delta _{0}'+k)p_{0} + (i\delta _{\infty }+j\delta _{\infty }')p_{\infty }))\\
&+\chi (\O _{C} ((-i\tilde {m}-j\tilde {m}')p + (i\delta _{0}+j\delta _{0}')p_{0} + (i\delta _{\infty }+j\delta _{\infty }')p_{\infty }))\\
=&\left\lfloor \frac{i\delta _{0}+j\delta _{0}'+k}{\max (a,a')}  \right\rfloor -  \left\lfloor \frac{i\delta _{0}+j\delta _{0}'}{\max (a,a')}  \right\rfloor \\
=& \delta _{0}\left( \left\lfloor \frac{i+k}{a}  \right\rfloor-\left\lfloor \frac{i}{a}  \right\rfloor \right) +
\delta' _{0}\left( \left\lfloor \frac{j+k}{a'}  \right\rfloor-\left\lfloor \frac{j}{a'}  \right\rfloor \right)
\end{align*}
where in the last equality we use the fact that at least one of
$\delta _{0},\delta _{0}'$ is zero. Computing similarly, we get that
\[
\nu _{k} (\O _{C} (-iD-jD')) = \delta _{\infty }\left( \left\lfloor
\frac{i+k}{b} \right\rfloor-\left\lfloor \frac{i}{b} \right\rfloor
\right) + \delta' _{\infty }\left( \left\lfloor \frac{j+k}{b'}
\right\rfloor-\left\lfloor \frac{j}{b'} \right\rfloor \right).
\]

Putting together the computations, we obtain
\begin{align*}
\sum _{i,j\in \lambda }\O _{C} (-iD-jD') =&\quad |\lambda |\cdot [\O _{C} (-1)] \\
&+\Bigg(C^{\lambda }_{\tilde m,\tilde m'}+\delta _{0}A_{\lambda } (0,a) +\delta _{0}' A_{\lambda '}(0,a') \\
&\quad \quad\quad \quad\,\,\,+ \delta _{\infty }A_{\lambda } (0,b) +\delta _{\infty}'A_{\lambda '} (0,b')\Bigg) [\O _{p}]\\
&+\sum _{k=1}^{a-1} \left(A_{\lambda\,\, } (k,a\,\,)-A_{\lambda\,\, } (0,a\,\,) \right) \cdot [\O _{p (f\,\,)}\otimes \rho _{k}] \\
&+\sum _{k=1}^{a'-1} \left(A_{\lambda' } (k,a')-A_{\lambda' } (0,a') \right) \cdot [\O _{p (f')}\otimes \rho _{k}] \\
&+\sum _{k=1}^{b-1} \left(A_{\lambda\,\, } (k,b\,\,)-A_{\lambda\,\, } (0,b\,\,) \right) \cdot [\O _{p (g\,\,)}\otimes \rho _{k}] \\
&+\sum _{k=1}^{b'-1} \left(A_{\lambda' } (k,b')-A_{\lambda' } (0,b') \right) \cdot [\O _{p (g')}\otimes \rho _{k}] .
\end{align*}
Note that we can multiply the $f$ (respectively $f'$, $g$, $g'$) sum
by $\delta _{0}$ (respectively $\delta _{0}'$, $\delta _{\infty }$,
$\delta _{\infty }'$) without changing the equality. Thus applying the
relation~\eqref{eqn: O_p=O_p(e)otimes Reg}, we get
\begin{align*}
\sum _{i,j\in \lambda } \O _{C} (-iD-jD')=& \quad |\lambda |\cdot [\O _{C} (-1)]\\
& +C^{\lambda }_{\tilde m,\tilde m'}\cdot [\O _{p (e)}\otimes \rho _{0}]\\
&+\delta _{0}\sum _{k=1}^{a-1} A_{\lambda\,\, } (k,a\,\,)\cdot [\O _{p (f\,\,)}\otimes \rho _{k}]\\
&+\delta _{0}'\sum _{k=1}^{a'-1} A_{\lambda' } (k,a')\cdot [\O _{p (f')}\otimes \rho _{k}]\\
&+\delta _{\infty }\sum _{k=1}^{b-1} A_{\lambda\,\, } (k,b\,\,)\cdot [\O _{p (g\,\,)}\otimes \rho _{k}]\\
&+\delta _{\infty }'\sum _{k=1}^{b'-1} A_{\lambda' } (k,b')\cdot [\O _{p (g')}\otimes \rho _{k}]
\end{align*}
which proves Proposition~\ref{prop: K-theory decomposition of O_Y} in
the case where $n=1$ and hence completes its proof.

\section{The Sign Formula}\label{sec: signs}

\bigskip

\begin{verse}
Sign, sign, everywhere a sign \\
Blocking out the scenery, breaking my mind\\
\end{verse}
\smallskip
\begin{flushright}
{\emph{---Five Man Electrical Band}}
\end{flushright}

\subsection{Overview}

By \cite[Theorem~3.4]{Behrend-Fantechi08} and Lemma~\ref{lem: fixed
points in bijection with partitions}, the invariant $DT_{\alpha } (\X
)$ is given by a signed count of torus invariant ideal sheaves $I$
where the sign is given by $(-1)^{\Ext ^{1}_{0} (I,I)}.$ This section
is devoted to computing those signs and arranging them into vertex and
edge terms. In \S~\ref{subsec: general sign formula} we derive a
general sign formula, theorem~\ref{thm: general sign formula}, and in
\S~\ref{subsec: sign formula in A_n case}, we compute the sign formula
in the case where $\X $ has transverse $A_{n-1}$ orbifold structure.

\subsection{General Sign Formula}\label{subsec: general sign formula}
Let $I\subseteq\O_{\X }$ be the ideal sheaf of $Y$.  The Zariski
tangent space to $Y$ in $\Hilb(\X)$ is isomorphic to
$\Ext^1_0(I,I)$. We want to compute its dimension modulo $2$ in terms
of the associated partitions $\{\lambda (e) \}$ and $\{\pi (v) \}$.
Let $T$ be the 3-dimensional torus acting on $\X $.

For a $T$-representation $V$, we use $V^{\vee}$ to denote the dual
representation.  By equivariant Serre duality, we
have 
\[
\Ext^i(\cF,\cG)^\vee = \Ext^{3-i}(\cG,\cF\otimes\omega_{\X }),
\]
and likewise for traceless Ext.  If $w\in \Hom(T,\cnums^*)$, we use
the notation $\cnums[w]$ to denote a 1-dimensional $T$-representation
with weight $w$.
\begin{lemma}
As a $T$-equivariant line bundle, $\omega_{\X }\cong\O_{\X
}\otimes_{\cnums}\cnums[\mu]$ for some primitive weight $\mu$.
\end{lemma}
\begin{proof}
The \CY condition on $\X $ implies that $\omega _{\X }$ must be
an equivariant lift of $\O _{\X }$ and hence it is of the form $\O
_{\X } \otimes \cnums [\mu ]$. If $\mu $ is not primitive, then
the generic stabilizer of $\X $ is non-trivial.
\end{proof}

\begin{defn}\label{def:shifted_dual}
We define the shifted dual of a $T$-representation $V$ by the
formula 
\[
V^* = V^\vee\otimes\cnums[-\mu].
\]
\end{defn}
Note that the shifted dual induces a fixed-point free involution on
characters of $T$.

\begin{prop}\label{prop:shifted_dual_properties}
The shifted dual satisfies the following properties.
\begin{enumerate}
\item For any $T$-equivariant sheaves $\cF$ and $\cG$, 
\[
\Ext^i(\cF,\cG)^* \cong \Ext^{3-i}(\cG,\cF).
\]
\item Let $V$ and $W$ be virtual $T$-representations such that
\[
V-V^*=W-W^*.
\]
Then the virtual dimensions of $V$ and $W$ are equal modulo 2.
\end{enumerate}
\end{prop}
\begin{proof}
The first statement is a restatement of equivariant Serre duality.  The second statement follows by comparing the dimensions of the $\nu$ and $-\nu-\mu$ weight spaces of $V$ and $W$ as $\nu$ runs through half the characters of $T$.
\end{proof}

\begin{defn}
Let $V$ be a virtual $T$-representation.  We define $s(V)\in\znums/2\znums$ to be the dimension modulo $2$ of $V$.  We also define $\sigma(V-V^*)=s(V)$, where the input of $\sigma$ is required to be an anti-self shifted dual virtual representation.  $\sigma$ is well-defined by Proposition \ref{prop:shifted_dual_properties}.
\end{defn}

Considered as $T$-representations, we have that 
\[
\Ext^1_0(I,I) - \Ext^2_0(I,I) = \chi(\O_{\X},\O_{\X})-\chi(I,I).
\]

Using the exact sequence $$0\to I\to \O_{\X}\to \O_Y\to 0,$$ we can
write $$\chi(\O_{\X},\O_{\X})-\chi(I,I) = \chi(\O_{\X},\O_Y) +
\chi(\O_Y,\O_{\X}) - \chi(\O_Y,\O_Y).$$ Since $\chi(\O_{\X},\O_Y)^* =
-\chi(\O_Y,\O_{\X})$, we have $$s(\Ext^1_0(I,I)) =
s(\chi(\O_{\X},\O_Y)) + \sigma(\chi(\O_Y,\O_Y)).$$ The first term is
$\chi(\O_Y)$ modulo 2, so we are left to compute the second term.  For
this we use the $K$-theory decomposition above.

Given any decomposition $\O_Y = \sum_i K_i$ in $K_T({\X})$, we have
\begin{align*}
\chi(\O_Y,\O_Y) &= \sum_{i,j}\chi(K_i,K_j) \\
&= \sum_i [(\Ext^0(K_i,K_i) - \Ext^1(K_i,K_i)) - (\Ext^0(K_i,K_i)-\Ext^1(K_i,K_i))^*] \\
&+ \sum_{i<j} [\chi(K_i,K_j) - \chi(K_i,K_j)^*],
\end{align*}
and therefore $$\sigma(\chi(\O_Y,\O_Y)) = \sum_i s(\Hom(K_i,K_i) - \Ext^1(K_i,K_i)) + \sum_{i<j} s(\chi(K_i,K_j)).$$

We treat the first sum first, and call these the diagonal terms.  It can be divided into edge terms and vertex terms.

\begin{prop}
If $K$ and $L$ are supported on curves, then $$\Ext^1(K,L)\cong H^0(\cExt^1(K,L))\oplus H^1(\cHom(K,L)).$$
\end{prop}

\begin{proof}
The local-to-global spectral sequence degenerates at the $E_2$ term.
\end{proof}

First we consider edge terms. Let $e$ be a compact edge and let $C=C
(e)$, $D=D (e)$, and $D'=D' (e)$ so that $C=D\cap D'$. For $A\in
\lambda (e)$ (recall Remark~\ref{rem: elements of partitions as
divisors}) we have

\begin{multline}\label{edge_term_resolution}
0 \to \O_{\X}(-A-D-D')\to\\
 \O_{\X}(-A-D)\oplus\O_{\X}(-A-D')\to \O_{\X}(-A)\to \O_C(-A)\to 0.
\end{multline}
If we apply the functor $\cHom(\cdot, \O_C(-A))$ to this we obtain a
complex which computes the local Ext sheaves.
\begin{enumerate}
\item $\cHom(\O_C(-A),\O_C(-A)) = \O_C$
\item $\cExt^1(\O_C(-A),\O_C(-A)) = N_{C/{\X}}$
\item $\cExt^2(\O_C(-A),\O_C(-A)) = \wedge^2N_{C/{\X}}$
\end{enumerate}
Since $h^0(\O_C)=1$ and $h^{1} (\O _{C})=0$ we deduce that each edge
$e$ contributes
\[
|\lambda (e) |(1+h^0(N_{C/{\X}}))
\]
to the diagonal terms.

We compute the vertex terms as follows. Let $v$ be a vertex and let
$p=p (v)$ and $D_{i}=D_{i} (v)$. For $A\in\pi (v)$, we have the
following exact sequence.

\begin{gather}
\label{vertex_term_resolution}
0 \to \O_{\X}(-A-\sum_i D_i)\to \bigoplus_{1\le i<j\le 3}\O_{\X}(-A-D_i-D_j) \to \\ \bigoplus_{1\le i\le 3}\O_{\X}(-A-D_i)\to \O_{\X}(-A)\to \O_p(-A)\to 0.\nonumber
\end{gather}
By a similar computation to the edge case, we see that every vertex
$v$ contributes
\[
|\pi (v) |(1+h^0(N_{p/{\X}}))
\]
to the diagonal terms.  Note that $|\pi (v) |$ is not the cardinality of $\pi (v)$, but $\sum_{A\in\pi (v)}\xi_{\pi }(A)$.

Finally, we must compute the off-diagonal terms $s(\chi(K_i,K_j))$.
These can be divided into edge terms, where $K_i$ and $K_j$ are
supported on the same edge, and vertex terms, which come in three
types:
\begin{enumerate}
\item $K_i$ and $K_j$ are supported at the same $p=p (v)$.
\item $K_j$ is supported at $p=p (v)$ and $K_i$ is supported along
$C=C (e)$ where $e$ is incident to $v$.
\item $K_i$ is supported on $C=C (e)$ and $K_j$ is supported on $C'=C
(e')$, where $e\neq e'$ have the vertex $v$ in common.
\end{enumerate}

It is convenient to introduce an arbitrary total order on each
partition $\lambda (e)$ and $\pi (v)$.  For each $A<B$ in $\lambda
(e)$, if we apply $\cHom(\cdot, \O_C(-B))$ to
(\ref{edge_term_resolution}), we obtain the complex which computes the
local Ext sheaves:
\[
\cExt  ^{i} (\O _{C} (-A),\O _{C} (-B)) = \O _{C} (A-B)\otimes \wedge
^{i}N_{C/\X }.
\]
It follows that each edge $C\in E$ contributes 
\[
\sum_{\substack{A,B\in\lambda (e) \\ A<B}}
\chi(\O_C(A-B)\otimes\lambda_{-1}(N_{C/{\X}}))
\]
to the off-diagonal terms of $\sigma(\chi(\O_Y,\O_Y))$.

For each $A<B$ in $\pi (v)$, we can apply the same argument to
(\ref{vertex_term_resolution}) to obtain a contribution
of 
\[
\sum_{\substack{A,B\in\pi (v) \\ A<B}} \xi_{\pi (v)}(A)\xi_{\pi
(v)}(B) \chi(\O_p(A-B)\otimes\lambda_{-1}(N_{p/{\X}}))
\]
to the type 1 terms.

If $v$ is incident to $e$, $A\in\lambda (e)$, and $B\in\pi (v)$, then
applying $\cHom(\cdot,\O_p(-B))$ to (\ref{edge_term_resolution})
produces
\[
0 \to \O_p(A-B) \to \O_p(A-B)\otimes N_{C/{\X}} \to \O_p(A-B)\otimes\wedge^2N_{C/{\X}} \to 0,
\]
which yields a type 2 vertex contribution at $v$ of 
\[
\sum_{i=1}^{3} \sum_{A\in\lambda (e_{i})}\sum_{B\in\pi (v)} \xi_{\pi (v)}(B) \chi(\O_p(A-B)\otimes\lambda_{-1}(N_{C (e_{i}) /{\X}})).
\]

Finally, suppose $C=C (e) =D\cap D'$, $C'=C (f') =D\cap D_0$, and $p=p
(v) =C\cap C'$ (see figure~\ref{fig: edge e with f,f',g,g'}). Let
$A\in\lambda (e)$, and $B\in\lambda (f')$.  If we apply
$\cHom(\cdot,\O_{C'}(B))$ to (\ref{edge_term_resolution}), we
obtain the complex
\[
0\to\O_{C'}(A-B) \to
\O_{C'}(A-B+D)\oplus\O_{C'}(A-B+D')\to\O_{C'}(A-B+D+D')\to 0.
\]
Using the fact that $\O_{C'}\to\O_{C'}(D')$ is injective, we compute
the cohomology of the above complex to obtain
\begin{enumerate}
\item $\cHom(\O_C(A),\O_{C'}(B))=0$,
\item $\cExt^1(\O_C(A),\O_{C'}(B))=\O_p(A-B+D')$
\item $\cExt^2(\O_C(A),\O_{C'}(B))=\O_p(A-B+D+D')$
\end{enumerate}
Note that $\O_p(D')=N_{p/C'}$ and $\O_p(D+D')=N_{p/C}^{-1}$ since by
the \CY condition, $\O_p(D+D'+D_{0})=\O_p$. Therefore
\begin{align*}
s (\chi (\O _{C} (A),\O _{C'} (B))) &= h^{0} (\O _{p} (A-B )\otimes
N_{p/C'}) + h^{0} (\O _{p} (A-B )\otimes N^{\vee }_{p/C})\\
&= h^{0} (\O _{p} (A-B )\otimes N_{p/C'}) +h^{0} (\O _{p} (B-A)\otimes
N_{p/C}).
\end{align*}  
Now summing up over all contributions of this type we can write
the type 3 off-diagonal vertex contribution of a vertex $v$ as
\[
\sum_{i\neq j}\sum_{A\in\lambda (e_{i} (v))}\sum_{B\in\lambda (e_{j}
(v))} h^0(\O_p (v) (A-B)\otimes N_{p (v) /C (e_{j} (v))}).
\]
\smallskip

Putting it all together yields the following sign formula.
\begin{align}
\label{eq:sign_formula}
s(\Ext^1_0(I,I)) &= \chi(\O_Y) + \sum_{e\in \Edges } |\lambda (e) |(1+ h^0(N_{C (e) /{\X}})) + \sum_{v\in \Vertices} |\pi (v) |(1+h^0(N_{p (v) /{\X}})) \\
&+ \sum_{e\in \Edges } \sum_{\substack{A,B\in\lambda (e) \\ A<B}} \chi(\O_{C (e)} (A-B)\otimes\lambda_{-1}(N_{C (e) /{\X}})) \nonumber \\
&+ \sum_{v\in \Vertices } \sum_{\substack{A,B\in\pi (v) \\ A<B}} \xi_{\pi (v)}(A)\xi_{\pi (v)}(B) h^0(\O_{p (v)} (A-B)\otimes\lambda_{-1}(N_{p (v) /{\X}})) \nonumber \\
&+ \sum_{v\in \Vertices }\sum _{i=1}^{3} \sum_{A\in\lambda (e_{i} (v))}\sum_{B\in\pi (v)} \xi_{\pi (v)}(B) h^0(\O_{p (v)} (A-B)\otimes\lambda_{-1}(N_{C (e_{i}) /{\X}})) \nonumber \\
&+ \sum _{v\in \Vertices }\sum_{i\neq j}\sum_{A\in\lambda (e_{i} (v))}\sum_{B\in\lambda (e_{j} (v))}
 h^0(\O_{p (v)}(A-B)\otimes N_{p (v) /C (e_{j} (v))}) \nonumber
\end{align}

The above formula can be divided into three pieces. The first is an
overall $\chi (\O _{Y})$, the second is a sum over edges and the third
is a sum over vertices. The contribution of an edge $e$ is 
\[
|\lambda (e)| (1+h^{0} (N_{C (e)/\X })) + \sum _{\begin{smallmatrix} A,B\in \lambda (e)\\
A<B \end{smallmatrix}} \chi (\O _{C (e)} (A-B)\otimes \lambda _{-1}
(N_{C (e)/\X })).
\]
Recall that $<$ was an arbitrary total order. We can resymmetrize as follows. Let $C=C (e)$, $D=D (e)$, and $D'=D' (e)$. We have that 
\begin{align*}
N_{C/\X }&=\O _{C} (D)+\O _{C} (D'),\\
\lambda _{-1}N_{C/\X } &= \O _{C}-\O _{C} (D)-\O _{C} (D')+K_{C}.
\end{align*}
By Serre duality
\begin{align*}
\chi (\O _{C} (A-B)) &= -\chi (\O _{C} (B-A)\otimes K_{C}),\\
\chi (\O _{C} (A-B+D)) &= -\chi (\O _{C} (B-A+D')), \\
h^{0} (\O _{C} (D')) &=h^{1} (\O _{C} (D)).
\end{align*}
This allows the second half of the edge contribution to be rewritten
as a sum over all pairs $(A,B)$, where the diagonal terms are
accounted for by the first half. So the edge contribution is given by
\[
\sum _{A,B\in \lambda (e)}\chi (\O _{C} (A-B)+\O _{C} (A-B+D)).
\]

At each vertex $v$, we can do a similar cancellation with the terms
\[
|\pi (v)| (1+h^{0} (N_{p/\X })) + \sum _{\begin{smallmatrix} A,B\in \pi (v)\\
A<B \end{smallmatrix}}\xi _{\pi (v)} (A)\xi _{\pi (v)} (A) h^{0} (\O
_{p} (A-B)\otimes \lambda _{-1} (N_{p/\X }))
\]
using the fact that 
\[
\lambda _{-1} (N_{p/\X })= \sum _{i=1}^{3} (\O _{p} (-D_{i})-\O _{p}
(D_{i})).
\]
These terms become
\[
|\pi (v)| + \sum _{A,B\in \pi (v)} \xi _{\pi (v)} (A)\xi _{\pi (v)}
(B)h^{0} \left(\sum _{i=1}^{3}\O _{p} (A-B+D_{i}) \right).
\]

The computations of this section are summarized by the following
theorem.
\begin{theorem}\label{thm: general sign formula}
Let $I\subset \O _{\X }$ be a torus fixed ideal corresponding to a
substack $Y\subset \X $ and let $\{\lambda (e),\pi (v) \}$ be the
corresponding sets of partitions. Then $s(\Ext^1_0(I,I))$, the parity
of the dimension of the Zariski tangent space of $I$ in $\Hilb (\X )$,
is given by
\[
s(\Ext^1_0(I,I)) = \chi (\O _{Y}) + \sum _{e\in \Edges }\SEsf _{\lambda (e)}(e)+ \sum _{v\in \Vertices } \SVsf _{\pi (v)}(v)
\]
where
\[
\SEsf_{\lambda } (e) = \sum _{A,B\in \lambda } \chi \left(\O _{C (e)}
(A-B)+\O _{C (e)} (A-B+D (e)) \right)
\]
and
\begin{align*}
\SVsf _{\pi } (v) = |\pi| &+ \sum _{A,B\in \pi } \xi _{\pi } (A)\xi _{\pi } (B) h^{0}\left(\sum _{i=1}^{3} \O _{p (v)} (A-B+D_{i} (e)) \right)\\
&+\sum _{i=1}^{3} \sum _{A\in \lambda (e_{i} (v))}\sum _{B\in \pi}
\xi _{\pi } (B) h^{0} \left(\O _{p (v)} (A-B)\otimes \lambda _{-1}N_{C (e_{i} (v))/\X } \right)\\
&+\sum _{i\neq j}\sum _{A\in \lambda (e_{i} (v))}\sum _{B\in
\lambda (e_{j} (v))}h^{0} (\O _{p (v)} (A-B+D_{j})).
\end{align*}
\end{theorem}

\begin{exmpl}
If ${\X}=X$ is a scheme, then $\SEsf_{\lambda } (e)$ simplifies to $m
(e)|\lambda |$ and $\SVsf_{\pi } (v)$ simplifies to 0 and we recover
the signs of the classical topological vertex. The simplifications are
straightforward:
\begin{align*}
\SEsf _{\lambda } (e)& = \sum _{A,B\in \lambda } \deg (A-B) +1 +\deg (A-B+D (e))+1\\
&= \sum _{A,B\in \lambda } \deg (D (e))\\
&= m (e)|\lambda |^{2} = m (e)|\lambda | \mod 2.
\end{align*}
As for the vertex term, note that $\lambda _{-1}N_{C (e_{i} (v))/X}$
restricted to $p (v)$ is zero, so
\begin{align*}
\SVsf_{\pi }  (v)&= |\pi | + \sum _{A,B\in \pi } 3\,\,\xi _{\pi } (A)\xi _{\pi } (B) \\
&\quad \quad \quad + \sum _{i\neq j} \sum _{A\in \lambda (e_{i} (v))}\sum _{D\in \lambda (e_{j} (v))} 1\\
&= |\pi |+3|\pi |^{2} +2\sum _{i<j} |\lambda (e_{i} (v))|\cdot  |\lambda (e_{j} (v))|\\
&= 0 \mod 2.
\end{align*}
\end{exmpl}

\begin{exmpl}
If $\X =[\cnums ^{3}/G]$ then there is a single vertex and each torus
invariant ideal $I$ corresponds to a single (leg-less) partition $\pi
$. Let $r_{1},r_{2},r_{3}\in \hat{G}$ be the characters of $G$ given
by $\O _{p} (D_{i})$ and let $0\in \hat{G}$ be the trivial
character. Let $|\pi |_{r}$ be the number of boxes in $\pi $ colored
by the character $r$. Then the sign associated to $I$ simplifies as
follows.
\begin{align*}
s (\Ext _{0}^{1} (I,I)) &= \chi (\O _{Y}) + \SVsf _{\pi }\\
&= |\pi |_{0} + |\pi | + \sum _{A,B\in \pi }\sum _{i=1}^{3} h^{0} (\O
_{p} (A-B+D_{i}))\\
&=|\pi |_{0} + |\pi | +\sum _{r\in \hat{G}} |\pi |_{r} \left(|\pi |_{r+r_{1}}+|\pi |_{r+r_{2}}+|\pi |_{r+r_{3}} \right)
\end{align*}
\end{exmpl}

\begin{remark}\label{rem: general vertex formula}
A general orbifold vertex formula can now be obtained. Using our
identification of the torus fixed points (Lemma~\ref{lem: fixed points
in bijection with partitions}), our $K$-Theory decomposition of torus
fixed ideas (Proposition~\ref{prop: K-theory decomposition of O_Y}),
our general sign formula (Theorem~\ref{thm: general sign formula}),
and the Behrend-Fantechi theorem, we get a combinatorial formula for
$DT (\X )$ of the form given by equation~\eqref{eqn: general vertex
formula}. The details of the formula, particularly the edge term,
depend on the choice of generators for $F_{1}K (\X )$.
\end{remark}

\subsection{Sign formula in the transverse $A_{n-1}$ case}\label{subsec: sign formula in A_n case}

In this section we simplify the sign formula from Theorem~\ref{thm:
general sign formula} in the case where $\X $ has transverse $A_{n-1}$
orbifold structure.

We first simplify the edge term $\SEsf _{\lambda } (e)$. First suppose
that $n=n (e)>1$ so that $C=C (e)$ is a $B\znums _{n}$ gerbe. Let $D=D
(e)$. Then
\begin{align*}
\SEsf _{\lambda } (e) &= \sum _{k=0}^{n-1} \sum _{\begin{smallmatrix} A,B\in \lambda \\
A\in \lambda [k,n] \end{smallmatrix}} \chi \left(\O _{C} (A-B)
\right)+\chi \left(\O _{C} (A-B+D) \right)\\
&=\sum _{k=0}^{n-1}\left(\sum _{A,B\in \lambda [k,n]}\deg (A)-\deg (B)+1 \right)\\
&\quad  +\left(\sum _{\begin{smallmatrix} A\in \lambda [k,n]\\
B\in \lambda [k-1,n] \end{smallmatrix}} (\deg (A)+1)- (\deg (B)+1)+\deg (D)+1 \right)\\
&=\sum _{k=0}^{n-1} |\lambda |_{k}^{2} +|\lambda |_{k-1}C^{\lambda }_{m,m'}[k,n] - |\lambda |_{k}C^{\lambda }_{m,m'}[k-1,n]+ (1+m)|\lambda |_{k}|\lambda |_{k-1}\\
&=\sum _{k=0}^{n-1} |\lambda |^{2}_{k}+C^{\lambda
}_{m,m'}[k,n]\left(|\lambda |_{k-1}-|\lambda |_{k+1} \right) +
(1+m)|\lambda |_{k}|\lambda |_{k-1}.
\end{align*}
Now suppose that $n=1$ so that $C$ is a football. Extracting the edge
terms from equation~\eqref{eq:sign_formula}, we get
\[
\SEsf _{\lambda } (e) = |\lambda |\left(1+h^{0} (N_{C/\X }) \right)
+\sum _{\begin{smallmatrix} A,B\in \lambda \\
A<B \end{smallmatrix}} \chi \left(\O _{C} (A-B)\otimes \lambda
_{-1}N_{C/\X } \right).
\]
Since $C$ is a football, and $\lambda _{-1}N_{C/\X } $ has rank and
degree zero, it is trivial in $K$-theory and so the term in the sum is
zero. Thus we compute (mod 2):
\begin{align*}
\SEsf _{\lambda } (e) &= |\lambda | \left(1+h^{0} (\O _{C} (D)\oplus \O _{C} (D')) \right)\\
&= |\lambda | \left(1+h^{0} (\O _{C} (D))+h^{1} (\O _{C} (-D'+K)) \right)\\
&= |\lambda |\left(1+\chi (\O _{C} (D)) \right)\\
&= |\lambda |\left(1+\tilde {m}+1-\delta _{0}-\delta _{0}' \right)\\
&= |\lambda |\left(\tilde m +\delta _{0}+\delta _{0}' \right).
\end{align*}

The vertex term simplifies as follows. Writing $\lambda _{i}=\lambda
(e_{i} (v))$ and using the facts that $\lambda _{-1}N_{p (v)/\X }=0$
and $\lambda _{-1}N_{C (e_{i})/\X }=0$ if $i=1$ or 2, the vertex terms
from equation~\eqref{eq:sign_formula} simplify to become
\begin{align*}
\SVsf _{\pi }&=|\pi | +\sum _{A\in \lambda _{3}}\sum _{B\in \pi }\xi _{\pi } (B) h^{0}\left(\O _{p} (A-B+D_{1})+\O _{p} (A-B+D_{2}) \right)\\
&\quad \quad + \sum _{i\neq j}\sum _{A\in \lambda _{i}}\sum _{B\in \lambda _{j}}
h^{0}\left(\O _{p} (A-B+D_{j}) \right)\\
&= \quad \sum _{k=0}^{n-1} |\pi |_{k}\left(|\lambda _{3}|_{k-1}+ |\lambda _{3}|_{k+1} \right)\\
&\quad +\sum _{k=0}^{n-1}|\lambda _{3}|_{k}\left(|\lambda _{1}|_{k}+|\lambda _{2}|_{k}+|\lambda _{1}|_{k-1}+|\lambda _{2}|_{k+1} \right).
\end{align*}

The above computations yield the following theorem.

\begin{theorem}\label{thm: sign formula for An case}
Let $\X $ be a orbifold toric \CYthree with transverse
$A_{n-1}$ orbifold structure. Then the sign formula in
theorem~\ref{thm: general sign formula} simplifies as follows
\[
s (\Ext ^{1}_{0} (I,I)) =\chi (\O _{Y}) + \sum _{e\in \Edges } \SEsf
_{\lambda (e)} (e) +\sum _{v\in \Vertices } \SVsf _{\pi (v)} (v)
\]
where
\[
\SEsf _{\lambda } (e) = 
\sum _{k=0}^{n-1}C^{\lambda }_{m,m'}[k,n]\left(|\lambda |_{k-1} -|\lambda |_{k+1} \right) +|\lambda |_{k } \left(1+ (1+m)|\lambda |_{k-1} \right)
\]
if $n=n (e)>1$,
\[
\SEsf _{\lambda } (e) = |\lambda |\left(\tilde {m}+\delta _{0}+\delta _{\infty } \right) 
\]
if $n=1$, and
\begin{align*}
\SVsf _{\pi } &=\quad  \sum _{k=0}^{n-1}|\pi |_{k} \left(|\lambda _{3}|_{k-1}+|\lambda _{3}|_{k+1} \right)\\
&\quad +\sum _{k=0}^{n-1}|\lambda _{3}|_{k}\left(|\lambda _{1}|_{k}+|\lambda _{2}|_{k}+|\lambda _{1}|_{k-1}+|\lambda _{2}|_{k+1} \right).
\end{align*}
\end{theorem}

Theorem~\ref{thm: main formula for transverse An}, our vertex formula
for $DT (\X )$ in the transverse $A_{n-1}$ case is now easily
proved. By Lemma~\ref{lem: fixed points in bijection with partitions}
and \cite[Theorem~3.4]{Behrend-Fantechi08}, the partition function is
given by a signed sum over edge assignments and compatible 3D
partitions at the vertices. Using Proposition~\ref{prop: K-theory
decomposition of O_Y}, Lemma~\ref{lem: vertex K-theory decomp}, and
Proposition~\ref{prop: K-theory decomp of an edge}, the variable
associated to each term in the sum is assigned. Finally, the sign of
each term is determined by Theorem~\ref{thm: sign formula for An
case}: the $\chi (\O _{Y})$ term is accounted for by adding a sign to
the $q$ variable and all the $q_{e,0}$ variables. The $\SEsf _{\lambda
} (e)$ term is accounted for by the $\Ssf _{\lambda } (e)$ term in the
formula, the first term in $\SVsf _{\pi } (v)$ is accounted for by
changing the signs on the vertex variables as in equation~\ref{eqn:
signs of vertex vars}, and the second term in $\SVsf _{\pi } (v)$ is
accounted for by the sign $(-1)^{\Sigma _{\pi (v)}}$.

\section{Proof of Theorem~\ref{thm: formula for Z_n
vertex}}\label{sec: proof of Zn vertex formula}

The proof of Theorem~\ref{thm: formula for Z_n vertex} involves some
intricate combinatorics, and thus we have broken it into several
subsections.

\subsection{Review of vertex operators}\label{subsec: review of vert ops}
\[
\quad 
\]
Let $\lambda \subset \znums _{\geq 0}^{2}$ be a partition (considered
as a Young diagram).  The \emph{rows} or \emph{parts} of $\lambda $
are the integers $\lambda _j = \min \{i \;|\;(i,j) \not \in \lambda
\}$, for $j \geq 0$.  Let $\lambda $ and $\mu $ be two partitions.  We
write $\lambda \succ \mu$ if
\[
\lambda_0 \geq \mu_0 \geq \lambda_1 \geq \mu_1 \geq \cdots
\]
and note that $\lambda \succ \mu$ if and only if as diagrams $\mu
\subset \lambda $ and $\lambda$ and $\mu$ are two adjacent diagonal
slices in some 3D partition (see for example \cite[\S3]{Ok-Re-Va}).

Fix $n$, and let $q_0, \ldots ,q_{n-1}$ be indeterminates.  Let
$\mathcal{R}$ be the ring of formal Laurent series in $q_0, \ldots,
q_{n-1}$.  Let $\mathcal{P}$ be the set of all Young diagrams, and let
$\mathcal{RP}$ be the free $\mathcal{R}$ module generated by elements
of $\mathcal{P}$.

We define two types of operators on $\mathcal{RP}$ in terms of their
action upon an element of $\mathcal{P}$.
\begin{defn}Let $x$ be a monomial in $q_0, \ldots, q_{n-1}$. Then 
\label{defn:Gamma_and_Q_operators}
\begin{align*}
\Gamma_{+}(x) \lambda &\stackrel{\text{def}}{=} \sum_{\mu \prec \lambda} x^{|\lambda| - |\mu|}\mu \\
\Gamma_{-}(x) \lambda &\stackrel{\text{def}}{=} \sum_{\mu \succ \lambda} x^{|\mu| - |\lambda|}\mu \\
Q_i \lambda &= q_i^{|\lambda|} \lambda \quad (0 \leq i \leq n-1)
\end{align*}
We will sometimes use the following shorthand notation:
\begin{align*}
\Gamma_{+1}(x) &= \Gamma_+(x) &
\Gamma_{-1}(x) &= \Gamma_-(x) &
Q &= Q_0Q_1\cdots Q_{n-1}
\end{align*}
\end{defn}

\begin{lemma} 
\label{lemma:Gamma_schur_property}
Let $\{x_i \;|\; i \in \znums_{\geq 0}\}$ be monomials in $q_0, \ldots, q_{n-1}$.  Then
\begin{align*}
\left< \mu  \left| \prod_{i \in \znums_{\geq 0}}\Gamma_{-}(x_i) \right| \lambda ' \right> &= s_{\mu /\lambda '}(x_0, x_1, x_2, \ldots), \\
\left< \mu  \left| \prod_{i \in \znums_{\geq 0}}\Gamma_{+}(x_i) \right| \lambda ' \right> &= s_{\lambda '/\mu }(x_0, x_1, x_2, \ldots). \\
\end{align*}
\end{lemma}

\begin{proof}
By elementary properties of Schur functions, this reduces immediately
to the case where $x_i = 0$ for $i > 1$ --- which in turn follows from
the semistandard Young tableau definition of the Schur
function\cite[Definition 7.10.1]{ECII}.
\end{proof}

\begin{corollary}
\label{cor:Gamma_bischur_property}
Let $\{x_i\}, \{y_i\}$ be monomials in $q_0, \ldots, q_{n-1}$.  Then
\begin{align*}
\left< \mu \left| \prod_{i \in \znums_{\geq 0}}\Gamma_{-}(x_i)
\prod_{i \in \znums_{\geq 0}}\Gamma_{+}(y_i) \right| \lambda ' \right>
&= \sum_{\eta }
s_{\mu /\eta}(\{x_i\} )s_{\lambda '/\eta}(\{y_i\}). \\
\end{align*}
\end{corollary}

\begin{proof}
If $\eta$ is a partition, then let $\delta_{\eta}$ be the projection
operator onto the space spanned by $\left|\eta\right>$.  Then
\begin{align*}
&\left< \mu  \left| \prod_{i\in \znums _{\geq 0}}\Gamma_{-}(x_i)  \prod_{i\in \znums _{\geq 0}}\Gamma_{+}(y_i) \right| \lambda ' \right>  \\
&= 
\sum_{\eta} \left< \mu  \left| \left(\prod_{i \in \znums_{\geq 0}}\Gamma_{-}(x_i) \right) \delta_{\eta} \left(\prod_{i \in \znums_{\geq 0}}\Gamma_{+}(y_i)\right) \right| \lambda ' \right> \\
&= \sum_{\eta} 
s_{\mu /\eta}(\{x_i\} )s_{\lambda' /\eta}(\{y_i\}). \\
\end{align*}
\end{proof}

It follows from Lemma~\ref{lemma:Gamma_schur_property} that these
$\Gamma_{\pm}$ are the same vertex operators as used
in~\cite{Ok-Re, Ok-Re-Pearcy,Ok-Re-Va}.  We therefore have

\begin{lemma}  
\label{lemma:Gamma_commutation}
Let $\sigma, \tau = \pm 1$ and let $a, b$ be monomials in $q_0, \ldots, q_{n-1}$.  Then
\[
\Gamma_{\sigma}(a) \Gamma_{\tau}(b) = (1 - ab)^{\frac{\tau - \sigma}{2}}
\Gamma_{\tau}(b) \Gamma_{\sigma}(a).
\]
\end{lemma}
\begin{proof}
The identity is derived by expressing $\Gamma_{+}$ and $\Gamma_{-}$ as
the exponential of another operator, and then applying the
Campbell-Baker-Hausdorff theorem.  This is done for the case $\sigma =
-\tau$ in~\cite[Lemma 31]{Bryan-Young} and the other cases are
essentially the same.
\end{proof}

\begin{lemma} 
\label{lemma:Gamma_Q_commutation}
Let $z$ be a monomial in $q_0, \ldots, q_{n-1}$.  Then
\[
\Gamma_{\sigma}(z)Q_i = Q_i\Gamma_{\sigma}(zq_i^{\sigma}).
\]
\end{lemma}
\begin{proof}
It is easy to check that, for any partitions $\lambda, \mu$, 
\begin{align*}
\left< \lambda \left|\Gamma_{+}(z)Q_i \right| \mu \right>
=
\left< \lambda \left|Q_i\Gamma_{+}(zq_i^{+1}) \right| \mu \right>
&=
\begin{cases}
z^{|\mu/\lambda|}q_i^{|\mu|}, & \lambda \subseteq \mu \\
0, &\lambda \not \subseteq \mu,
\end{cases} \\
\left< \lambda \left|\Gamma_{-}(z)Q_i \right| \mu \right>
=
\left< \lambda \left|Q_i\Gamma_{-}(zq_i^{-1}) \right| \mu \right>
&=
\begin{cases}
z^{|\lambda/\mu|}q_i^{|\mu|} & \mu \subseteq \lambda \\ 
0, &\mu \not \subseteq \lambda.
\end{cases}
\end{align*}
\end{proof}

We must also establish some notation for the edge sequence of the
partition $\nu$. Define the set $S(\nu)$ by
\[
S (\nu )=\{\nu_j - j - 1 \;|\; j \geq 0 \}.
\]
We define the \emph{edge sequence} of $\nu$: for $t \in \znums$,
\begin{equation}\label{eqn: edge sequence associated to partition}
\nu(t) = \begin{cases}
+1 &\text{if } t \in S (\nu ), \\
-1 &\text{if } t \not \in S (\nu ).
\end{cases}
\end{equation}
For example
\[
S (\emptyset ) = \{-1,-2,-3,\dotsc  \},\quad \quad  \emptyset (t) =\begin{cases}
+1 &t<0,\\
-1&t\geq 0.
\end{cases}
\]
Note that the complement of $S (\nu )$ is given by
\[
S (\nu )^{c} = -S (\nu  ')-1 = \{-\nu '_{j}+j\;   |\;  j\geq 0 \}.
\]

We use the following shorthand:

\begin{defn}
If $\alpha$ and $\beta$ are partitions, and $\sigma = \pm 1$, we write
$\alpha \spaceship{\sigma} \beta$ to mean
\[
\begin{cases}
\alpha \prec \beta &\text{if } \sigma = +1, \\
\alpha \succ \beta &\text{if } \sigma = -1. \\
\end{cases}
\]
\end{defn}

\subsection{Writing $\Vsf ^{n}_{\lambda \mu \nu }(q_{0},\dotsc ,q_{n-1})$ as a vertex operator product}
$\quad $\smallskip

Recall the following notation:
\begin{align*}
A_{\lambda } (k,n) &=\sum _{(i,j)\in \lambda } \left\lfloor
\frac{i+k}{n} \right\rfloor, \\
q^{-A_{\lambda } }&=\prod _{k=0}^{n -1} q_{k}^{-A_{\lambda }(k,n)}, \\
\frakq _{t} &= q^{-N} \prod _{k=-nN+1}^{t}q_{k} \quad \text{ for large } N, \text{ and} \\
q &= q_0\cdots q_{n-1}. \\
\end{align*}
Recall also that an overline denotes the exchange of variables
$q_{k}\leftrightarrow q_{-k}$ with subscripts in  $\znums_n$.

We will apply the following conventions for products of possibly
non-commuting operators. For operators $\Phi _{t}$ depending on $t\in
S\subset \znums $ we let
\[
\prod _{t\in S}^{\longrightarrow }  \Phi _{t}
\]
denote the product where $t$ increases from left to right in the order
the operators are written. We denote the retrograde expression as 
\[
\prod _{t\in S}^{\longleftarrow } \Phi _{t}.
\]

\begin{prop} \label{prop: vertex as operator expression}
The orbifold vertex is given by the following vertex operator
expression:
\[
\Vsf^{n}_{\lambda \mu \nu} = q^{-A_{\lambda}}
\overline{q^{-A_{\mu'}}}q_0^{-|\lambda|} \left< \mu \left|
\prod_{t}^{\longrightarrow}
\Gamma_{\nu'(t)}\left(\frakq_t^{-\nu'(t)}\right) \right| \lambda'
\right>.
\]
\end{prop}

\begin{proof}
We first make a slight refinement to the definition of
$\Vsf^{n}_{\lambda \mu \nu}$, as follows: Fix an integer $N$, and set
\[
\Vsf ^{n,N}_{\lambda \mu \nu } = \sum _{\pi } q_{0}^{|\pi |_{0}}\dotsb q_{n-1}^{|\pi |_{n-1}}.
\]
where the sum is now taken over all 3D partitions $\pi $ asymptotic to
$(\lambda, \mu, \nu )$ such that any boxes $(i,j,k)$ not contained in
the $\lambda$-leg or the $\mu$-leg satisfy $i<nN, j<nN$.  It is clear
that
\[
\lim_{N \rightarrow \infty}\Vsf ^{n,N}_{\lambda \mu \nu }
= \Vsf ^{n}_{\lambda \mu \nu }
\]
in the sense that the low order terms of $\Vsf ^{n,N}_{\lambda \mu \nu
}$ and $\Vsf^{n}_{\lambda \mu \nu}$ agree.

Following the strategy of~\cite{Ok-Re, Bryan-Young}, we will calculate
this generating function as a matrix coefficient in a product of
vertex operators.  The simplest case, $\lambda = \mu = \nu =
\emptyset$, is done in full detail in~\cite{Bryan-Young}. The case
$n=1$ but with $\lambda, \mu, \nu$ arbitrary is handled in
\cite{Ok-Re}.

Consider, as a first approximation to $\Vsf ^{n,N}_{\lambda \mu \nu
}$, the expression

\begin{align}
\label{eqn:vertex_op_expression}
\left< \mu \left|
\prod_{-nN+1\leq t\leq nN-1}^{\longrightarrow }Q_t \Gamma_{\nu'(t)}(1)
\right| \lambda' \right>
\end{align}

Observe that, for each $t$, 
\[
Q_t \Gamma_{\nu'(t)}(1) \left| \gamma \right> = \sum_{\eta \spaceship{\nu'(t)} \gamma} q_{t}^{|\eta|} \left| \eta \right>.
\]

So, in other words, $Q_t \Gamma_{\nu'(t)}(1)$ sends a partition
$\left|\gamma\right>$ to a weighted formal sum of all partitions
$\left|\eta\right>$ such that $\gamma$ and $\eta$ are the $(t+1)$st
and $t$th slices, respectively, in a 3D partition.  In this sum, each
$\eta$ is weighted by $q_{t}^{|\eta|}$.  Since $Q_t\Gamma_{\nu'(t)}(1)$
is a linear operator, this property extends to linear combinations of
such $\left|\gamma \right>$, so
\[
Q_{t+1}\Gamma_{\nu'(t+1)}(1)Q_{t} \Gamma_{\nu'(t)}(1) \left|
\gamma\right\rangle = \sum_{\alpha \spaceship{\nu'(t+1)} \beta
\spaceship{\nu'(t)} \gamma } Q_{t+1}^{|\alpha|} Q_{t}^{|\beta|}
\left|\alpha\right>,
\]
and so forth.  Since the indices on the $Q_i$ operators are taken
modulo $n$, the $\left<\mu\right|$ coordinate of
\[
\prod^{\longrightarrow }_{-nN+1\leq t\leq nN-1}Q_t \Gamma_{\nu'(t)}(1)\left|\lambda'\right>
\]
counts sequences of $\znums_n$-weighted Young diagrams, interlacing
according to $\nu'$, \footnote{The easiest way to see that we must use
the edge sequence associated to $\nu '$ and not $\nu $ is to look at
Figure~\ref{fig:add_n_border_strip} and note that the $i$ and $j$ axes
not in the standard order so that we are looking at $\nu $ ``from the
bottom'' and hence getting $\nu '$.} beginning with $\lambda'$ and
ending with $\mu$, as does $\Vsf^{n,N}_{\lambda \mu \nu }$.  However,
there are two important differences between $\Vsf^{n,N}_{\lambda \mu
\nu }$ and~\eqref{eqn:vertex_op_expression}.

First, the contribution of the box $(i,j,k)$ in the 3D partition $\pi$
to $\Vsf^{n}_{\lambda \mu \nu}$ is $q_{i-j}^{\xi_{\pi}(i,j,k)}$, where
recall that
\[
\xi_{\pi}(i,j,k) = 1 - \# \text{ of legs of $\pi $ containing }
(i,j,k) .
\]
By contrast, \eqref{eqn:vertex_op_expression} assigns weight
$q_{i-j}^{\xi'_{\pi}(i,j,k)}$ where
\[
\xi'_{\pi}(i,j,k) = \begin{cases}
1 & (i,j,k) \in \pi \setminus \{\nu \text{ leg}\}, \\
0 & \text{otherwise}.
\end{cases}
\]
See Figure~(\ref{fig:venn_diagram_legs}) for a comparison of $\xi$ and
$\xi'$.

\begin{figure}
\caption{Differences in weighting between \eqref{eqn:vertex_op_expression} and $\Vsf^{n,N}_{\lambda \mu \nu}$}\label{fig:venn_diagram_legs}
\begin{center}
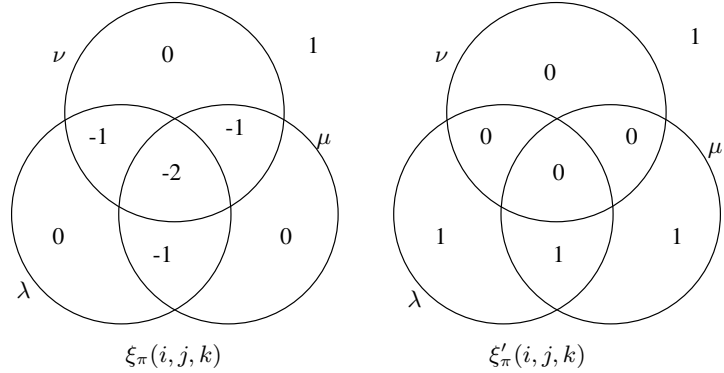
\end{center}
\end{figure}
To account for this difference, we
divide~\eqref{eqn:vertex_op_expression} by the weight of the $\lambda$
and $\mu$ legs, $(q_0 \cdots q_{n-1})^{N(|\lambda| + |\mu|)}$.  This
may be achieved with the $Q$ operators:
\[
\left< \mu \left| Q^{-N} \prod_{-nN+1\leq t\leq nN-1}^{\longrightarrow
}Q_t \Gamma_{\nu'(t)}(1) Q^{-N} \right| \lambda' \right>
\]
Using the commutation relations of
Lemma~\ref{lemma:Gamma_Q_commutation}, we move the operators $Q_t$ to
the left if $t \leq 0$, and to the right if $t > 0$, giving
\begin{equation}
\label{eqn:normalized_vertex_op_expression} \left< \mu \left| \left(
\prod_{-nN+1\leq t\leq nN-1}^{\longrightarrow } \Gamma_{\nu'(t)}
(\frakq_t^{-\nu'(t)}) \right)Q_0^{-1} \right| \lambda' \right>.
\end{equation}

The second difference between
\eqref{eqn:normalized_vertex_op_expression} and $\Vsf^{n,N}_{\lambda
\mu \nu}$ is that each partition in~\eqref{eqn:vertex_op_expression}
has a contribution from the boxes which lie inside the $\lambda$ or
$\mu$-leg, outside the region $i,j \leq n$, and inside the region $|i
-j| \leq n$; these are the regions in
Figure~\ref{fig:normalization_framing} at the left and right sides of
the first picture, whose projections to the $xy$ plane are triangular,
and whose cross-sections, when viewed from the left, are $\lambda$ and
$\mu'$.

The weights contributed by these regions are $q^{A_{\lambda}}$ and
$\overline{q^{A_{\mu'}}}$, as explained in
Lemma~\ref{lem:framingfactor} below.  In the non-orbifold case,
\cite{Ok-Re} refers to these constants as \emph{framing factors}.  The
terms from the corresponding partitions in $\Vsf^{n,N}_{\lambda \mu
\nu}$ do not have this contribution.

At this point we have nearly proven the proposition.  We have 
\[
\Vsf^{n,N}_{\lambda \mu \nu} = q^{-A_{\lambda}}
\overline{q^{-A_{\mu'}}}q_0^{-|\lambda|} \left< \mu \left| \prod_{
-nN+1\leq t\leq nN-1}^{\longrightarrow }
\Gamma_{\nu'(t)}\left(\frakq_t^{-\nu'(t)}\right) \right| \lambda'
\right> + \text{error}
\]
where the expressions in both sides assign the same weight to a 3D
partition.  All that remains is to understand the ``error'' term:
$\Vsf^{n,N}_{\lambda \mu \nu}$
and~\eqref{eqn:normalized_vertex_op_expression}, written as formal
sums over 3D partitions, are not supported on the same index set.  In
particular, \eqref{eqn:normalized_vertex_op_expression} includes
contributions from 3D partitions which have boxes outside of $[0,N]
\times [0,N] \times [0,\infty]$ but inside the region $|x-y| < N$.
However, the smallest such 3D partition grows without bound as $N$
grows large, so the error term disappears in the large-$N$ limit.
\end{proof}

\begin{figure}
\caption{A 3D partition which fits within an $N \times N \times
\infty$ box, compared with the corresponding sequence of $2N+1$ Young
diagrams.}\label{fig:normalization_framing}
\begin{center}
\input{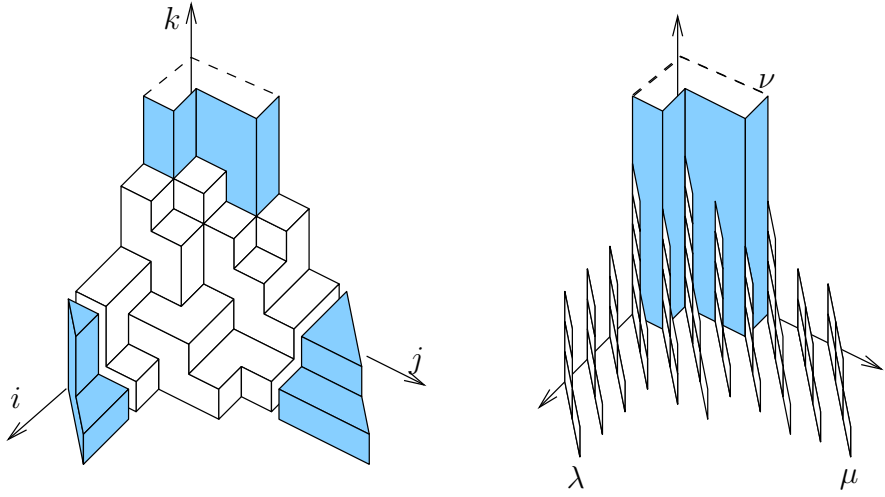}
\end{center}
\end{figure}

\begin{lemma}
\label{lem:framingfactor}
Let $L,M \subseteq (\znums_{\geq 0})^3$ be the regions
\begin{align*}
L&=\left\{ (i,j,k) \;|\; (j,k) \in \lambda, i > nN-1, i-j \leq nN-1\right\},\\
M&=\left\{ (i,j,k) \;|\; (i,k) \in \mu ', j > nN-1, i-j \geq
-nN+1\right\}.
\end{align*}
Then 
\begin{align*}
\prod_{(i,j,k) \in L} q_{i-j}& = q^{A_{\lambda}},&
\prod_{(i,j,k) \in M} q_{i-j} &= \overline{q^{A_{\mu '}}}.
\end{align*}
\end{lemma}

\begin{proof}
Let $L_t$ denote the diagonal slice 
\[L_t = \{ (i,j,k) \in L\;|\; i-j = t\}. \]
Observe that when $t > nN-1$, $L_t$ is the empty set.  Moreover,
$L_{nN-1}$ is the largest of the $L_t$; it consists of boxes $(nN-1+j,
j, k)$ where $(j,k) \in \lambda$ and $j \geq 1$ (see
Figure~\ref{fig:framingfactor}).  Each of these boxes contributes
weight $q_{-1}$ to $\prod_L q_{i-j}$, since the subscripts of $q$ are
taken mod $n$.  Similarly, for $c > 0$, 
\[
L_{nN-c} = \{ (nN-c+j, j, k) \; |\; (j,k) \in \lambda,\; j \geq c\}
\]
where each box in $L_{nN-c}$ has color $-c$. It follows that
\begin{align*}
\prod_{(i,j,k) \in L} q_{i-j} 
&= \prod_{c=1}^{\infty} \prod_{(i,j,k) \in L_{nN-c}} q_{-c} \\
&= \prod_{m=0}^{\infty} \prod_{\widetilde{c} = 1}^n \prod_{\substack{(j,k) \in \lambda\\j \geq nm + \widetilde{c}}}q_{-\widetilde{c}} \\
&= \prod_{\tilde {c}=1}^{n} \prod_{(j,k) \in \lambda} \prod_{m=0}^{\left\lfloor \frac{j-\tilde {c}}{n}  \right\rfloor} q_{-\widetilde{c}} \\
&= \prod_{\widetilde{c} = 1}^{n} \prod_{(j,k) \in \lambda}q_{-\widetilde{c}}^{\left\lfloor \frac{j-\tilde {c}}{n}  \right\rfloor + 1} \\
&= \prod_{c = 0}^{n-1} \prod_{(j,k) \in \lambda}
q_{c}^{\left\lfloor \frac{j+c}{n}  \right\rfloor}.
\end{align*}
The second line uses the fact that the subscripts of the $q_i$ are
taken modulo $n$.  In the last line we changed variables by
$\widetilde{c} \mapsto n - c$. The end result is precisely equal to
$q^{A_{\lambda}}$ as defined in Section~\ref{subsec: the vertex
formula}.

Similarly, let 
\[
M_t = \{(i,j,k) \in M \;|\; i-j = t \}.
\]
When $t < -nN + 1$, $M_t$ is empty; otherwise, for $c > 0$,
\[
M_{-nN+c} = \{ (i, nN - c + i, k) \;|\; (i,k) \in \mu', i \geq c \},
\]
where each box in $M_{-nN+c}$ has color $c$. It follows that
\begin{align*}
\prod_{(i,j,k) \in M} q_{i-j} &= \prod_{c=1}^{\infty} \prod_{(i,j,k) \in M_{-nN+c}}q_c \\
&= \prod_{m=0}^{\infty} \prod_{\tilde{c} = 1}^{n} \prod_{\substack{(i,k) \in \mu' \\ i \geq nm + \tilde{c}}} q_{\tilde{c}} \\
&= \prod_{\tilde{c} = 1}^\infty\prod_{(i,k) \in \mu'} \prod_{m=0}^{\left\lfloor\frac{i - \tilde{c}}{n} \right\rfloor} q_{\widetilde{c}} \\
&= \prod_{\tilde{c} = 1}^{n} \prod_{(i,k) \in \mu'}q_{\tilde{c}}^{\left\lfloor\frac{i - \tilde{c}}{n} \right\rfloor + 1} \\
&= \prod_{c=0}^{n-1} \prod_{(i,k) \in \mu'} q_{-c}^{\left\lfloor \frac{(i+c)}{n}\right\rfloor} \\
&= \overline{q^{A_{\mu'}}}.
\end{align*}

\end{proof}

\begin{figure}
\caption{Computation of the framing factor associated to $\lambda$}
\label{fig:framingfactor}
\begin{center}
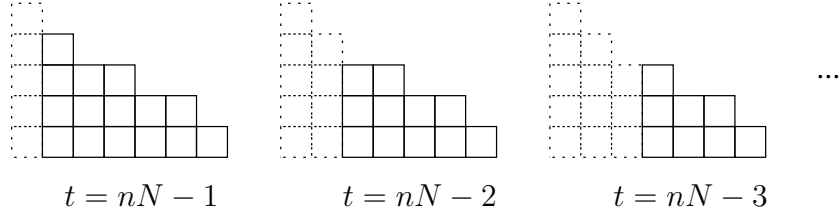
\end{center}
\end{figure}

\subsection{$n$-quotient, $n$-core, and the retrograde}

\begin{figure}
\caption{A three-core $\nu$, viewed as a partition and as a triple of integers $(3, -2, -1)$ summing to zero.  Note that $\nu_0, \nu_1, \nu_2$ are empty partitions.}\label{fig:threecore}
\begin{center}
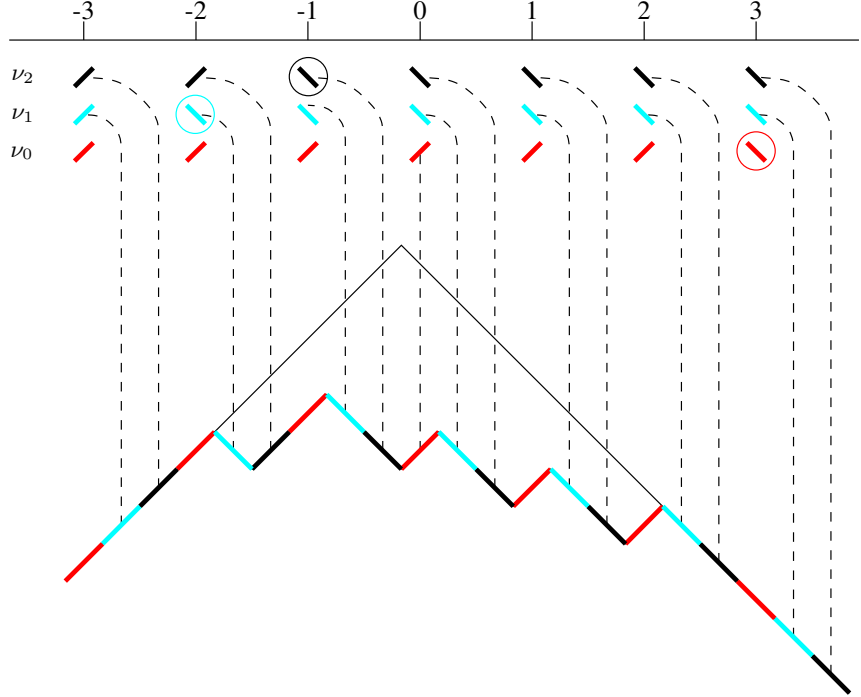
\end{center}
\end{figure}

$\quad $
\smallskip

\subsubsection{Edge sequences and charge}
Let $\underline{\nu}: \znums  \rightarrow \{\pm 1\}$ be a function
satisfying $\underline{\nu}(t) = -1$ for $t \gg 0$, and
$\underline{\nu}(t) = 1$ for $t \ll 0$. We say that $\underline{\nu }
(t)$ is an \emph{edge sequence}, and to such a sequence we associate
its \emph{slope diagram} which consists of the graph of a continuous,
piecewise linear function having slopes $\pm 1$, such that the slope
of the function at $t$ is given by $\underline{\nu } (t)$ and such the
changes in slope occur at half-integers. 

The slope diagram associated to a sequence $\underline{\nu }$
determines a Young diagram and hence a partition. The Young diagram is
given by rotating the slope diagram 135 degrees counterclockwise and
translating so that the positive $x$ and $y$ axes eventually coincide
with the rotated slope diagram. Note that this association is
consistent with the edge sequence $\nu (t)$ associated with a
partition $\nu $ as defined in equation~\eqref{eqn: edge sequence
associated to partition}. However, there are many edge sequences
having the same associated partition. If $\underline{\nu } (t)$ is an
edge sequence having associated partition $\nu $, then there exists a
unique integer $c (\underline{\nu })\in \znums $ such that
\[
\underline{\nu }  = R^{c (\underline{\nu })} \nu  
\]
where $R$ is the right-shift operator, which acts on an edge sequence $\underline{\eta}$ by
\[
R\underline{\eta}(t) = \underline{\eta}(t-1).
\]
We call $c (\underline{\nu })$ the \emph{charge } of $\underline{\nu
}$.  The edge sequence $\nu(t)$ associated to a partition by
equation~\eqref{eqn: edge sequence associated to partition} always has
charge zero; we adopt the convention an edge sequence without an
underline always has charge zero.  The uniqueness of $c
(\underline{\nu })$ implies that the map
\begin{align}\label{eqn: bijections between edge seqs and partitions,charge}
\{\text{edge sequences}\} &\rightarrow \{\text{partitions}\} \times \znums \\
\underline{\nu}(t) &\mapsto (\nu, c(\underline{\nu}(t))\nonumber 
\end{align}
is a bijection, so we will use these notations interchangeably.

\subsubsection{Ribbons, the $n$-quotient, and $n$-core} 
$\quad $\smallskip

There is an operation known as \emph{adding a ribbon} to an edge
sequence $\underline{\nu}$.  Fix $t_1 < t_2$ with
$\underline{\nu}(t_1) = 1, \underline{\nu}(t_2) = -1$ (there are
infinitely many such pairs $(t_1, t_2)$).  Then construct a new edge
sequence $\underline{\rho}$ such that
\[
\underline{\rho}(t) = \begin{cases}
-1 & t=t_1 \\
+1 & t=t_2 \\
\underline{\nu}(t) &\text{otherwise}.
\end{cases}
\]
If $\nu$ and $\rho$ are the Young diagrams associated to
$\underline{\nu}$ and $\underline{\rho}$, then the set-theoretic
difference $\rho - \nu$ is a connected strip of boxes which contains
no $2\times2$ region, commonly called a \emph{ribbon}, \emph{border
strip} or \emph{rim hook} in the combinatorics literature; we shall
use the term to refer to either the strip of boxes or to the endpoints
$(t_1, t_2)$, according to whether we are speaking of Young diagrams
or edge sequences.  We say that the ribbon is of \emph{length}
$t_2-t_1$ and to lie at \emph{position} $t_1$.  It is easy to check
that adding a ribbon does not affect the charge of an edge sequence.

Observe that any charge-zero edge sequence can be constructed from
$\underline{\emptyset}$ by adding ribbons of length 1.  This
corresponds to adding boxes to a Young diagram in such a way that the
result remains a Young diagram.  

If $\underline{\nu }$ is an edge sequence, we define its
\emph{associated $n$-tuple} $(\underline{\nu}_0, \ldots,
\underline{\nu}_{n-1})$ of edge sequences by
\[
\underline{\nu}_i(t) = {\underline{\nu }}(nt + i).
\]
Letting $(\nu_i, c_i)=\underline{\nu}_i$ under the bijection
\eqref{eqn: bijections between edge seqs and partitions,charge}, we
then define the \emph{$n$-quotient} and the \emph{$n$-core} of
$\underline{\nu } $ to be $(\nu_{0}, \ldots, \nu_{n-1})$ and $(c_0,
\ldots, c_i)$ respectively. 

The process of passing from an edge sequence to its $n$-core and
$n$-quotient is reversible: there is a unique way to construct an edge
sequence $\underline{\nu }$ with a prescribed $n$-core and
$n$-quotient.  As such, we identify $\underline{\nu }$ with its
$n$-quotient together with its $n$-core:
\[
\underline{\nu } \leftrightarrow ((\nu_0, \ldots, \nu_{n-1}), (c_0,
\ldots, c_{n-1})).
\]
If the edge sequence $\underline{\nu }$ is charge zero (i.e. it came
from a partition), then $\sum_i c_i = 0$. Customarily, one only
considers $n$-cores arising from partitions, and so unless otherwise
stated, we will assume that all $n$-cores 
satisfy $\sum _{i}c_{i}=0$.  One special case is worthy of
note.  The partition whose $n$-quotient is $c = (c_0, \ldots,
c_{n-1})$ and whose $n$-quotient is $(\emptyset, \ldots, \emptyset)$
is often identified with $c$, and is customarily also called an
$n$-core.

Note that adding an $n$-hook to $\nu$ at position $t \equiv t_0
\pmod{n}$ corresponds to adding a 1-hook (i.e. a single box) to
$\nu_{t_0}$, without altering the $n$-core, or any of the other
$\nu_i$.  It follows that the $n$-core of $\nu$ is the (unique)
partition obtained by iteratively removing $n$-hooks from $\nu$ until
it is impossible to do so.

Let $R_k$ be the operator which acts on an edge sequence
$\underline{\nu}$ by right-shifting the $k$th component of the
associated $n$-tuple of $\underline{\nu}$:
\[
R_k(\underline{\nu}_0, \underline{\nu}_1, \ldots,
\underline{\nu}_{n-1}) =(\underline{\nu}_0, \underline{\nu}_1, \ldots,
R\underline{\nu}_k, \ldots, \underline{\nu}_{n-1}).
\]
Note that $R_k$ increases the charge of $\underline{\nu }$ by one.  It
follows that the operator $R_k R_{k+1}^{-1}$ leaves the charge of
$\underline{\nu }$ unaffected, so it restricts to an operator on
partitions and hence defines an operator on $\mathcal{RP}$.  The
effect of ${R}_k R_{k+1}^{-1}$ is to leave the $n$-quotient of
$\underline{\nu }$ unaffected, while incrementing $c_k$ and
decrementing $c_{k+1}$.  Moreover, the operators ${R}_k R_{k+1}^{-1}$
and their inverses, acting on $\emptyset$, are sufficient to generate
any $n$-core. Indeed, if $\nu $ is an $n$-core $(c_{0},\dotsc
,c_{n-1})$, then the associated edge sequence is given by
\[
\nu  = \prod _{i=0}^{n-1} R^{c_{i}}_{i}\emptyset .
\]

\begin{remark}\label{rem: induction strategy for proofs using
n-quotients}We can prove statements about partitions inductively, in
the following manner.  To prove the statement $P(\nu)$:
\begin{enumerate}
\item Prove $P(\emptyset)$.
\item Prove that $P(\nu) \Leftrightarrow P(R_k R_{k+1}^{-1}\nu)$ for each $k$.
\item Prove that $P(\nu) \Rightarrow P(\rho)$, where $\rho$ is any partition obtained from $\nu$ by adding a ribbon.
\end{enumerate}
Proving (1) and (2) establishes $P$ for all $n$-core partitions, and
then (3) extends the proof to all partitions.
\end{remark}
 
\begin{figure}
\caption{Applying $R_0  R^{-1}_1$ to a 4-core generates a new 4-core, increasing the weight by $q_{1}^{-1}q$.}\label{fig:fourcoreenlarge}
\begin{center}
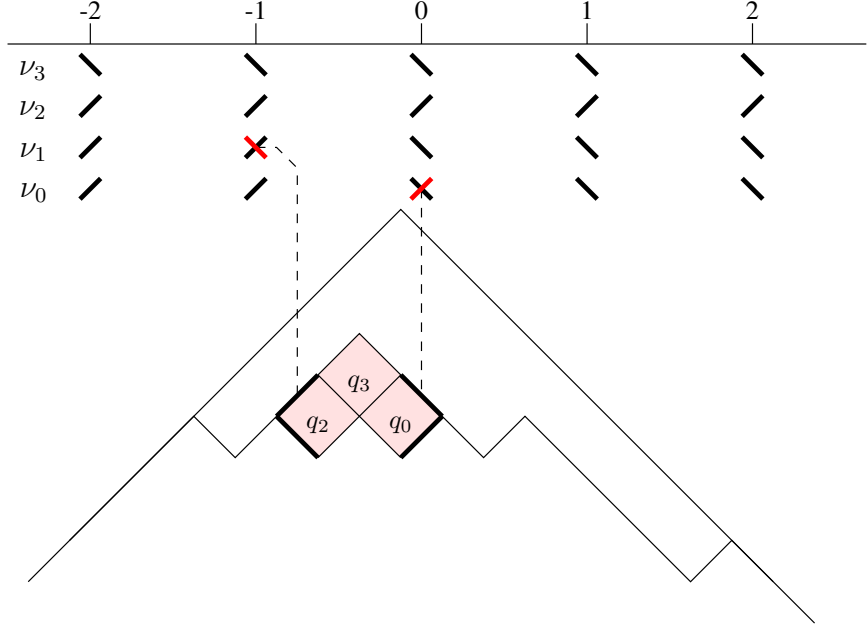
\end{center}
\end{figure}

\subsubsection{Comparison of the operator with its retrograde.}

\begin{prop}\label{prop: retrograde prop}
The operator expression appearing in Proposition~\ref{prop: vertex as
operator expression} can be written in terms of its retrograde and a
scalar operator, namely
\[
\prod_{t}^{\longrightarrow}
\Gamma_{\nu'(t)}\left(\frakq_t^{-\nu'(t)}\right) = V^{n}_{\emptyset
\emptyset \emptyset} \cdot O_{\nu} \cdot \Mon _{\nu '}^{-1}\cdot
\prod_{t}^{\longleftarrow}
\Gamma_{\nu'(t)}\left(\frakq_t^{-\nu'(t)}\right)
\]
where
\begin{align*}
O_{\nu } &= \prod _{k=0}^{n-1}\Vsf ^{n}_{\emptyset \emptyset \emptyset } (q_{k},q_{k+1},\dotsc ,q_{n+k-1})^{-2|\nu |_{k} + |\nu |_{k+1}+ |\nu |_{k-1}},\\
\Vsf ^{n}_{\emptyset \emptyset \emptyset } &= M (1,q)^{n}\prod
_{0<a\leq b<n} M (q_{a}\dotsb q_{b},q) M (q^{-1}_{a}\dotsb q^{-1}_{b},q), \\
M(v,q) &= \prod _{m=1}^{\infty } \frac{1}{(1-vq^{m})^{m}},\\
\Mon _{\nu '}&= (-1)^{|\nu |} \prod _{(j,i)\in \nu '} \prod _{s=0}^{n-1} q_{s}^{h^{s}_{\nu '} (j,i)},\text{ and}\\
h^{s}_{\nu '} (j,i) &= \text{the number of boxes of color $s$ in the $(j,i)$-hook of $\nu '$}.
\end{align*}
\end{prop}
\begin{proof}
Replacing the product with its retrograde has the effect of reversing
the order of every pair of operators
$\Gamma_{\nu'(t)}(\frakq_t^{-\nu'(t)})$,
$\Gamma_{\nu'(t')}(\frakq_{t'}^{-\nu'(t')})$ for $t'>t$.  By
Lemma~\ref{lemma:Gamma_commutation}, this introduces a scalar factor:
\begin{equation*}
\prod_{t}^{\longrightarrow}
\Gamma_{\nu'(t)}\left(\frakq_t^{-\nu'(t)}\right)= \prod_{t <
t'}(1-\frakq_{t}^{-\nu' (t)} \frakq_{t'}^{-\nu'
(t')})^{\frac{1}{2}(\nu'(t')-\nu'(t))}\cdot \prod_{t}^{\longleftarrow}
\Gamma_{\nu'(t)}\left(\frakq_t^{-\nu'(t)}\right)
\end{equation*}
so that to prove the Lemma, we must prove
\[
\prod_{t < t'}(1-\frakq_{t}^{-\nu' (t)} \frakq_{t'}^{-\nu'
(t')})^{\frac{1}{2}(\nu'(t')-\nu'(t))}= \Mon_{\nu'}^{-1}\cdot \Vsf
_{\emptyset \emptyset \emptyset }\cdot O_{\nu' }.
\]

We begin by simplifying the right hand side.  Let 
\[
\Hook_{\nu'} = \{ (t, t') \in \znums^2 \;|\; t < t', \nu'(t) = -1, \nu'(t') = 1\}.
\]
Observe that $\Hook _{\nu'}$ is a finite set, and indeed is in
bijection with the set of hooks of $\nu'$, as the ordered pairs $(t,
t')$ represent the ends of the legs of a hook.  In turn, each hook of
$\nu '$ corresponds uniquely to some $(j,i)\in \nu '$ given by the
corner of the hook. The product over the hooks then becomes:
\begin{align*}
&\prod _{(t,t')\in \Hook _{v'}}\left(1 - \frakq_{t}^{-\nu ' (t)}\frakq_{t'}^{-\nu ' (t')}\right)^{\frac{1}{2}(\nu'(t')-\nu'(t))}\\
=&\prod _{(t,t')\in \Hook _{v'}}\left(1 - \frakq_{t}^{+1}\frakq_{t'}^{-1}\right) \\
=& \prod _{(t,t')\in \Hook _{v'}}\left(-\frakq_t^{+1} \frakq_{t'}^{-1} \right)\prod _{(t,t')\in \Hook _{v'}}\left(1-\frakq_t^{-1}\frakq_{t'}\right)\\
=&(-1)^{|\nu |} \prod _{(t,t')\in \Hook _{v'}} q_{t+1}^{-1}\dotsb q_{t'}^{-1}\prod _{(t,t')\in \Hook _{v'}} (1-\frakq _{t}^{-1}\frakq _{t'}^{+1})\\
=&\Mon _{\nu '}^{-1} \prod _{(t,t')\in \Hook _{v'}} (1-\frakq
_{t}^{-1}\frakq _{t'}^{+1})
\end{align*}
where the last equality follows from the fact that $(t+1,\dotsc ,t')$
are exactly the set of colors of the boxes in the $(j,i)$-hook of $\nu
'$ corresponding to $(t,t')$.

Using the above, we can then write
\[
\prod _{t<t'}\left(1-\frakq _{t}^{-\nu ' (t)}\frakq _{t'}^{-\nu '
(t')} \right)^{\frac{1}{2}\left(\nu ' (t')-\nu ' (t) \right)} = \cC
(\nu ')\cdot \Mon ^{-1}_{\nu '}
\]
where
\begin{equation}\label{eqn: definition of C(nu')}
\cC(\nu') = \prod_{t<t'}\left(1 - \frakq_t^{-1}\frakq_{t'}^{+1}\right)^{\frac{1}{2}(\nu'(t') - \nu'(t))}.
\end{equation}

We need to prove that $\cC (\nu ') = \Vsf ^{n}_{\emptyset \emptyset
\emptyset }\cdot O_{\nu }$ and we will do so using the induction
strategy described in Remark~\ref{rem: induction strategy for proofs
using n-quotients}.

We first study the base case for this strategy, $\nu ' = \emptyset $.

\[
\cC (\emptyset ) =\prod _{t<0}\prod _{t'\geq 0} \left(1-\frakq
_{t}^{-1} \frakq _{t'}^{+1} \right)^{-1}.
\]
Letting $t=nt_{0}+c$ and $t'=nt_{0}'+d$ where $c,d\in \{0,\dotsc ,n-1
\}$ we see that
\[
\frakq _{t}^{-1} \frakq _{t'}^{+1} = \begin{cases}
q^{t_{0}'-t_{0}}\cdot q_{c+1}\dotsb q_{d} &d>c,\\
q^{t_{0}'-t_{0}} &d=c,\\
q^{t_{0}'-t_{0}}\cdot q_{d+1}^{-1}\dotsb q_{c}^{-1} &d<c.
\end{cases}
\]
Then writing $m=t_{0}'-t_{0}$ we get
\begin{align*}
\cC (\emptyset ) &=\prod _{m=1}^{\infty }\left(1-q^{m} \right)^{-mn} \prod _{0<a\leq b<n} \left(1-q_{a}\dotsb q_{b}\cdot q^{m} \right)^{-m} \left(1-q^{-1}_{a}\dotsb q^{-1}_{b}\cdot q^{m} \right)^{-m}\\
&=\Vsf ^{n}_{\emptyset \emptyset \emptyset }\cdot O_{\emptyset }
\end{align*}
which proves the base case of the induction.

Observe that adding an $n$-hook to $\nu'$ leaves the quantity
\[
-2|\nu'|_k + |\nu'|_{k-1} + |\nu'|_{k+1}
\]
invariant, for each $k$: an $n$-border strip contains one box of each
of the $n$ colors.  As such, $O_{\nu'}$ depends only upon the $n$-core
of $\nu'$.  We will show that $\mathcal{C} ({\nu'})$ also depends only
upon the $n$-core of $\nu'$, which lets us reduce to the case where
$\nu'$ itself is an $n$-core partition.

To do this, let $\rho'$ be a partition obtained by adding an
$n$-border strip to $\nu'$ at position $T$ (see
Figure~\ref{fig:add_n_border_strip}).  In particular, this means that
\begin{align*}
\rho'(T) &= -1, & \nu'(T) &= +1, \\ 
\rho'(T+n) &= +1, & \nu'(T+n) &= -1. \\ 
\end{align*}

\begin{figure}
\caption{The partition $\rho'$ is obtained from $\nu'$ by adding a
length $n$ border strip at time $T =
i-j$.}\label{fig:add_n_border_strip}
\begin{center}
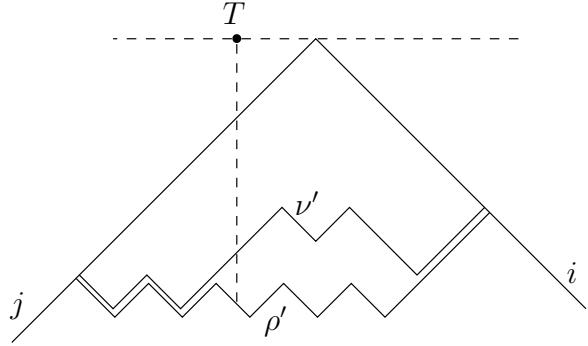
\end{center}
\end{figure}

We will show that $\mathcal{C}(\rho') / \mathcal{C}(\nu')$ = 1.  
First, it is helpful to rewrite $\mathcal{C} (\nu ')$ as follows:

\[
\mathcal{C}(\nu') = \prod_{k \geq 0} \prod_{t \in \znums}
\left(1 - \frakq_{t+k} ^{+1} \frakq_{t}^{-1}\right)^{\frac{1}{2}(\nu(t+k) - \nu(t))}.
\]
Let
\[
K(t,k) = \frac{1}{2}\left[(\rho(t+k) - \nu(t+k)) - (\rho(t) - \nu(t)) \right]
\]
so that
\[
\frac{\mathcal{C'}(\rho)}{\mathcal{C'}(\nu)} = \prod_{k \geq 0}\prod_{t \in \znums}
\left(1 - \frakq_{t+k} ^{+1}\frakq^{-1}_{t}\right)^{K(t,k)}.
\]
Observe that $K(t,k) = 0$ unless $\rho'(t) \neq \nu'(t)$ or $\rho'(t+k)
\neq \nu'(t+k)$.  Moreover, the edge sequences of $\rho'$ and $\nu'$
differ only at $T$ and at $T+n$, since $\rho'$ is the result of adding
an $n$-border strip at position $T$ to $\nu'$.  Therefore,
\begin{align*}
\frac{\mathcal{C}(\rho')}{\mathcal{C}(\nu')} &= \,\,
\prod_{k \geq 0} (1 - \frakq_{T+k} ^{+1} \frakq^{-1}_{T})^{K(T,k)}
\quad \prod_{k \geq 0} (1 - \frakq_{T+n+k} ^{+1} \frakq^{-1}_{T+n})^{K(T+n,k)} \\
&\quad \cdot\prod_{k \geq 0} (1 - \frakq_{T} ^{+1} \frakq^{-1}_{T-k})^{K(T-k,k)}
\prod_{k \geq 0} (1 - \frakq_{T+n} ^{+1} \frakq^{-1}_{T+n-k})^{K(T+n-k,k)} \\
&= \,\,\prod_{k \geq 0}(1 - \frakq_{T+k} ^{+1} \frakq^{-1}_{T})^{K(T,k) + K(T+n,k)}\\
&\quad \cdot 
 \prod_{k \geq 0}(1 - \frakq_{T} ^{+1} \frakq^{-1}_{T-k})^{K(T-k,k) + K(T+n-k,k)}.
\end{align*}

We next examine the quantity $K(T,k)+K(T+n,k)$.  Consider first the
case $k \neq n$.  In this case we have $\rho(T+k) = \nu(T+k)$,
$\rho(T+k-n) = \nu(T+k-n)$, so
\begin{align*}
2(K(T,k) + K(T+n, k)) 
&=  \rho'(T) - \nu'(T) +\rho'(T+n) - \nu'(T+n)\\
&= 0
\end{align*}
because $\rho'(T) = -\rho'(T+n)$, $\nu'(T) = -\nu'(T+n)$.  As such,
all terms other than possibly those where $k=n$ cancel from the
product, so
\begin{align*}
\frac{\mathcal{C} (\rho' )}{\mathcal{C} (\nu' )} 
&= (1 - \frakq_{T+n}^{+1}\frakq^{-1}_{T})^{K(T,n) + K(T+n,n)}
 (1 - \frakq_{T}^{+1}\frakq^{-1}_{T-n})^{K(T-n,n) + K(T,n)} \\
&= (1-q)^{2K(T,n) + K(T+n,n)+ K(T-n,n) }.
\end{align*}
All of the terms in the exponent can now be computed explicitly, since they involve only known quantities:  
\begin{align*}
K(T+n,n) &= \frac{1}{2} (-\rho'(T+n) + \nu'(T+n)) = -1, \\
K(T,n) &=\frac{1}{2}( (\rho'(T+n) - \nu'(T+n)) - (\rho'(T) - \nu'(T))) = 2 ,\\
K(T-n,n) &= \frac{1}{2} (-\rho'(T+n) + \nu'(T+n)) = 1.
\end{align*}
Thus we have $\mathcal{C}(\rho')/\mathcal{C}(\nu') = 1$.  This means
that adding an $n$-border strip to $\nu'$ does not affect
$\mathcal{C}(\nu')$, and as such $\mathcal{C}(\rho')$ depends only
upon the $n$-core of $\nu'$.

Thus to finish the proof of proposition~\ref{prop: retrograde prop}
using the induction argument outlined in Remark~\ref{rem: induction
strategy for proofs using n-quotients}, it remains only to prove the
following lemma.

\begin{lemma}\label{lem: C(rho)/C(nu)=O_rho/O_nu}
Let $\nu '$ be an $n$-core and let $\rho ' = R_{k}R^{-1}_{k+1}\nu '$,
then
\[
\frac{\cC (\rho ')}{\cC (\nu ')} = \frac{O_{\rho '}}{O_{\nu '}}. 
\]
\end{lemma}

We prove the lemma by direct computation. To streamline the notation
we will drop the primes from $\nu '$ and $\rho '$.

We define $T_{k}$ to be the operator which cyclically permutes the
variables by $k$:
\[
(T_{k}F) (q_{0},\dotsc ,q_{n-1}) = F (q_{k},\dotsc ,q_{k+n-1}).
\]
Note that it follows immediately from equation~\eqref{eqn: definition
of C(nu')} that
\begin{equation}\label{eqn: C(R^k nu)=TkC(nu)}
\cC (R^{k}\nu ) = T_{k}\cC (\nu ).
\end{equation}
We begin with a computation:
\begin{align*}
\frac{\Vempty }{T_{1}\Vempty } &= \prod _{0<a\leq b<n} \frac{M (q_{a}\dotsb q_{b},q)M (q^{-1}_{a}\dotsb q^{-1}_{b},q)}{M (q_{a+1}\dotsb q_{b+1},q)M (q^{-1}_{a+1}\dotsb q^{-1}_{b+1},q)}\\
&=\prod _{m=1}^{\infty }\prod _{0<a\leq b<n} \frac{(1-q_{a+1}\dotsb q_{b+1}q^{m})^{m}(1-q^{-1}_{a+1}\dotsb q^{-1}_{b+1}q^{m})^{m}}{(1-q_{a}\dotsb q_{b}q^{m})^{m}(1-q^{-1}_{a}\dotsb q^{-1}_{b}q^{m})^{m}}\\
&=\prod _{m=1}^{\infty } \frac
{\prod _{a=1}^{n-1} (1-q_{a+1}\dotsb q_{n}q^{m})^{m}(1-q^{-1}_{a+1}\dotsb q^{-1}_{n}q^{m})^{m}}
{\prod _{b=1}^{n-1}(1-q_{1}\dotsb q_{b}q^{m})^{m}(1-q^{-1}_{1}\dotsb q^{-1}_{b}q^{m})^{m}}\\
&=\prod _{m=1}^{\infty }\prod _{c=1}^{n-1} \frac
{(1-q_{c+1}\dotsb q_{n}q^{m})^{m} (1-q_{1}\dotsb q_{c}q^{m-1})^{m}}
{(1-q_{1}\dotsb q_{c}q^{m})^{m}(1-q_{c+1}\dotsb q_{n}q^{m-1})^{m}}\\
&=\prod _{m=1}^{\infty }\prod _{c=1}^{n-1} \frac{(1-q_{1}\dotsb q_{c}q^{m-1})}{(1-q_{c+1}\dotsb q_{n}q^{m-1})}.
\end{align*}
In the above, the equality from the second to the third line is
because all the terms cancel except for those in the numerator with
$(a,b)= (a,n-1)$ and those in the denominator with $(a,b)= (1,b)$. The
equality from the fourth to the last line uses the reindexing
$m\mapsto m-1$ on the first terms in the numerator and denominator.

We now wish to compare $\cC (\nu )$ to $\cC (R_{0}\nu )$. Since $\nu = (c_{0},\dotsc ,c_{n-1})$ is an $n$-core, we have that
\[
\nu (cn) = \begin{cases}
+1&c<c_{0},\\
-1&c\geq c_{0}.
\end{cases}
\]
Thus $R_{0}\nu = (c_{0}+1,c_{1},\dotsc ,c_{n-1}) $ differs from $\nu $
(as an edge sequence) only at $t=c_{0}n$ where we have
\begin{align*}
(R_{0}\nu ) (c_{0}n)&=1 &\nu (c_{0}n)&=-1.
\end{align*}
Thus
\begin{align*}
\frac{\cC (\nu )}{\cC (R_{0}\nu )} &= \prod _{t<c_{0}n} \left(1-\frakq ^{-1}_{t}\frakq _{c_{0}n} \right)^{\frac{1}{2} (-1-\nu (t))-\frac{1}{2} (1-\nu (t))}\\
&\,\,\,\, \cdot \prod _{c_{0}n<t} \left(1-\frakq ^{-1}_{c_{0}n}\frakq _{t} \right)^{\frac{1}{2} (\nu (t)+1)-\frac{1}{2} (\nu (t)-1)}\\
&=\prod _{t<c_{0}n}\left(1-\frakq _{t}^{-1}\frakq _{c_{0}n} \right)^{-1} \cdot \prod _{c_{0}n<t}\left(1-\frakq ^{-1}_{c_{0}n}\frakq _{t} \right)^{+1}.
\end{align*}
In the above expression, we can rewrite the product over $t>c_{0}n$ as
a product over $m=1,2,\dotsc $ and $a=1,\dotsc ,n$ by setting $t=
(c_{0}+m-1)n+a$ so that 
\[
\frakq ^{-1}_{c_{0}n}\frakq _{t} = q_{c_{0}n+1}\dotsb q_{(c_{0}+m-1)n+a} = q_{1}\dotsb q_{a}\cdot q^{m-1}.
\]
Similarly, we can rewrite the product over $t<c_{0}n$ as a product
over $m=1,2,\dotsc $ and $a=n-1,n-2,\dotsc ,0$ by $t=n (c_{0}-m)+a$ so
that
\[
\frakq _{t}^{-1}\frakq _{c_{0}n} = q_{n (c_{0}-m)+a+1}\dotsb
q_{c_{0}n} =q_{a+1}\dotsb q_{n}\cdot q^{m-1}.
\]
Thus
\begin{align*}
\frac{\cC (\nu )}{\cC (R_{0}\nu )} &= \prod _{m=1}^{\infty }\frac{\prod _{a=1}^{n} (1-q_{1}\dotsb q_{a}q^{m-1})}{\prod _{a=0}^{n-1} (1-q_{a+1}\dotsb q_{n}q^{m-1})}\\
&=\prod _{m=1}^{\infty }\prod _{a=1}^{n-1} \frac{(1-q_{1}\dotsb q_{a}q^{m-1})}{(1-q_{a+1}\dotsb q_{n}q^{m-1})}\\
&=\frac{\Vempty }{T_{1}\Vempty }
\end{align*}
and so we have shown
\begin{equation}\label{eqn: C(R0nu) = T1V/V*C(nu)}
\cC (R_{0}\nu ) = \frac{T_{1}\Vempty }{\Vempty }\cdot \cC (\nu )
\end{equation}
for any edge sequence $\nu $ with empty $n$-quotient. 

The operator $R_{k}$ can be obtained from $R_{0}$ by conjugating with
$R^{k}$:
\[
R_{k} = R^{k}R_{0}R^{-k}
\]
and thus
\[
R_{k}R^{-1}_{k+1} = R^{k}R_{0}RR^{-1}_{0}R^{-k-1}.
\]
We now compute using equations \eqref{eqn: C(R^k nu)=TkC(nu)} and
\eqref{eqn: C(R0nu) = T1V/V*C(nu)}:
\begin{align*}
\cC (R_{k}R_{k+1}^{-1}\nu ) &= \cC (R^{k}R_{0}R\,R_{0}^{-1}R^{-k-1}\nu )\\
&= T_{k}\cC (R_{0} (R\,R_{0}^{-1}R^{-1-k}\nu ))\\
&= T_{k}\left(\frac{T_{1}\Vempty }{\Vempty }\cdot  \cC (R\,R_{0}^{-1}R^{-k-1}\nu ) \right)\\
&=\frac{T_{k+1}\Vempty }{T_{k}\Vempty }\cdot T_{k+1}\cC (R_{0}^{-1}R^{-k-1}\nu )\\
&=\frac{T_{k+1}\Vempty }{T_{k}\Vempty }\cdot T_{k+1}\left(\frac{\Vempty }{T_{1}\Vempty }\cdot  \cC (R^{-k-1}\nu ) \right)\\
&=\frac{T_{k+1}\Vempty }{T_{k}\Vempty }\cdot \frac{T_{k+1}\Vempty }{T_{k+2}\Vempty }\cdot T_{k+1}T_{-k-1}\cC (\nu )
\end{align*}
and so 
\[
\frac{\cC (R_{k}R_{k+1}^{-1}\nu )}{\cC (\nu )} = \frac{(T_{k+1}\Vempty )^{2}}{T_{k}\Vempty \cdot T_{k+2}\Vempty }.
\]

On the other hand, we have
\[
\frac{O_{R_{k}R^{-1}_{k+1}\nu }}{O_{\nu }} = \prod _{l=0}^{n-1}\left(T_{l}\Vempty  \right)^{\epsilon _{l}}
\]
where
\[
\epsilon _{l} = -2\left(\left| R_{k}R^{-1}_{k+1}\nu \right|_{l}-|\nu
|_{l } \right) + \left(\left| R_{k}R^{-1}_{k+1}\nu \right|_{l-1}-|\nu
|_{l-1 } \right) +\left(\left| R_{k}R^{-1}_{k+1}\nu \right|_{l+1}-|\nu
|_{l+1 } \right) .
\]

The operation $R_{k}R^{-1}_{k+1}$ adds one box of each color except
for $k+1$ to an $n$-core $\nu $ (see
figure~\ref{fig:fourcoreenlarge}). Therefore
\begin{align*}
\epsilon _{l} &= -2 (1-\delta _{k+1,l}) + (1-\delta _{k+1,l-1}) + (1-\delta _{k+1,l+1})\\
&= +2 \delta _{k+1,l} - \delta _{k+1,l-1} - \delta _{k+1,l+1} 
\end{align*}
and so
\[
\frac{O_{R_{k}R^{-1}_{k+1}\nu }}{O_{\nu }} = \frac{(T_{k+1}\Vempty )^{2}}{T_{k}\Vempty \cdot T_{k+2}\Vempty } = \frac{\cC (R_{k}R_{k+1}^{-1}\nu )}{\cC (\nu )}.
\]

This completes the proof of Lemma~\ref{lem: C(rho)/C(nu)=O_rho/O_nu}
and hence of Proposition~\ref{prop: retrograde prop}.
\end{proof}

\begin{prop}\label{prop: retrograde expression equals Hnu*schur function expression}
\[
\left< \mu \left| \prod_{t}^{\longleftarrow}
\Gamma_{\nu'(t)}\left(\frakq_t^{-\nu'(t)}\right) \right| \lambda '
\right> = H_{\nu' }\cdot \Mon_{\nu' }\cdot 
\mathsf{Schur}_{\lambda \mu \nu }
\]
where
\begin{align*}
H_{\nu' } &= \prod _{(j,i)\in \nu' }\frac{1}{1-\prod
_{s=0}^{n-1}q_{s}^{h^{s}_{\nu' } (j,i)}},\\
\Mon_{\nu' }&= (-1)^{|\nu |} \prod _{(j,i)\in \nu '} \prod _{s=0}^{n-1}
q_{s}^{h^{s}_{\nu '} (j,i)},\text{ and}\\
\mathsf{Schur}_{\lambda \mu \nu }&=\sum _{\eta } s_{\mu /\eta }\left(\frakq _{t}|_{\nu ' (t)=-1} \right) s_{\lambda ' /\eta }\left(\frakq^{-1} _{t}|_{\nu ' (t)=+1 }\right).
\end{align*}
\end{prop}
\begin{proof}
We commute the operators so that all the $\Gamma _{+}$s are on the
right and all the $\Gamma _{-}$s are on the left. Using the
commutation relations we obtain
\[
\prod _{t}^{\longleftarrow }\Gamma _{\nu ' (t)}\left(\frakq _{t}^{-\nu ' (t)} \right) = \prod _{\begin{smallmatrix} t'>t\\
\nu ' (t)=-1\\
\nu ' (t')=+1 \end{smallmatrix}} \frac{1}{1-\frakq _{t'}^{-1}\frakq _{t}^{+1}} 
\prod _{\nu ' (t)=-1} \Gamma _{-} (\frakq _{t})\prod _{\nu ' (t)=+1} \Gamma _{+} (\frakq^{-1} _{t}).
\]

Observe that $\nu'(t) = -1, \nu'(t') = 1$ for $t < t'$ if and only if
there is a hook of $\nu'$ with endpoints at $t, t'$. Moreover, we can
rewrite the monomials appearing in the scalar factor above as follows:
\begin{align*}
\frakq _{t'}^{-1}\frakq _{t}^{+1}& = q_{t+1}^{-1}\cdot q_{t+2}^{-1}\dotsb q_{t'}^{-1}\\
&= \prod _{s=0}^{n-1} q_{s}^{-h^{s}_{\nu '} (j,i)}
\end{align*}
where the hook corresponding to $(t',t)$ has corner $(j,i)\in \nu '$
and $h^{s}_{\nu '} (j,i)$ is the number of boxes of color $s$ in the
hook. Clearing the denominators of inverses, we find that the scalar
factor is exactly equal to
\[
H_{\nu '}\cdot \Mon _{\nu '}.
\]
The equality
\[
\left\langle \mu \left| \prod _{\nu ' (t)=-1} \Gamma _{-} (\frakq
_{t})\prod _{\nu ' (t)=+1} \Gamma _{+} (\frakq^{-1}
_{t})\right|\lambda ' \right\rangle = \mathsf{Schur}_{\lambda \mu \nu
}
\]
follows immediately from Corollary~\ref{cor:Gamma_bischur_property}
and the lemma is proved.
\end{proof}

We can now put it all together and complete the proof of
Theorem~\ref{thm: formula for Z_n vertex}. Combining
Propositions~\ref{prop: vertex as operator expression} and \ref{prop:
retrograde prop} we get
\[
\Vsf ^{n}_{\lambda \mu \nu } = \Vsf ^{n}_{\emptyset \emptyset
\emptyset }\cdot  q^{-A_{\lambda }}\cdot \overline{q^{-A_{\mu '}}} O_{\nu }
\cdot q_{0}^{-|\lambda |} \cdot \left< \mu \left| \prod_{t}^{\longleftarrow}
\Gamma_{\nu'(t)}\left(\frakq_t^{-\nu'(t)}\right) \right| \lambda '
\right>.
\]
Applying Proposition~\ref{prop: retrograde expression equals Hnu*schur
function expression} and using homogeneity of Schur functions, we get
\begin{multline}\label{eqn: almost final eqn for V in the proof}
\Vsf ^{n}_{\lambda \mu \nu } = \Vsf ^{n}_{\emptyset \emptyset
\emptyset } \cdot q^{-A_{\lambda }}\cdot \overline{q^{-A_{\mu '}}}\cdot  O_{\nu } \cdot H_{\nu
}\cdot \\
\sum _{\eta } q_{0}^{-|\eta |}s_{\mu /\eta }\left(
\frakq _{t}|_{\nu ' (t)=-1} \right) s_{\lambda ' /\eta
}\left(q_{0}^{-1}\cdot\frakq^{-1} _{t}|_{\nu ' (t)=+1 }\right).
\end{multline}
Finally, using
\begin{align*}
S (\nu ')& = -S (\nu ')^{c}-1, & q_{0}^{-1}\cdot \frakq _{t}^{-1} &=
\overline{\frakq _{-1-t}},
\end{align*}
we observe the following equalities of sets:
\begin{align*}
\{\frakq _{t}\,:\, \nu ' (t)=-1 \} &= \{\frakq _{t}\,:\, t\in S (\nu ')^{c} \}\\
& = \frakq _{\bullet -\nu '}\\
\{ q_{0}^{-1}\cdot\frakq^{-1} _{t}\,:\,{\nu ' (t)=+1 }\}&=\{\overline{\frakq _{-1-t}} \,:\, t\in S (\nu' ) \}\\
&=\{\overline{\frakq _{-1-t}} \,:\, t\in -S (\nu' )^{c}-1 \}\\
&=\{\overline{\frakq _{T}}\,:\, T\in S (\nu ')^{c} \}\\
&=\overline{\frakq _{\bullet -\nu '}}
\end{align*}
which, when substituted into equation~\eqref{eqn: almost final eqn for
V in the proof}, completes the proof of Theorem~\ref{thm: formula for
Z_n vertex}.\qed

\appendix 

\bigskip\bigskip\bigskip 

\section{Grothendieck-Riemann-Roch for orbifolds and the Toen
operator.}\label{app: GRR and Toen operator.}

We briefly review Grothendieck-Riemann-Rock for Deligne-Mumford stacks
and we work out some examples needed in the paper. The basic reference
is \cite{Toen}; see also \cite[Appendix~A]{Tseng-GandT}.

Let $\X $ be a smooth Deligne-Mumford stack. Let $I\X $ be the inertia
stack of $\X $. The objects of $I\X $ are pairs $(x,g)$ where $x$ is
an object of $\X $ and $g$ is an automorphism of $x$. There is a local
immersion
\[
\pi :I\X \to \X 
\]
which forgets $g$.

Let $E$ be a vector bundle on $I\X $. There is a canonical
automorphism\footnote{induced by the canonical 2-morphism $\pi
\Rightarrow \pi $ given by $(x,g)\mapsto g$.} of $E$ and consequently
there is a decomposition
\[
E=\oplus _{\omega } E^{\omega }
\]
where the sum is over roots of unity $\omega \in \cnums $ and the
canonical automorphism acts by multiplication by $\omega $ on
$E^{\omega }$.

We define an endomorphism $\rho $ of $K (I\X )\otimes \cnums $ by 
\[
\rho (E) = \sum _{\omega } \omega [E^{\omega }].
\]
Let $N$ be the normal bundle to the local immersion $\pi :I\X \to \X $
and let
\[
\lambda _{-1} (N^{\vee }) = \sum _{i} (-1)^{i} \Lambda ^{i}N^{\vee }
\in K (I\X ).
\]
We define the \emph{Toen operator}
\[
\tau_{\X } :K (\X ) \to A (I\X )
\]
by
\[
\tau _{\X } (E) = \frac{ch (\rho (\pi ^{*}E))}{ch (\rho (\lambda
_{-1}N^{\vee }))} \cdot td (T_{I\X })
\]
where $td (T_{I\X }) $ is the Todd class of $I\X $. 

Toen's Grothendieck-Riemman-Roch theorem for stacks asserts
\cite[4.10,~4.11]{Toen} that $\tau $ is functorial with respect to
proper pushforwards.
\begin{theorem}[Toen]\label{thm: Toen's GRR}
Let $f:\X \to \mathcal{Y}$ be a proper morphism of smooth
Deligne-Mumford stacks, then for all $E\in K (\X )$,
\[
f_{*} (\tau _{\X } (E)) = \tau _{\mathcal{Y}} (f_{*}E).
\]
In particular, for $f:\X \to \pt $, we get
\[
\chi (E) =\int _{I\X } \tau _{\X } (E).
\]
\end{theorem}

\begin{exmpl}\label{exmpl: toen operator for BZn}
Let $\X =B\znums _{n}$, then $I\X =\cup _{l=0}^{n-1}\X ^{l}$ where $\X
^{l}\cong B\znums _{n}$. If $L_{k}\to \X $ is the line bundle
determined by the 1-dimensional representation of $\znums _{n}$ having
character $\omega ^{k}$ where $\omega =\exp\left(\frac{2\pi i}{n}
\right)$, then the canonical automorphism acts by multiplication with
$\omega ^{kl} $ on $L_{k}$ restricted to $\X ^{l}$. Thus
\[
\tau _{\X } (L_{k})|_{\X ^{l}} = \omega ^{kl}.
\]
\end{exmpl}

\begin{exmpl}\label{exmpl: toen operator for BZn gerbe over a curve}
Let $\mathcal{C}\to \P ^{1}$ be a $B\znums _{n}$ gerbe and let
$L_{k,m}\to \mathcal{C}$ be a line bundle with
\[
\deg (L_{k,m}) = m \in \frac{1}{n}\znums 
\]
and such that the restriction of $L_{k,m} $ to a point $B\znums _{n}\in \mathcal{C}$ is the bundle $L_{k}$ from example~\ref{exmpl: toen operator for BZn}. Then $I\mathcal{C}=\cup _{l=0}^{n-1}\mathcal{C}^{l}$ and 
\[
\tau (L_{k,m})|_{\mathcal{C}^{l}} = \omega ^{kl} (1+m[\pt ]).
\]
\end{exmpl}
\begin{exmpl}\label{exmpl: GRR for the football}
Let $\P ^{1}_{a,b}$ be the football, i.e. the stack given by root
constructions \cite{Cadman} of orders $a$ and $b$ at the points
$[0]\in \P ^{1}$ and $[\infty ] \in \P ^{1}$ respectively. Let $[\pt
]\in \P ^{1}_{a,b}$ be a non-stacky point. The following lemma gives a
formula for the Euler characteristic of a line bundle on the football.
\begin{lemma}
\[
\chi \left(\O _{\P ^{1}_{a,b}} (d[\pt ]+s[0]+t[\infty ]) \right)=d+1+\left\lfloor \frac{s}{a_{\,}}  \right\rfloor+\left\lfloor \frac{t}{b}  \right\rfloor.
\]
\end{lemma}
The inertia stack breaks into components as follows:
\[
I\P ^{1}_{a,b} = \P ^{1}_{a,b} \,\,\bigcup _{k=0}^{a-1} P_{k} \,\,\bigcup _{l=0}^{b-1} Q_{l}
\]
where $P_{k}\cong B\znums _{a}$ and $Q_{l}\cong B\znums _{b}$. Let
$\omega _{a}=\exp\left(\frac{2\pi i}{a} \right)$ and $\omega
_{b}=\exp\left(\frac{2\pi i}{b} \right)$, then $N_{P_{k}/\P
^{1}_{a,b}}$ is the line bundle on $B\znums _{a}$ with character
$\omega _{a}^{k}$ and $N_{Q_{l}/\P ^{1}_{a,b}}$ is the line bundle on
$B\znums _{b}$ with character $\omega _{b}^{l}$. Therefore
\[
\tau \left(\O_{\P ^{1}_{a,b}} (d[\pt ]+s[0]+t[\infty ])
\right)\restricts{\P ^{1}_{a,b}} = \left(1+d[\pt ]+s[0]+t[\infty ]
\right)\cdot \left(1+\frac{1}{2} ([0]+[\infty ]) \right),
\]
\[
\tau \left(\O_{\P ^{1}_{a,b}} (d[\pt ]+s[0]+t[\infty ])
\right)\restricts{P_{k}} =\frac{ch\left(\rho (\O
(s[0])\restricts{P_{k}}) \right)}{ch\left(\rho \left(1-N^{\vee
}_{P_{k}/I\P ^{1}_{a,b}} \right) \right)} = \frac{\omega
_{a}^{ks}}{1-\omega _{a}^{-k}},
\]
and similarly
\[
\tau \left. \left(\O_{\P^1_{a,b}}  (d[\pt ]+s[0]+t[\infty ]) \right) \right|_{Q_{l}} = \frac{\omega _{b}^{lt}}{1-\omega _{b}^{-l}}.
\]
Now integrating $\tau \left(\O (d[\pt ]+s[0]+t[\infty ]) \right) $ over
$I\P ^{1}_{a,b}$, we get
\[
\chi \left(\O_{\P ^{1}_{a,b}} (d[\pt ]+s[0]+t[\infty ]) \right) = d+\frac{s}{a}+\frac{t}{b} + \frac{1}{2a} +\frac{1}{2b} + \frac{1}{a}\sum _{k=1}^{a-1} \frac{\omega _{a}^{ks}}{1-\omega _{a}^{-k}} +\frac{1}{b}\sum _{l=1}^{b-1} \frac{\omega _{b}^{lt}}{1-\omega _{b}^{-l}}.
\]
The lemma then follows from the identity\footnote{You can have some
fun and try to prove this elementary identity for yourself. If you get
stuck, a complete proof can be found at:\\
\texttt{www.math.ubc.ca/$\sim$jbryan/papers/identity.pdf}.}
\[
\frac{1}{a} \sum _{k=1}^{a-1} \frac{\omega _{a}^{ks}}{1-\omega _{a}^{-k}} =\left\lfloor \frac{s}{a_{\,}}  \right\rfloor -\frac{s}{a} +\frac{a-1}{2a}
\]
and its counterpart for the sum over $l$.
\end{exmpl}

\section{Orbifold toric \CYthrees and web
diagrams}\label{app: web diagrams}

A orbifold toric \CYthree is a smooth toric Deligne-Mumford stack
$\X $ with generically trivial stabilizers and having trivial
canonical bundle.

\begin{lemma}\label{lem:toric_orbiCY_determined_by_cspace}
A orbifold toric \CYthree $\X $ is uniquely determined by its
coarse moduli space $X$.
\end{lemma}
\textsc{Proof:} This follows from the classification result of
Fantechi, Mann, and Nironi \cite{Fantechi-Mann-Nironi}. They show that
if $\X $ is a smooth Deligne-Mumford toric stack, then the structure
morphism to the coarse space factors canonically via toric morphisms
\[
\X \to \X ^{rig}\to \X ^{can}\to X
\]
where $\X \to \X ^{rig} $ is an Abelian gerbe over $\X ^{rig}$, $\X
^{rig}\to \X ^{can}$ is a fibered product of roots of toric divisors,
and $\X ^{can }\to X$ is the minimal orbifold having $X$ as
its coarse moduli space. They prove that $\X ^{can}$ is unique and
canonically associated to $X$. Since we assume $\X $ is an orbifold,
we have $\X =\X ^{rig}$. Since we assume $K_{\X } $ is trivial, the
stacky locus in $\X $ has codimension at least two and hence $\X =\X
^{can}$. \QED

The combinatorial data determining a toric variety is well understood
and is most commonly expressed as the data of a fan (by the above
lemma, we do not require the stacky fans of Borisov, Chen and Smith
\cite{Borisov-Chen-Smith}). In the case of an orbifold toric \CYthree,
it is convenient to use equivalent (essentially dual) combinatorial
data, namely that of a web diagram.

\begin{defn}\label{defn: web-diagram} A \emph{web diagram} consists of
the data
\begin{itemize}
\item A graph $\Gamma$ which is trivalent and
embedded in the plane. The graph is finite and necessarily has some
non-compact edges.
\item A marking $\{x_{v,e}\}$, which consists of a non-zero vector
$x_{v,e}\in \znums ^{2}$ for each pair $(v,e)$ where $e$ is an edge
incident to a vertex $v$.
\end{itemize}
The data satisfies the following.
\begin{itemize}
\item For each compact edge $e$ with bounding vertices $v$ and $v'$,
\[
x_{v,e}+x_{v'e}=0.
\]
\item For each vertex $v$ with incident edges $(e_{1},e_{2},e_{3})$,
\[
x_{v,e_{1}}+x_{v,e_{2}}+x_{v,e_{3}} =0.
\]
\end{itemize}
Two markings $\{x_{v,e} \}$ and $\{x'_{v,e} \}$ are equivalent if
there exists $g\in SL_{2} (\znums )$ such that $gx_{e,v}=x'_{v,e}$ for
all $(v,e)$.
\end{defn}

\begin{lemma}\label{lem: a toric orbiCY3 corresponds to a web diagram}
Every orbifold toric \CYthree $\X $ determines a
web diagram $\Gamma _{\X }$, unique up to equivalence.
\end{lemma}
\textsc{Proof:} By lemma~\ref{lem:toric_orbiCY_determined_by_cspace}, $\X $ is determined by its coarse space $X$, a toric
variety with Gorenstein finite quotient singularities and trivial
canonical bundle. Such an $X$ determines a simplicial fan $\Sigma
\subset N\otimes \qnums $ with $N\cong \znums ^{3}$. Since the
canonical divisor is trivial, there exists a linear function $l:N\to
\znums $ such that $l (v_{i})=1$ for all the generators $v_{i}$ of the
one dimensional cones of $\Sigma $. Thus $\Sigma $ intersects the
plane $\{l=1 \}$ in a triangulation $\hat{\Gamma }$ having integral
vertices. Let $\Gamma _{\X }=\Gamma $ be the graph dual to
$\hat{\Gamma }$ in the plane $\{l=1 \}$. We define a marking of
$\Gamma $ as follows. Under duality, a vertex in $\Gamma $ with
incident edge $e$ corresponds to a triangle $\hat{v}$ in $\hat{\Gamma
}$ and a bounding edge $\hat{e}$. Fixing an orientation on the plane,
the edge $\hat{e} $ inherits an orientation from the triangle
$\hat{v}$. The oriented edge defines an integral vector $x_{v,e}$ in
$\{l=0 \}$. The set $\{x_{v,e} \}$ satisfies the conditions of a
marking by construction.\QED

\begin{remark}
When we picture the web diagram $\Gamma $ in relation to the
triangulation $\hat{\Gamma }$, we will use an element of $SL_{2}
(\znums )$ to rotate the vectors $x_{v,e } $ counterclockwise by
ninety degrees so that the edges of $\Gamma $ are perpendicular to the
edges of $\hat{\Gamma }$. In Figure~\ref{fig: web diagrams}, we show
the web diagrams and the dual fan triangulation for (1) local $\P
^{1}\times \P ^{1}$, namely the total space of the canonical bundle
over $\P ^{1}\times \P ^{1}$ and, (2) local $\P ^{1}\times B\znums
_{2}$, namely the orbifold quotient of the resolved conifold $\O
(-1)\oplus \O (-1)\to \P ^{1}$ by $\znums _{2}$ acting fiberwise. Note
that the coarse space of local $\P ^{1}\times B\znums _{2}$ has a
transverse $A_{1}$ singularity and its unique crepant resolution is
given by local $\P ^{1}\times \P ^{1}$.
\end{remark}

\begin{remark}
The term web-diagram comes from physics (e.g. \cite{AKMV}). It is
essentially the same as the data determining a tropical plane curve
\cite{Mikhalkin}. The tropical curve associated to $\Gamma _{\X }$ may
be interpreted as the tropicalization of the curve mirror to $\X $
\cite[\S~4]{Gross2000}.
\end{remark}

\begin{remark}
The vertices of $\Gamma _{\X }$ correspond to torus fixed points in
$\X $, the edges correspond to torus invariant curves, and the regions
in the plane delineated by the graph correspond to torus invariant
divisors.
\end{remark}

\begin{figure}
\definecolor{gray}{rgb}{.5,.5,1}  

\setlength{\unitlength}{.7cm}
\begin{picture}(25,12) (3.5,0)
%
%
\color{gray}
\drawline (1,6) (6,1) (6,11) (1,6) (11,6)
\drawline (6,1) (11,6) (6,11)
\put(1 ,6 ){\circle*{.2}}
\put(6 ,6 ){\circle*{.2}}
\put(11 ,6 ){\circle*{.2}}
\put(6 ,11 ){\circle*{.2}}
\put(6 ,1 ){\circle*{.2}}
\put(-.5 ,6.1 ){$\scriptstyle{(1,-1,1)}$}
\put(6.1 ,6.1 ){$\scriptstyle{(1,0,1)}$}
\put(11.1 ,6.1 ){$\scriptstyle{(1,1,1)}$}
\put(6.1 ,11.1 ){$\scriptstyle{(1,0,2)}$}
\put(6.1 ,0.8 ){$\scriptstyle{(1,0,0)}$}
%
\color{black}
\put(4 ,4 ){\line(1 ,0 ){4}}
\put(4 ,4 ){\line(0 ,1 ){4}}
\put(4 ,4 ){\line(-1 ,-1 ){3}}
\put(8 ,8 ){\line(0 ,-1 ){4}}
\put(8 ,8 ){\line(-1 ,0 ){4}}
\put(8 ,8 ){\line(1 ,1 ){3}}
\put(4 ,8 ){\line(-1 ,1 ){3}}
\put(8 ,4 ){\line(1 ,-1 ){3}}
\thicklines
\put(4 ,4 ){\vector (1,0){1}}
\put(4 ,4 ){\vector (0,1){1}}
\put(4 ,4 ){\vector(-1 ,-1 ){1}}
\put(8 ,8 ){\vector(-1 ,0 ){1}}
\put(8 ,8 ){\vector(-1 ,0 ){1}}
\put(8 ,8 ){\vector(0 ,-1 ){1}}
\put(8 ,8 ){\vector(1 ,1 ){1}}
\put(4 ,8 ){\vector(-1 ,1 ){1}}
\put(4 ,8 ){\vector(1 ,0 ){1}}
\put(4 ,8 ){\vector(0 ,-1 ){1}}
\put(8 ,4 ){\vector(1 ,-1 ){1}}
\put(8 ,4 ){\vector(-1 ,0 ){1}}
\put(8 ,4 ){\vector(0 ,1 ){1}}
\thinlines
\put(8.2,4.5 ){$\scriptstyle{(0,1)}$}
\put(6.7,3.5 ){$\scriptstyle{(-1,0)}$}
\put(9.0,3.3 ){$\scriptstyle{(1,-1)}$}
\put(2.9,4.5 ){$\scriptstyle{(0,1)}$}
\put(4.1,3.5 ){$\scriptstyle{(1,0)}$}
\put(1.6,3.3 ){$\scriptstyle{(-1,-1)}$}
\put(2.7,7.3 ){$\scriptstyle{(0,-1)}$}
\put(4.1,8.3 ){$\scriptstyle{(1,0)}$}
\put(1.8,8.5 ){$\scriptstyle{(-1,1)}$}
\put(8.2,7.3 ){$\scriptstyle{(0,-1)}$}
\put(6.7,8.3 ){$\scriptstyle{(-1,0)}$}
\put(9.0,8.5 ){$\scriptstyle{(1,1)}$}
%
%

\put(20,6){\line(1,1){3}}
\put(20,6){\line(1,-1){3}}
\put(16,6){\line(1,0){4}}
\put(16,6){\line(-1,1){3}}
\put(16,6){\line(-1,-1){3}}
\thicklines
\put(20,6){\vector(1,1){1}}
\put(20,6){\vector(1,-1){1}}
\put(20,6){\vector(-1,0){2}}
\put(16,6){\vector(-1,-1){1}}
\put(16,6){\vector(-1,1){1}}
\put(16,6){\vector(1,0){2}}
\thinlines
\put(16.4,6.2){$\scriptstyle{(2,0)}$}
\put(18.6,6.2){$\scriptstyle{(-2,0)}$}
\put(19.8,7.0){$\scriptstyle{(1,1)}$}
\put(15.3,7.0){$\scriptstyle{(-1,1)}$}
\put(19.4,5){$\scriptstyle{(1,-1)}$}
\put(15.4,5){$\scriptstyle{(-1,-1)}$}
%
\color{gray}
\drawline(18,2)(18,10) (14,6) (18,2) (22,6) (18,10)
\put(18,2){\circle*{.2}}
\put(18,10){\circle*{.2}}
\put(14,6){\circle*{.2}}
\put(22,6){\circle*{.2}}
\put(18.2,1.8){$\scriptstyle{(1,0,0)}$}
\put(18.2,10){$\scriptstyle{(1,0,2)}$}
\put(12.6,6.1){$\scriptstyle{(1,-1,1)}$}
\put(22.1,6.1){$\scriptstyle{(1,1,1)}$}
\color{black}
\end{picture}
\caption{The web diagrams and fan triangulations for local $\P ^{1}\times \P ^{1}$ and local $\P ^{1}\times B\znums _{2}$}\label{fig: web diagrams}
\end{figure}
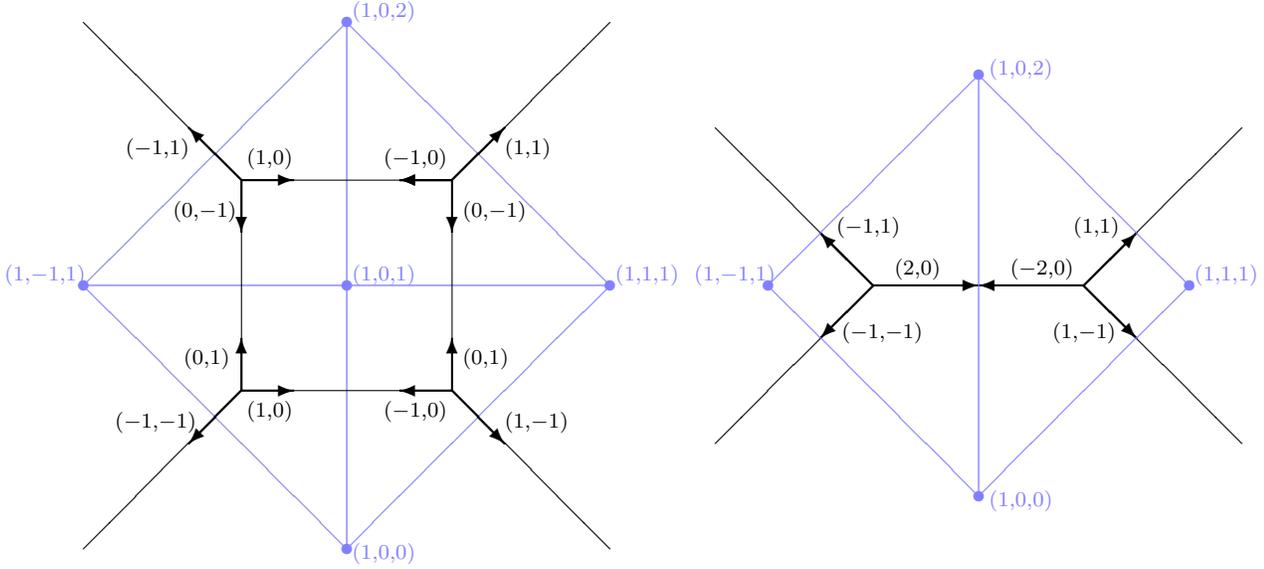

\bigskip

\subsection{Reading off the local model at a point from the web
diagram}

The local model for $\X $ at a torus fixed point is given as follows.
\begin{lemma}\label{lem: getting G from x1,x2,x3}
Let $v$ be the vertex of $\Gamma _{\X }$, let $(e_{1},e_{2},e_{3})$ be
the three edges incident to $v$, and let $x_{v,e_{i}}= (a_{i},b_{i})$
be the markings. Then $\X $ has an open neighborhood about the torus
fixed point corresponding to $v$ given by $[\cnums ^{3}/G]$ where $G$
is the subgroup of the torus $T=\left(\cnums ^{*} \right)^{3}$ given
by
\[
t_{1}t_{2}t_{3}=1, \quad t_{i}^{a_{j}}=t_{j}^{a_{i}}, \quad
t_{i}^{b_{j}}=t_{j}^{b_{i}}.
\]
The action of $G$ on $\cnums ^{3}$ is given by
\[
(z_{1},z_{2},z_{3})\mapsto (t_{1}z_{1},t_{2}z_{2},t_{3}z_{3})
\]
where the $z_{i}$ coordinate axis is the $T$ invariant curve
corresponding to the edge $e_{i}$.
\end{lemma}
\textsc{Proof:} The local model is easily read off from the fan
(e.g. \cite[Eqn.~3]{Borisov-Horja}). The lemma is obtained by simply
translating the fan data into the web diagram.\QED

For
\[
x_{i}= (a_{i},b_{i})
\]
we define
\[
x_{i}\wedge x_{j} = a_{i}b_{j}-a_{j}b_{i}.
\]
We order the edges $(e_{1},e_{2}, e_{3}) $ cyclically in the
counterclockwise direction. Then it follows from the lemma that the
order of $G$ is given by:
\[
|G| = x_{1}\wedge x_{2} =x_{2}\wedge x_{3} =x_{3}\wedge x_{1} .
\]
Moreover, the order of $H_{i}$, the stabilizer group of a generic
point on the $T$ invariant curve corresponding to $e_{i}$ is given by
\[
|H_{i}| = \operatorname{div} (x_{i})
\]
where $\operatorname{div} (x_{i}) = \gcd (a_{i},b_{i})$ is the
divisibility of $x_{i}$.

\subsection{Reading off the local data at a curve from the web
diagram}

Let $e$ be a compact edge in the web diagram and let $C\subset \X$
be the corresponding torus invariant curve. By the
Fantechi-Mann-Nironi classification, $C  $ is given by an Abelian
gerbe over a football. There is a neighborhood of $C $ in $\X $
isomorphic to the total space of the normal bundle of $C $ in $\X
$. The normal bundle is the sum of two line bundles, so to specify the
neighborhood of $C $ we must determine the two normal bundles. In
the case where $C $ is a scheme, a line bundle is determined by its
degree. In general, the line bundles are determined by a slight
generalization of the numerical degree, and we explain below how to
extract this data from the web-diagram.

\begin{defn}\label{defn: degree for a line bundle on a gerbe over a football}
Let $\football $ be the stack obtained from $\P ^{1}$
by root constructions \cite{Cadman} of order $k_{0}$ and $k_{\infty }$
at the points $[0],[\infty ] \in \P ^{1}$ (the so-called
``football'').  Let
\[
\pi :C\to \football
\]
be a $\znums_{h}$ gerbe over the football $\football $ and let $L\to
C$ be a line bundle. We define the \emph{type} of $L$ to be the triple
of integers $(a_{0},a_{\infty },m)$ such that
\[
0\leq a_{0}<k_{0}, \quad 0\leq a_{\infty }<k_{\infty },
\]
and
\[
L^{\otimes h}\cong \pi ^{*}\O _{\football}
(a_{0}[0]+a_{\infty }[\infty ]+m[p])
\]
where $[p]\in \football $ is a generic point. $L$ is
determined up to isomorphism by its type
the \emph{degree} of $L$ to be
\[
\deg (L) = \frac{1}{h}\left(\frac{a_{0}}{k_{0}}+\frac{a_{\infty }}{k_{\infty }}+ m\right).
\]
\end{defn}

The web diagram of $\X $ near the edge $e$ is given by the following
diagram:

\setlength{\unitlength}{.7cm}
\begin{picture}(25,6) (10,3.5)
\thinlines
\put(22,6){\line(1,1){2}}
\put(22,6){\line(1,-1){2}}
\put(16,6){\line(1,0){6}}
\put(16,6){\line(-1,1){2}}
\put(16,6){\line(-1,-1){2}}
\thicklines
\put(22,6){\vector(1,1){1}}
\put(22,6){\vector(1,-1){1}}
\put(22,6){\vector(-1,0){1.5}}
\put(16,6){\vector(-1,-1){1}}
\put(16,6){\vector(-1,1){1}}
\put(16,6){\vector(1,0){1.5}}
\thinlines
\put(16.5,6.3){${x_{1}^{0}}$}
\put(20.6,6.3){${x_{1}^{\infty }}$}
\put(21.8,7.0){${x_{3}^{\infty }}$}
\put(15.3,7.0){$x_{2}^{0}$}
\put(21.4,5){$x_{2}^{\infty }$}
\put(15.4,5){$x_{3}^{0}$}
\put(13.5,6){$D_{0}$}
\put(23.5,6){$D_{\infty }$}
\put(18.5,7){$D$}
\put(18.5,4.5){$D'$}
\end{picture}

Since the divisibility of $x_{1}^{0}$ is $h$, we may use the action of
$SL_{2} (\znums )$ to set $x_{1}^{0}= (h,0)$ and thus $x_{1}^{\infty
}= (-h,0)$. Since the order of the local groups at $0$ and $\infty $
is $k_{0}h$ and $k_{\infty }h$ respectively, we know that $x_{3}^{0}$
and $x_{2}^{\infty }$ have the form
\[
x_{3}^{0} = (\tilde{a}_{0},-k_{0})\quad x_{2}^{\infty } = (-\tilde{a}_{\infty },-k_{\infty })
\]
for some integers $\tilde{a}_{0}$ and $\tilde{a}_{\infty }$. We define
$a_{0}$, $a_{\infty }$, and $m$ such that
\begin{align*}
a_{0}& = \tilde{a}_{0} \,\,  \mod k_{0} \quad \,\,
0\leq a_{0}<k_{0},\\
a_{\infty}& = \tilde{a}_{\infty} \mod k_{\infty} \quad 0\leq
a_{\infty}<k_{\infty},
\end{align*}
and
\[
m= \frac{\tilde{a}_{0}-a_{0}}{k_{0}}+\frac{\tilde{a}_{\infty
}-a_{\infty }}{k_{\infty }}.
\]

\begin{lemma}\label{lem: normal bundle degree from diagram}
The type of $\O _{C} (D)$ is given by $(a_{0},a_{\infty },m)$ and the
numerical degree of $\O _{C} (D)$ is given by
\[
\frac{1}{hk_{0}k_{\infty
}}x_{2}^{\infty }\wedge x_{3}^{0}.
\]
\end{lemma}
\textsc{Proof:} The generators of the one dimensional cones in the fan
of $\X$ corresponding to the divisors $D'$, $D$, $D_{0}$, and
$D_{\infty }$ can be taken to be $(1,0,0)$, $(1,h,0)$,
$(1,-\tilde{a}_{0},k_{0})$ and $(1,-\tilde{a}_{\infty },-k_{\infty })$
respectively (c.f. proof of Lemma~\ref{lem: a toric orbiCY3
corresponds to a web diagram}). Linear functions on the fan give rise
to relations among the divisors
\cite[Theorem~4.10]{Borisov-Horja}. The linear functions corresponding
to the second and third entries of the above vectors give rise to
relations which we restrict to $C$:
\begin{align*}
\O _{C} (hD-\tilde{a}_{0}D_{0}-\tilde{a}_{\infty }D_{\infty }) &\cong  \O _{C},\\
\O _{C} (k_{0}D_{0}-k_{\infty }D_{\infty }) &\cong  \O _{C}.
\end{align*}
Both relations pullback from $\football$ where the
second can be written
\[
\O _{\football } (k_{0}[0])\cong \O _{\football } (k_{\infty }[\infty ]) \cong \O _{\football } ([p]).
\]
Then the first assertion of the lemma, which is equivalent to
\[
\O _{C} (hD) = \pi ^{*}\O _{\football } (a_{0}[0]+a_{\infty }[\infty
]+m[p]),
\]
follows from the definitions and the above relations.

Computing the degree from the above relation, we get
\begin{align*}
\deg (\O _{C} (D)) &= \frac{1}{h}\left(\frac{a_{0}}{k_{0}} +\frac{a_{\infty }}{k_{\infty }} + m \right)\\
&=\frac{1}{h}\left(\frac{\tilde {a}_{0}}{k_{0}}+ \frac{\tilde {a}_{\infty }}{k_{\infty }} \right)\\
&=\frac{1}{hk_{0}k_{\infty }} (k_{\infty }\tilde {a}_{0}+k_{0}\tilde {a}_{\infty })\\
&=\frac{1}{hk_{0}k_{\infty }} (-\tilde {a}_{\infty },-k_{\infty })\wedge (\tilde {a}_{0},-k_{0})\\
&=\frac{1}{hk_{0}k_{\infty }}x_{2}^{\infty }\wedge x_{3}^{0}.
\end{align*}
\qed 

The \CY condition implies
\[
\O _{C} (D+D') \cong \pi ^{*}\O _{\football } (-[0]-[\infty ]).
\]
In terms of the corresponding types $(a_{0},a_{\infty },m)$ and
$(a'_{0},a'_{\infty },m')$, the condition is given by
\begin{align*}
a_{0}+a'_{0} &= -1 \mod k_{0},\\
a_{\infty }+a'_{\infty } &= -1 \mod k_{\infty },
\end{align*}
and
\[
\frac{a_{0}}{k_{0}}+ \frac{a_{\infty }}{k_{\infty }}+m +\frac{a'_{0}}{k_{0}}+ \frac{a'_{\infty }}{k_{\infty }}+m' =-\frac{1}{k_{0}}-\frac{1}{k_{\infty }}.
\]

\bibliographystyle{plain}
\bibliography{mainbiblio}
\end{document}